\documentclass[11pt]{amsart}

\usepackage{amsmath}
\usepackage{amssymb}
\usepackage{amsfonts}
\usepackage{amsthm}
\usepackage{enumerate}
\usepackage[mathscr]{eucal}
\usepackage{graphicx}
\usepackage{verbatim}

\newcommand{\ci}[1]{_{{}_{\!\scriptstyle{#1}}}}
\newcommand{\cin}[1]{_{{}_{\scriptstyle{#1}}}}
\newcommand{\lcs}[1]{\lesssim\cin{#1}}

\def\phqu{{\tfrac{|\om|}\sigma}}

\def\bess{{\mathscr{J}\ci{d_2}\!}}

\def\mut{t}

\def\cal{\mathcal}
\def\frak{\mathfrak}
\def\Bbb{\mathbb}

\def\g{\frak g}
\def\z{\frak z}

\def\tr{{\rm tr}}

\def\lan{\langle}
\def\ran{\rangle}

\def\vpi{{\varpi}}
\def\pa{\partial}

\def\si{{\sigma}}

\def\de{{\delta}}

\def\ka{{\kappa}}

\def\var{{\varphi}}
\def\al{\alpha}

\def\NN{\Bbb N}
\def\ZZ{\Bbb Z}
\def\RR{\Bbb R}

\def\HH{{\Bbb H}}
\def \trans{\,{}^t\!}

\def\bpm{\begin{pmatrix}}
\def\epm{\end{pmatrix}}

\def\bee{\begin{enumerate}}
\def\ee{\end{enumerate}}

\newcommand{\Be}{\begin{equation}}
\newcommand{\Ee}{\end{equation}}

\newcommand{\Bm}{\begin{multline}}
\newcommand{\Em}{\end{multline}}
\newcommand{\Bea}{\begin{eqnarray}}
\newcommand{\Eea}{\end{eqnarray}}
\newcommand{\Beas}{\begin{eqnarray*}}
\newcommand{\Eeas}{\end{eqnarray*}}
\newcommand{\Benu}{\begin{enumerate}}
\newcommand{\Eenu}{\end{enumerate}}
\newcommand{\Bi}{\begin{itemize}}
\newcommand{\Ei}{\end{itemize}}

\def\Re{\operatorname{Re\,}}
\def\Im{\operatorname{Im\,}}

\def\vth{\vartheta}

\def\emph#1{{\it #1 }}

\def\ga{\gamma}

\def\cf{{\it cf}}

\def\sgn{{\text{\rm sign}}}

\def\supp{{\text{\rm supp}}}
\def\rad{{\text{\rm rad}}}

\def\inn#1#2{\langle#1,#2\rangle}
\def\biginn#1#2{\big\langle#1,#2\big\rangle}

\def\noi{\noindent}

\def\meas{{\text{\rm meas}}}

\def\lc{\lesssim}
\def\gc{\gtrsim}

\def\eps{\varepsilon}

\def\ka{\kappa}
            
\def\la{\lambda}             \def\La{\Lambda}
             
\def\si{\sigma}              
\def\vphi{\varphi}
             
\def\om{\omega}              \def\Om{\Omega}
\def\ups{\upsilon}

\def\fA{{\mathfrak {A}}}

\def\fC{{\mathfrak {C}}}

\def\fF{{\mathfrak {F}}}

\def\fI{{\mathfrak {I}}}

\def\fT{{\mathfrak {T}}}

\def\fg{{\mathfrak {g}}}

\def\bbC{{\mathbb {C}}}

\def\bbH{{\mathbb {H}}}

\def\bbN{{\mathbb {N}}}

\def\bbR{{\mathbb {R}}}

\def\bbZ{{\mathbb {Z}}}

\def\cA{{\mathcal {A}}}

\def\cF{{\mathcal {F}}}

\def\cI{{\mathcal {I}}}
\def\cJ{{\mathcal {J}}}
\def\cK{{\mathcal {K}}}
\def\cL{{\mathcal {L}}}

\def\cN{{\mathcal {N}}}

\def\cP{{\mathcal {P}}}

\def\cR{{\mathcal {R}}}
\def\cS{{\mathcal {S}}}
\def\cT{{\mathcal {T}}}

\def\cV{{\mathcal {V}}}
\def\cW{{\mathcal {W}}}

\def\iso{{\text{\rm iso}}}

\theoremstyle{plain}
   \newtheorem{theorem}{Theorem}[section]

   \newtheorem{prop}[theorem]{Proposition}
   
   \newtheorem{lemma}[theorem]{Lemma}

\theoremstyle{remark}

\theoremstyle{remark}
   \newtheorem*{remark*}{Remark}
   \newtheorem*{remarks*}{Remarks}

\theoremstyle{definition}

\begin{document}

\title[The wave equation on groups of
 Heisenberg type]{
Sharp $\text{L}^{\text{p}}$-bounds for
the
wave equation on groups of
 Heisenberg type}

\author{Detlef M\"uller and Andreas Seeger}

\address{
D. M\"uller\\
Mathematisches Seminar\\
  C.A.-Universit\"at Kiel\\
Ludewig-Meyn-Str.4\\ D-24098 Kiel\\ Germany}
\email{mueller@math.uni-kiel.de}

\address{A. Seeger   \\
Department of Mathematics\\ University of Wisconsin-Madison\\Madison, WI 53706, USA}
\email{seeger@math.wisc.edu}


\begin{thanks} {Research partially supported
by NSF grant DMS 1200261.}
\end{thanks}

\date{August 13, 2014}

\begin{abstract} Consider  the wave equation associated with  the
Kohn Laplacian on groups of Heisenberg type. We construct parametrices using
oscillatory integral representations and use them to prove
sharp  $L^p$ and Hardy space regularity results.
\end{abstract}

\maketitle


\section*{Introduction}\label{introduction}

Given a second order differential operator $L$ on a suitable manifold we
consider the Cauchy  problem  for the associated wave equation
\Be\label{cauchy}
\big(\partial_\tau^2+L\big)u=0, \quad u\big|_{\tau=0} =f, \quad \partial_\tau
u\big|_{\tau=0}=g.
\Ee
This paper is a contribution to the problem of $L^p$
bounds of the solutions at fixed time $\tau$,
 in terms of $L^p$ Sobolev norms of the initial data $f$ and $g$.
This problem is well understood if $L$ is the standard Laplacian
$-\Delta$ (i.e. defined as a positive operator) in
$\Bbb R^d$ (Miyachi \cite{miyachi}, Peral \cite{peral}),
or the Laplace-Beltrami operator  on a compact manifold
(\cite{SSS}) of
dimension $d$.
In this case \eqref{cauchy} is a strictly hyperbolic problem and
reduces to estimates for Fourier integral operators associates to a
local canonical graph. The known sharp  regularity results  in this case
say that
if $\gamma(p)=(d-1)|1/p-1/2|$ and the initial data $f$ and $g$
belong to the $L^p$-Sobolev spaces  $L^p_{\gamma(p)}$ and
$L^p_{\gamma(p)-1}$, resp., then the solution $u(\cdot,\tau)$ at fixed
time $\tau$ (say $\tau=\pm 1$)  belongs to $L^p$.

In the absence of strict hyperbolicity, the classical Fourier
integral operator techniques do not seem available anymore and it is
not even clear how to efficiently construct parametrices for the
solutions;
consequently the $L^p$ regularity problem is largely open.
However some considerable progress has been made for the specific case
of an invariant operator on the
Heisenberg group $\bbH_m$ which is often considered as a model case
for more general situations.
Recall that
coordinates on $\Bbb H_m$ are given by $(z,u)$, with $z=x+iy\in
\bbC^m$, $u\in \bbR$, and the group law
is given by  $(z,u)\cdot (z',u')= (z+z', u+u'-\frac 12\Im (z\cdot
\overline{z'})$. A basis of left invariant vector
fields is
given by
$X_j=\frac{\partial}{\partial   x_j}-
\frac{y_j}{2}\frac{\partial}{\partial u}$, $Y_j=\frac{\partial}{\partial   y_j}+
\frac{x_j}{2}\frac{\partial}{\partial u},$ and we consider the Kohn Laplacian
$$L=-\sum_{j=1}^m (X_j^2+Y_j^2).$$
This operator is perhaps the simplest example of a nonelliptic  sum of squares
operator
in the sense of H\"ormander \cite{hoermander-hypo}.
In view of the
Heisenberg group structure
it is natural to analyze the corresponding wave group using tools
 from noncommutative Fourier analysis. The operator $L$
is essentially selfadjoint on $C^\infty_0(G)$ (this follows from the methods used in \cite{ne-st})
 and  the solution of
\eqref{cauchy} can be expressed using the spectral theorem in terms of
functional calculus; it
is  given by
$$u(\cdot,\tau)= \cos(\tau\sqrt L)f +  \frac{   \sin (\tau\sqrt L)}{\sqrt L} g.$$
We are then  aiming to prove
 estimates of the form
\Be\label{SobL}
\|u(\cdot,\tau)\|_p \lc
\|(I+\tau^2L)^{\frac \gamma 2} f\|_p +  \|\tau (I+\tau^2L)^{\frac \gamma 2 -1} g\|_p.
\Ee
involving versions of $L^p$-Sobolev spaces defined by the 
subelliptic operator $L$.
Alternatively, one can consider  equivalent
uniform  $L^p\to L^p$ bounds for operators $a(\tau\sqrt L)
e^{\pm i \tau\sqrt{L}}$
where $a$ is a standard (constant coefficient) symbol of order $-\gamma$.
Note that it suffices to prove those bounds  for times
 $\tau=\pm 1$, after a scaling using the
automorphic dilations $(z,u)\mapsto (rz,r^2 u),\ r > 0$.

A first study about the  solutions to \eqref{cauchy} has been
undertaken
by Nachman \cite{nachman} who showed that
 the wave operator on $\HH_m$ has  a fundamental solution
whose
singularities lie on  the cone $\Gamma$ formed by the
characteristics
through the origin. He showed
that the singularity set  $\Gamma$ has a far more complicated structure
 for $\bbH_m$ than
the corresponding cone in
the Euclidean case.
The fundamental solution is given by a series involving Laguerre polynomials and Nachman  was able to examine  the
asymptotic behavior  as one approaches a
generic
singular point on $\Gamma$. However his  method does not seem to yield
uniform estimates in a neighborhood of the singular set which are crucial for obtaining  $L^p$-Sobolev estimates for solutions to
\eqref{cauchy}.

In \cite{MuSt-wave} the first author and E. Stein were able to derive nearly sharp
$L^1$  estimates (and by interpolation also $L^p$ estimates, leaving open the interesting endpoint bounds).
Their approach
 relied on explicit calculations using Gelfand transforms
for the algebra of radial
 $L^1$ functions on the Heisenberg group, and the geometry
of the singular support remained hidden in this approach.
Later, Greiner, Holcman and Kannai \cite{ghk}
 used
contour integrals and an explicit formula
for the heat kernel
on the Heisenberg group to derive
an  integral
formula for the fundamental solution of the wave equation on $\bbH^m$
which exhibits the singularities of the wave kernel.
We shall follow a somewhat different approach,
which allows us to link the geometrical picture to a
decomposition of the joint spectrum of $L$ and the
operator $U$ of differentiation in the central direction
(see also Strichartz \cite{str}); this linkage is crucial to prove optimal $L^p$ regularity estimates.

In order to derive parametrices we will
use a subordination argument based on stationary phase calculations
to write the wave
operator as an integral involving Schr\"odinger operators for which
explicit formulas are available (\cite{gaveau}, \cite{hulanicki}).
This will yield some type of oscillatory integral representation of the
kernels,  as in the theory of Fourier integral operators which will be amenable to proving $L^p$ estimates.
Unlike in the classical theory of Fourier integral operators (\cite{hoermander-fio}) our
phase functions
are not smooth everywhere and have substantial  singularities; this
 leads to considerable complications.
Finally, an important point in our proof is the identification of a
suitable Hardy
space
for the problem, so that $L^p$ bounds can be proved by interpolation
of $L^2$  and Hardy space estimates. We then obtain the
 following  sharp $L^p$ regularity result which is a direct analogue of the result  by Peral \cite{peral} and Miyachi \cite{miyachi} on the wave equation
in the Euclidean setting.

\medskip

\noi{\bf Theorem.}
{\it  Let $d=2m+1$,
$1<p<\infty$, and
$\ga \ge  (d-1)|1/p - 1/2|$.
Then the  operators
 $ (I+\tau^2 L)^{-\ga/2} \exp(\pm i\tau\sqrt L)$
 extend to  bounded operators on $L^p(\bbH^m)$.
The solutions $u$ to the initial value problem \eqref{cauchy} satisfy the Sobolev type inequalities
\eqref{SobL}.}

\medskip

Throughout the paper we shall in fact consider the  more general situation
of
{\it groups of  Heisenberg type}, introduced by  Kaplan \cite{kaplan}.
These  include groups with  center
of dimension $>1$.
The extension of the above result for the wave operator to  groups of
Heisenberg type   and further results will be formulated in the next section.

\section{The results for  groups of Heisenberg type}

\subsection{\it Groups of Heisenberg type}\label{groupsofheis}
Let $d_1$, $d_2$ be positive integers, with $d_1$ {\it even}, and consider a Lie algebra $\fg$
of Heisenberg type, where
$\g=\g_1\oplus\g_2,$ with $\dim\g_1=d_1$ and  $\dim\g_2=d_2,$
 and
 $$
[\g,\g]\subset \g_2\subset \z(\g)\ ,
$$
$\z(\g)$ being the center of $\g$. Now  $\g$ is
endowed with an inner product
$\lan \ ,\ \ran$ such that $\g_1$ and $\g_2$ or orthogonal subspaces and,
and if we define for  $\mu\in\g_2^*\setminus\{0\}$ the symplectic  form $\om_\mu$ on $\g_1$  by
\begin{equation} \label{omegamu}
\om_\mu(V,W): =\mu\big([V,W]\big)\ ,
\end{equation}
then there is a unique  skew-symmetric linear endomorphism  $J_\mu$ of $\g_1$ such that
\Be\label{repofomegabyJ}
\om_\mu(V,W)=
\lan J_\mu(V),W\ran
\Ee
(here, we also used the natural identification of $\g_2^*$ with $\g_2$ via the inner product). Then on a Lie algebra of Heisenberg type
\begin{equation}\label{Jmusquared}
J_\mu^2=-|\mu|^2 I
\end{equation}
for every 
$\mu\in\g_2^*$.
As the  corresponding connected, simply connected Lie group $G$ 
we then choose the linear manifold $\g,$ endowed with the Baker-Campbell-Hausdorff product
$$(V_1,U_1)\cdot (V_2,U_2):=(V_1+V_2,U_1+U_2+\frac 12 [V_1,V_2]).
$$
As usual, we identify $X\in\g$ with the corresponding left-invariant vector field on $G$ given by the Lie-derivative
$$
X f(g):=\frac d{dt}   f(g\exp(tX))|_{t=0},
$$
where $\exp:\g\to G$ denotes the exponential mapping, which agrees with the identity mapping in our case.

Let us next fix an orthonormal basis $X_1,\dots,X_{d_1}$ of $\g_1,$ as well as an orthonormal basis $U_1,\dots,U_{d_2}$ of $\g_2.$ We may then
 identify $\g=\g_1+\g_2$ and $G$ with $\RR^{d_1}\times\RR^{d_2}$
 by means of the basis $X_1,\dots,X_{d_1}$, $U_1,\dots,U_{d_2}$ of $\g.$
Then our inner product on $\g$ will agree with the canonical
Euclidean product $v\cdot w=\sum_{j=1}^{d_1+d_2}v_jw_j$ on $\RR^{d_1+d_2},$
and $J_\mu$ will be identified with a skew-symmetric $d_1\times d_1$
matrix. We shall also identify the dual spaces of $\g_1$ and $\g_2$ with $\RR^{d_1}$ and $\RR^{d_2},$ respectively, by means of this inner product.
Moreover, the Lebesgue measure $dx\, du$ on $\RR^{d_1+d_2}$ is a
 bi-invariant Haar measure on $G.$
By
\Be\label{dimension}
d:=d_1+d_2
\Ee
we denote the topological dimension of $G.$
The group law on $G$ is  then given by
\begin{equation} \label{grouplaw}
(x,u)\cdot (x',u')= (x+x', u+u' + \frac{1}{2} \inn{\vec Jx}{x'})
\end{equation}
where $\inn{\vec Jx}{x'}$ denotes the vector in $\bbR^{d_2}$ with components
$\inn{J_{U_i}x}{x'}$.

Let
\Be \label {subLaplacian} L:=-\sum_{j=1}^{d_1} X_j^2
\Ee
denote the sub-Laplacian corresponding to the basis $X_1,\dots,X_{d_1}$ of $\g_1.$

\medskip

In the special case $d_2=1$ we may assume that $J_\mu=\mu J, \mu\in \RR,$ where
\begin{equation}\label{cansympl}
J:=\left ( \begin{array}{cc}
  0 & I_{d_1/2}\\
  -I_{d_1/2} & 0
  \end{array}
\right)
\end{equation}
and $ I_{d_1/2}$ is the identity matrix on $\Bbb R^{d_1/2}$.
 In this case $G$  is  the {\it Heisenberg group} $\HH_{d_1/2},$ discussed
 in the introduction.

Finally, some dilation structures and the corresponding metrics
will play an important role in our
proofs; we shall work with both isotropic and nonisotropic dilations.
First, the natural dilations on the Heisenberg type groups are the
automorphic dilations
\begin{equation}\label{automorphic}
 \de_r(x,u):=(rx,r^2u),\quad r>0,
 \end{equation}
 on $G$. We work with the
{\it Koranyi norm}
 $$
 \|(x,u)\|_{\text{Ko}}:=(|x|^4+|4u|^2)^{1/4}
 $$
which is  a homogeneous norm with respect to the dilations $\de_r.$
Moreover, if we denote the corresponding balls by
 \begin{align*}
 Q_r(x,u):=\{(y,v)\in G:\|(y,v)^{-1}\cdot(x,u)\|_{\text{Ko}}<r\},\ (x,u)\in G,\ r>0,
 \end{align*}
 then the volume $|Q_r(x,u)|$ is given by
 $$
 |Q_r(x,u)|=|Q_1(0,0)|\, r^{d_1+2d_2}.
 $$

 Recall that  $d_1+2d_2 = d+d_2$
 is the {\it homogeneous dimension} of $G.$

We will also have to work
with a variant of the \lq
Euclidean'  balls, i.e.  'isotropic balls"  skewed by the
Heisenberg translation, denoted by
$Q_{r,E}(x,u)$.
 \begin{equation}\label{isotropicballs}
 \begin{aligned}Q_{r,E}(x,u)
:&=\{(y,v)\in G:|(y,v)^{-1}(x,u)|_E<r\},\\
&=\big\{ (y,v)\in G:\big| x-y|+| u-v+\frac{1}{2}\inn{\vec
     Jx}{y}|<r\big\};
\end{aligned}
\end{equation}
    here
 $$
  |(x,u)|_E:=|x|+|u|
$$
is  comparable with the standard   Euclidean norm
$(|x|^2+|u|^2)^{1/2}$.   Observe that the
balls $ Q_r(x,u)$ and $Q_{r,E}(x,u)$ are the left-translates by $(x,u)$ of the corresponding balls centered at the origin.

\medskip

\subsection{\it The main results}\label{main}
We consider symbols $a$ of class $S^{-\gamma} $, {\it i.e.}  satisfying the
estimates
\Be \label{symbols}\Big|\frac {d^j}{(ds)^j} a(s)\Big |\le c_j (1+|s|)^{-\gamma-j} \Ee
for all $j=0,1,2,\dots$.
Our main boundedness
result is


\begin{theorem}\label{main-theorem} Let $1<p<\infty$,
$\gamma(p):=(d-1)|1/p-1/2|$
    and
$a\in S^{-\gamma(p)}$. Then for $-\infty<\tau<\infty$,
 the operators
 $a(\tau\sqrt L) e^{ i\tau\sqrt L}$ extend to  bounded operators
on $L^p(G)$.

The solutions $u$ to the initial value problem \eqref{cauchy} satisfy the Sobolev type inequalities
\eqref{SobL}, for $\gamma\ge \gamma(p)$.
\end{theorem}

Our proof also gives sharp $L^1$ estimates
for
operators with  symbols supported  in  dyadic intervals.

\begin{theorem}\label{main-theoremL1} Let $\chi\in C^\infty_c$ supported in $(1/2, 2)$ and let $\lambda\ge 1$.
Then the  operators
 $\chi(\la^{-1}\tau\sqrt L) e^{\pm i\tau\sqrt L}$ extend to  bounded operators
on $L^1(G)$, with operator norms $O(\la^{\frac{d-1}{2}})$.
\end{theorem}

In view of the invariance under automorphic dilations it suffices to
prove these results for  $\tau=\pm 1$, and by symmetry considerations, we only need to consider $\tau=1$.

An interesting question posed in \cite{MuSt-wave} concerns the
validity of an appropriate result in
the limiting case $p=1$ (such as a
Hardy  space bound). Here the situation is more complicated than
in the Euclidean case because of the interplay of isotropic and
nonisotropic dilations. The usual Hardy spaces  $H^1(G)$ are defined using
the nonisotropic
automorphic dilations
\eqref{automorphic}
together with the Koranyi balls. This geo\-metry is
not appropriate for our problem; instead the estimates for our kernels
require a Hardy space
that is
defined using isotropic dilations (just as in the Euclidean case) and
yet is compatible with the Heisenberg group structure.
On the other hand we shall use a dyadic decomposition of the spectrum
of $L$  which corresponds to a Littlewood-Paley decomposition
using nonisotropic dilations.

 This space $h^1_\iso(G)$ is a variant
of the isotropic  local or
(nonhomogeneous) Hardy space in the Euclidean setting.
To define it we first introduce the appropriate notion of atoms.
For  $0<r\le 1$   we define  a $(P,r)$ atom as a function
$b$ supported in the isotropic Heisenberg  ball
$Q_{r,E}(P)$ with radius $r$ centered at $P$
(\cf. \eqref{isotropicballs}), such that
$\|b\|_2\le r^{-d/2},$ and $\int b =0$ if $r\le 1/2$.
A function $f$  belongs to $h^1_\iso(G)$  if
$f=\sum c_\nu b_\nu$ where $b_{\nu}$ is a $(P_\nu, r^\nu)$ atom for
some point $P_\nu$ and some radius $r_\nu\le 1$, and   the sequence
$\{c_\nu\}$ is absolutely convergent. The norm on $h^1_\iso(G)$ is given by
$$\inf \sum_\nu|c_\nu|$$
where the infimum is taken over representations  of $f$ as a sum
$f=\sum_\nu c_\nu b_\nu$ where the $b_\nu$ are  atoms.
It is easy to see that $h^1_\iso(G)$ is a closed subspace of $L^1(G)$.
The spaces 
$L^p(G)$, $1<p<2$,
are  complex interpolation spaces for the couple 
$(h^1_\iso(G), L^2(G))$ (see  \S \ref{interpolation}) and by an analytic  interpolation argument 
Theorem \ref{main-theorem} can be deduced from  an $L^2$ estimate and the
following $h^1_\iso\to L^1$ result.

\begin{theorem} \label{h1thm}
Let $a\in S^{-\frac{d-1}{2}}$. Then the operators
$a(\sqrt L) e^{\pm i \sqrt L}$ map the isotropic Hardy space
$h^1_\iso(G)$
boundedly  to $L^1(G)$.
\end{theorem}
The norm in the  Hardy space
$h^1_\iso(G)$ is not invariant under the automorphic dilations
\eqref{automorphic}.
It is not currently known whether there is
a suitable  Hardy space result which can be used for interpolation and
works
 for all
$a(\tau\sqrt L) e^{ i\tau \sqrt L}$ with bounds uniform in $\tau$.

\medskip

\subsection{\it Spectral multipliers}\label{specmultipliers}
If $m$ is a bounded spectral multiplier, then clearly the operator $m(L)$ is bounded on $L^2(G).$  An important question is then under which additional conditions on the spectral
multiplier $m$ the operator $m(L)$ extends from $L^2\cap L^p(M)$ to an $L^p$-bounded
operator, for a given $p\ne 2$.

Fix a non-trivial cut-off function $\chi\in C^{\infty}_0(\RR)$ supported in
 the interval $[1,2]$; it is convenient  to assume that
$\sum_{k\in \bbZ}\chi(2^ks)=1$ for all  $s>0$.
Let $L^2_{\alpha}(\bbR)$ denote the classical  Sobolev-space of order $\alpha$.
Hulanicki and Stein  (see Theorem 6.25 in \cite{FollSt}),
proved  analogs of the classical Mikhlin-H\"ormander multiplier theorem
on stratified groups, namely
the inequality
\Be\label{locsob}
\|m(L)\|_{L^p\to L^p} \le C_{p,\alpha} \sup_{t>0}\|\chi m(t\cdot)\|_{L^2_\alpha},
\Ee for sufficiently large $\alpha$.
By the work of
M.~Christ \cite{christ}, and also Mauceri-Meda \cite{mauceri-meda},
the inequality \eqref{locsob} holds true for $\alpha> (d+d_2)/2$,  in fact
they established a more general result  for all stratified groups.
Observe that, in comparison to the classical case $G=\RR^d$,
 the homogeneous dimension $d+d_2$ takes over the role of the Euclidean
dimension $d$. However,  for the special case of the Heisenberg groups
it was shown by
E.M.~Stein  and the first author \cite{MuSt-mult} that
\eqref{locsob} holds for the larger range $\alpha>d/2$.
 This result, as well as an extension to Heisenberg type groups has
been proved independently by Hebisch \cite{hebisch}, and Martini \cite{martini}
showed that Hebisch's argument can be used to prove a similar result on
 M\'etivier groups. Here we use our estimate on the wave equation to prove,
only  for Heisenberg type groups, a result that  covers a larger
class of multipliers:

\begin{theorem}\label{multipliers}
Let $G$ be a group of Heisenberg type, with topological dimension $d$.
Let  $m\in L^\infty(\bbR)$,  let  $\chi\in C_0^\infty$ be as above, let
 $$
\fA_{R}:=\sup_{t>0} \int_{|s|\ge R}\big|\cF^{-1}_\bbR[\chi m(t\cdot)](s)\big|s^{\frac{d-1}{2}} ds
$$  and assume
\Be \label{Ftcondition}
\|m\|_\infty +  \int_{2}^\infty  \fA_{R} \frac{dR}{R} <\infty.
\Ee
Then the operator $m(\sqrt L)$
is of weak type $(1,1)$ and
  bounded on  $L^p(G)$,
$1<p<\infty$.
\end{theorem}
\begin{remarks*}
(i)
Let $H^1(G)$ be the standard  Hardy space defined using
the automorphic dilations
\eqref{automorphic}.
Our proof shows that under condition
\eqref{Ftcondition},
$m(\sqrt L)$  maps  $H^1(G)$ to $L^1(G)$.

(ii) Note that by
 an application of the Cauchy-Schwarz inequality and Plan\-cherel's theorem
that the condition
  $$\sup_{t>0}\|\chi m(t\cdot)\|_{L^2_\beta}<\infty\,,
 \text{ for some   $\beta>d/2$}
$$ implies $\fA_R\lc_\gamma R^{\frac d2-\beta}$ for $R\ge 2$  and thus
  Theorem \ref{multipliers}
covers and extends the  above mentioned
multiplier results in \cite{MuSt-mult}, \cite{hebisch}.

(iii) More refined  results for fixed $p>1$ could  be deduced by
interpolation, but  such  results would  likely not be  sharp.
\end{remarks*}

%

\section{Some notation}
\subsection{\it Smooth cutoff functions}\label{cutoffsect}
We denote by  $\zeta_0$  an even $C^\infty$ function
 supported in $(-1,1)$  and assume that $\zeta_0(s)=1$ for $|s|\le
9/16$.
Let $\zeta_1(s)=\zeta_0(s/2)-\zeta_0(s)$ so that $\zeta_1$ is supported
in $(-2,-1/2)\cup (1/2,2)$.
If we set $\zeta_j(s) = \zeta_1(2^{1-j}s)$ then $\zeta_j$ is supported
in $(-2^{j},-2^{j-2} )\cup(2^{j-2},2^j)$
and we have $1=\sum_{j=0}^\infty \zeta_j(s)$ for all $s\in \bbR$.

Let $\eta_0$ be a $C^\infty$ function supported in $(-\frac {5\pi}8 ,
\frac {5\pi}8 )$
which has the property that $\eta_0(s)=1$ for
$|s|\le \frac {3\pi}8$  and
 satisfies
$\sum_{k\in \Bbb Z} \eta_0(t-k\pi)=1$
for all $t\in \bbR$.
For $l=1,2,\dots$ let
$\eta_l(s)=\eta(2^{l-1} s)-\eta_0(2^l s)$ so that $\eta_0(s)=
\sum_{l=1}^\infty  \eta_l(s)$ for $s\neq 0$.


\subsection{\it Inequalities}
We use the notation $A\lc B$ to indicate
$A\le CB$ for some constant $C$. We sometimes use the notation
$A\lc_\ka B$ to emphasize  that the implicit constant depends on the
parameter  $\ka$. We use $A\approx B$ if $A\lc B$ and $B\lc A$.

\subsection{\it Other notation.}
We use the definition
$$\widehat f(\xi)\equiv \cF f(\xi)=\int f(y) e^{-2\pi i\inn{\xi}{y}} dy$$ for the
Fourier transform in Euclidean space $\bbR^d$.

The convolution on $G$ is given by
$$f*g(x,u)= \int f(y,v) g(x-y,
u-v+\tfrac 12 \inn{\vec Jx}{y}) \,dy\, dv.$$

\section{Background on  groups of
Heisenberg type  and the  Schr\"odinger group}

For more on the material reviewed here see, e.g.,  \cite{folland}, \cite{edinburgh} and
\cite{MR-solv}.

\subsection{\it The Fourier transform on a group of Heisenberg type}\label{FTHEIS}

Let us first briefly  recall some  facts about the unitary representation theory of a   Heisenberg type group $G.$ In many contexts, it is useful to establish analogues of the Bargmann-Fock representations of the Heisenberg group for such groups \cite{kaplan-ricci} (compare also \cite{ricci},\cite{damek-ricci}). For our purposes, it will be more convenient to work with Schr\"odinger type representations. It is well-known that these can  be reduced to the case  of the Heisenberg group $\HH_{d_1/2}$ whose product is given by
$(z,t)\cdot (z',t')=(z+z',t+t'+\frac 1 2 \om(z,z')),$
where $\om$ denotes the {\it canonical symplectic form}
$\om(z,w) :=\lan Jz, w\ran$, with $J$ is as in \eqref{cansympl}.
For the convenience of the reader, we shall outline this reduction to the Heisenberg group.

Let us split coordinates $z=(x,y)\in\RR^{d_1/2}\times\RR^{d_1/2}$ in $\RR^{d_1},$
and consider the associated natural basis of left-invariant vector fields
of the Lie algebra of $\HH_{d_1/2}$,
\begin{eqnarray*}
\tilde X_j:=\pa_{ x_j}-\tfrac 1 2 y_j\pa_{ t},\quad \ \tilde Y_j:=\pa_{y_j}+
\tfrac 1 2  x_j \pa_t,
\ \ j=1,\dots,\frac{d_1}{2},\ \mbox{ and }   T:= \pa_t\,.
\end{eqnarray*}

For $\tau\in \RR\setminus\{0\}$, the
{\it Schr\"odinger representation} $\rho_\tau$ of $\HH_{d_1/2}$ acts on the Hilbert space $L^2(\RR^{d_1/2})$
as follows:
\begin{eqnarray*}
[\rho_\tau(x,y,t)h](\xi):=e^{2\pi i\tau(t+y\cdot \xi +\frac 1 2 y\cdot x)}
 h(\xi+x), \quad h\in L^2(\RR^{d_1/2}).
\end{eqnarray*}
This is an irreducible, unitary representation, and every irreducible unitary representation of $\HH_{d_1/2}$ which acts non-trivially on the center is in fact unitarily equivalent to exactly one of these, by the Stone-von Neumann theorem (a good reference for these and related results is for instance \cite{folland}; see also \cite{edinburgh}).

Next, if $\pi$ is any unitary representation, say, of a Heisenberg type group $G,$ we denote by
$$
\pi(f):=\int_G f(g)\pi(g)\, dg,\quad f\in L^1(G),
$$
the associated representation of the group algebra $L^1(G).$ For $f\in L^1(G)$ and $\mu\in\g_2^*=\RR^{d_2},$ it will also be useful to define the partial Fourier transform $f^\mu$ of $f$ along the center by
\Be\label{partialFT}
f^\mu(x)\equiv \cF_2 f(x,\mu):= \int_{\RR^{d_2}} f(x,u) e^{-2\pi i\mu\cdot u}\,du\quad (x\in \RR^{d_1}).
\Ee

 Going
back to the Heisenberg group (where $\g_2^*=\RR$),  if $f\in\cS(\HH_{d_1/2}),$   then it is
well-known and easily seen  that
$$\rho_\tau(f)=\int_{\RR^{{d_1}}}f^{-\tau}( z)\rho_{\tau}(z,0)\, dz$$
 defines
 a trace class operator on $L^2(\RR^{d_1/2}),$ and its trace is given by
\begin{equation}\label{traceformula}
\tr (\rho_\tau(f))=|\tau|^{-d_1/2}\,\int_\RR f(0,0,t)e^{2\pi i \tau t}\, dt=|\tau|^{-d_1/2}\,f^{-\tau}(0,0),
\end{equation}
for every $\tau\in\RR\setminus 0.$
\medskip

From these facts, one  derives the Plancherel formula for our
Heisenberg type group $G.$  Given $\mu\in\g^*_2=\RR^{d_2},\  \mu\ne 0,$
consider the matrix  $J_\mu$
as in \eqref{repofomegabyJ}. By \eqref{Jmusquared} we have
$J_\mu^2=-I$ if  $|\mu|=1,$   and   $J_\mu$ has only
eigenvalues $\pm i$. Since it is orthogonal
there exists an orthonormal basis
$$
X_{\mu,1},\dots,X_{\mu, \frac{d_1}{2}},Y_{\mu,1},\dots,Y_{\mu,\frac{d_1}{2}}
$$
of $\g_1=\RR^{d_1}$ which is symplectic with respect to the form
$\om_\mu,$
i.e., $\om_\mu$ is represented by the standard symplectic matrix $J$
in
\eqref{cansympl} with respect to this basis.

This means that, for every $\mu\in\RR^{d_2}\setminus\{0\},$ there is
an
orthogonal matrix $R_\mu=R_{\frac{\mu}{|\mu|}}\in O(d_1,\RR)$ such that
\begin{equation}\label{JmufromJ}
J_\mu=|\mu|R_\mu J \trans R_\mu.
\end{equation}
Condition \eqref{JmufromJ}
 is in fact equivalent to $G$ being of Heisenberg type.

Now consider the subalgebra $L^1_\rad(G)$ of $L^1(G),$ consisting of all
\lq radial' functions $f(x,u)$ in the sense that they depend only on $|x|$ and
 $u$. As for  Heisenberg groups (\cite{folland},\cite{edinburgh}), this algebra is
 commutative for arbitrary  Heisenberg type groups (\cite{ricci}), i.e.,
\begin{equation}\label{commutativity}
f*g=g*f \quad \mbox {  for every } f,g\in L^1_\rad(G).
\end{equation}
This can indeed  be reduced to the corresponding result on Heisenberg groups by applying the partial Fourier transform in the central variables.

The following lemma is easy  to check and establishes a useful link between representations of  $G$ and those of $\HH_{d_1/2}.$

\begin{lemma}\label{epi}
The mapping $\al_\mu:G\to \HH_{d_1/2},$ given by
$$
\al_\mu(z,u):=(\trans R_\mu z,\tfrac {\mu\cdot u}{|\mu|}),\quad
(z,u)\in
\RR^{d_1}\times\RR^{d_2},
$$
is an epimorphism of Lie groups. In particular, $G/\ker \al_\mu$ is isomorphic to  $\HH_{d_1/2},$ where $\ker \al_\mu=\mu^\perp$ is  the orthogonal complement of $\mu$ in the center $\RR^{d_2}$ of $G.$
\end{lemma}

Given  $\mu\in\RR^{d_2}\setminus\{0\},$ we can now define an  irreducible unitary representation $\pi_\mu$ of $G$ on $L^2(\RR^{d_1})$ by  putting
$$
\pi_\mu:=\rho_{|\mu|}\circ \al_\mu.
$$
Observe that then $\pi_\mu(0,u)=e^{2\pi i \mu\cdot u} I.$ In fact, any irreducible representation of $G$ with central character $ e^{2\pi i \mu\cdot u}$ factors through the kernel of $\al_\mu$ and hence, by the Stone-von Neumann theorem, must be equivalent to $\pi_\mu.$

One then computes that, for $f\in\cS(G),$
$$
\pi_\mu(f)=\int_{\RR^{d_1}}f^{-\mu}(R_\mu z)\rho_{|\mu|}(z,0)\, dz,
$$
so that the trace formula \eqref{traceformula} yields the analogous trace formula
$$\tr\, \pi_\mu(f)=|\mu|^{-\frac{d_1}2}\, f^{-\mu}(0)$$
on $G.$ The Fourier inversion formula in $\RR^{d_2}$ then leads  to
$$f(0,0)=\int_{\mu\in\RR^{d_2}\setminus\{0\}} \tr\, \pi_\mu(f)\, |\mu|^{\frac{d_1}2}\, d\mu.$$
When applied to $\de_{g^{-1}}*f,$ we arrive at the Fourier inversion formula
\begin{equation}\label{Fourierinversion}
f(g)=\int_{\mu\in\RR^{d_2}\setminus\{0\}} \tr\,( \pi_\mu(g)^*\pi_\mu(f))\,|\mu|^{\frac{d_1}2}\, d\mu,\quad g\in G.
\end{equation}
Applying this to $f^* *f$ at $g=0,$ where $f^*(g):=\overline{f(g^{-1})},$ we obtain the Plancherel formula
\begin{equation}\label{plancherel}
\|f\|_2^2=\int_{\mu\in\RR^{n}\setminus\{0\}}\, \|
\pi_\mu(f)\|^2_{HS}\,
|\mu|^{\frac{d_1}2}\, d\mu,
\end{equation}
where $\|T\|_{HS}=(\tr\, (T^*T))^{1/2}$ denotes the Hilbert-Schmidt norm.

\subsection{\it The Sub-Laplacian and the group Fourier transform}\label{SublGrouptr}
Let us next consider the group Fourier transform of our sub-Laplacian $L$ on $G.$

We first observe that $d\al_\mu(X)=\trans R_\mu X$ for every
$X\in\g_1=\RR^{d_1},$ if we view, for the time being, elements of the Lie algebra as tangential vectors at the identity element. Moreover,  by
\eqref{JmufromJ},
we see that
$$\trans R_\mu X_{\mu,1}, \dots, \trans R_\mu X_{\mu,d_1/2},\trans R_\mu Y_{\mu,1},\dots,\trans R_\mu Y_{\mu,d_1/2}
$$
forms a symplectic basis with respect to the canonical symplectic form
$\om$ on $\RR^{d_1}.$  We may thus assume without loss of generality
that this basis agrees with our basis $\tilde X_1,\dots,\tilde
X_{d_1/2}, \tilde Y_1,\dots,
\tilde Y_{d_1/2}$ of $\RR^{d_1},$ so that
$$
d\al_\mu(X_{\mu,j})=\tilde X_j, \ d\al_\mu(Y_{\mu,j})=\tilde Y_j, \quad j=1,\dots, d_1/2.
$$
By our construction of the representation $\pi_\mu,$ we thus obtain for the derived representation $d\pi_\mu$ of $\g$ that
\begin{equation}\label{derivedrep}
d\pi_\mu(X_{\mu,j})=d\rho_{|\mu|}(\tilde X_j), \  d\pi_\mu(Y_{\mu,j})=d\rho_{|\mu|}(\tilde Y_j),  \quad j=1,\dots, d_1/2.
\end{equation}
Let us define the sub-Laplacians $L_\mu$ on $G$ and $\tilde L$ on $\HH_{d_1/2}$ by
$$
L_\mu:=-\sum_{j=1}^{d_1/2} (X_{\mu,j}^2+Y_{\mu,j}^2),\quad
\tilde L:=-\sum_{j=1}^{d_1/2} (\tilde X_j^2+\tilde Y_j^2),
$$
where from now on we consider elements of the Lie algebra again as left-invariant differential operators. Then, by \eqref{derivedrep},
$$
d\pi_\mu(L_\mu)=d\rho_{|\mu|}(\tilde L).
$$
Moreover, since the basis $X_{\mu,1},\dots,X_{\mu,d_1/2},
Y_{\mu,1},\dots,Y_{\mu,d_1/2}$ and our original basis
$X_1,\dots,X_{d_1}$
  of $\g_1$ are both orthonormal bases, it is easy to verify that the distributions $L\de_0$ and $L_\mu\de_0$  agree. Since $Af=f*(A\de_0)$ for every left-invariant differential operator $A,$ we thus have  $L=L_\mu,$ hence
\begin{equation}\label{8.8}
d\pi_\mu(L)=d\rho_{|\mu|}(\tilde L).
\end{equation}
But, it follows immediately from our definition of Schr\"odinger representation  $\rho_\tau$  that $d\rho_\tau(\tilde X_j)=\pa_{\xi_j}$ and
$d\rho_\tau(\tilde Y_j)=2\pi i\tau \xi_j$, so that $d\rho_{|\mu|}(\tilde L)=-\Delta_\xi+(2\pi|\mu|)^2|\xi|^2$ is  a rescaled Hermite operator (\cf.  also \cite{folland}),  and an orthonormal basis of $L^2(\RR^{d_1/2})$ is given by the tensor products
$$
h^{|\mu|}_\al:=h^{|\mu|}_{\al_1}\otimes \cdots\otimes h^{|\mu|}_{\al_{d_1/2}},\quad \al\in\NN^{d_1/2},
$$
where $h_k^\mu(x):=(2\pi|\mu|)^{1/4}h_k((2\pi|\mu|)^{1/2}x),$ and
$$
h_k(x)=c_k(-1)^k e^{x^2/2}\frac{d^k}{dx^k}e^{-x^2}
$$
denotes the $L^2$-normalized Hermite function of order $k$ on $\RR.$ Consequently,
\begin{equation}\label{dpiL}
d\pi_\mu(L)h^{|\mu|}_\al= 2\pi|\mu|(\frac{d_1}{2}+2|\al|) h^{|\mu|}_\al, \quad \al\in
\NN^{d_1/2}.
\end{equation}
It is also easy to see that
\begin{equation}\label{dpiU}
d\pi_\mu(U_j)=2\pi i\mu_j I, \quad j=1,\dots,d_2.
\end{equation}
Now, the operators $L, -iU_1,\dots,-iU_{d_2}$ form
a  set of  pairwise strongly commuting self-adjoint operators,  with joint core
$\cS(G),$
 so that  they admit a joint spectral resolution, and we can thus give meaning to expressions
like  $\vphi(L,-iU_1,\dots,-iU_{d_2})$ for each continuous  function $\vphi$ defined on the corresponding joint spectrum. For simplicity of notation we write
$$U:=(-i U_1,\dots, -i U_{d_2}).$$
If $\vphi$ is bounded, then $\vphi(L,U)$ is a bounded,
left invariant operator on $L^2(G),$ so that it is a convolution operator
$$
\vphi(L,U)f=f*K_\vphi, \quad f\in \cS(G),
$$
with a convolution kernel $K_\vphi\in\cS'(G)$ which  will also be denoted by $\vphi(L,U)\de.$ Moreover, if $\vphi\in \cS(\RR\times\RR^{d_2}),$ then $\vphi(L,U)\de\in\cS(G)$  (\cf. \cite{MRS2}).
 Since functional calculus is compatible with  unitary representation theory,  we obtain in this case  from  \eqref{dpiL}, \eqref{dpiU} that
\begin{equation}\label{pimuphi}
\pi_\mu(\vphi(L,U)\de)h^{|\mu|}_\al= \vphi\Big(2\pi|\mu|(\frac{d_1}{2}+2|\al|),2\pi \mu\Big) h^{|\mu|}_\al
\end{equation}
(this identity in combination with the Fourier inversion formula could in fact be taken as the definition of $\vphi(L,U)\de$).
In particular, the Plancherel theorem implies then that the operator norm on $L^2(G)$ is given by
\begin{equation}\label{opnorm}
\|\vphi(L,U)\|=\sup\{|\vphi(|\mu|(\frac{d_1}{2}+2q), \mu)|: \mu\in\RR^{d_2}, q\in\NN\}.
\end{equation}
Finally, observe that
\begin{equation}\label{Kphimu}
K_\phi^\mu=\vphi(L^\mu,2\pi \mu)\de;
\end{equation}
this follows for instance by applying the unitary representation induced from the character $ e^{2\pi i \mu\cdot u}$ on the center of $G$ to $K_\vphi$.

We shall in fact  only work with functions of $L$ and $|U|,$ defined by
$$
\pi_\mu(\vphi(L,|U|)\de)h^{|\mu|}_\al= \vphi\Big(2\pi|\mu|(\frac{d_1}{2}+2|\al|),2\pi |\mu|\Big) h^{|\mu|}_\al.
$$
Observe that if $\vphi$ depends only on the second variable,  then
$\vphi(|U|)$ is just the  radial convolution operator
 acting only in the central variables, given by
\Be\label{hU}\cF_{\bbR^{d_2}}[\vphi(|U|) f](x,\mu)
= \vphi(2\pi|\mu|)\cF_{\bbR^{d_2}}f(x,\mu)\quad\text{ for all $\mu\in (\bbR^{d_2})^*$}\,.\Ee

\medskip

\subsection{\it Partial Fourier transforms and twisted convolution}\label{PFTtwisedconv}



For $\mu\in\g_2^*,$ one  defines the $\mu$--{\it twisted convolution} of two suitable functions (or distributions)  $\vphi$ and $\psi$ on $\g_1=\RR^{d_1}$ by
$$
(\vphi*_\mu \psi)(x):=\int_{\RR^{d_1}} \vphi(x-y)\psi(y)e^{-i\pi\om_\mu(x,y)}\,dy
$$
where $\omega_\mu$ is as in \eqref{omegamu}.
Then,  with $f^\mu$ as in \eqref{partialFT},
$$
(f*g)^\mu=f^\mu *_\mu g^\mu,
$$
where $f*g$ denotes the convolution product of the two functions $f,g\in L^1(G).$
Accordingly, the vector fields $X_j$ are transformed into the $\mu$-twisted first order differential operators $X_j^\mu$ such that $(X_j f)^\mu=X_j^\mu f^\mu,$ and the
sub-Laplacian is transformed into the $\mu$-twisted Laplacian $L^\mu,$ i.e.,
$$
(Lf)^\mu=L^\mu f^\mu=-\sum_{j=1}^{d_1} (X_j^\mu)^2,
$$
say for $f\in\cS(G)$. A computation shows that explicitly
\begin{equation}\label{Xjmu}
X_j^\mu=\pa_{x_j}+i\pi\om_\mu(\cdot,X_j).
\end{equation}

\subsection{\it The Schr\"odinger group $\{e^{itL^\mu}$\}}\label{Schrgroup}
It will be   important for us  that the
Schr\"odinger operators 
$e^{itL^\mu}, \
t\in \bbR,$ generated by $L^\mu$ can be computed explicitly.

 \begin{prop}\label{eit}

(i)  For $f\in \cS(G),$
\begin{equation}\label{twistedschr}
e^{itL^\mu} f=f*_\mu \ga_t^\mu, \quad t\ge 0,
\end{equation}
where $\ga_t^\mu\in \cS'(\RR^{d_1})$ is a tempered distribution.

(ii) For all $t$ such that  $\sin (2\pi t |\mu|)\ne 0$ the distribution
$\gamma^\mu_t$ is
given by
\Be\label{gammatmu}
\ga_t^\mu(x)=2^{-d_1/2}\Big(\frac {|\mu|}{\sin (2\pi t|\mu|)}\Big)^{d_1/2}\,
e^{-i\frac {\pi}2 |\mu|\cot(2\pi t|\mu|)|x|^2}.
\Ee

(iii) For all $t$ such that $\cos (2\pi t |\mu|)\ne 0$    the Fourier transform of $\gamma^\mu_t$ is given by
 \Be\label{gammatmuF}
\widehat{\ga_t^\mu}(\xi)=\frac 1{(\cos (2\pi t|\mu|))^{d_1/2}}\,
e^{i\frac{2\pi}{|\mu|}\tan(2\pi t|\mu|)|\xi|^2}.
\Ee
\end{prop}

%

Indeed, for $\mu\ne 0,$ let us consider the symplectic vector space $V:=\g_1,$ endowed with the symplectic form $\si:=\frac 1{|\mu|}\om_\mu.$ Notice first that, because of \eqref{Jmusquared},
the volume form $\sigma^{\wedge (d_1/2)},$ i.e., the ${d_1}/{2}$-fold exterior product of $\si$
with itself, can be identified with Lebesgue measure on $\RR^{d_1}.$ As in \cite{dissip}, we then associate to the
pair $(V,\si)$ the Heisenberg group $\HH_V,$  with underlying manifold $V\times\RR$  and endowed with the product
$$
(v,u)(v',u'):=(v+v',u+u'+\frac 12 \si(v,v')).
$$
It  is  then common to denote  for $\tau\in\RR$ the  $\tau$-twisted convolution
by $\times_\tau$ in place of  $*_\tau$ (compare \S5 in \cite{dissip}).
The $\mu$-twisted convolution associated to the group  $G$  will then agree
with the $|\mu|$-twisted convolution $\times_{|\mu|}$ defined on the symplectic vector
space $(V,\si).$
 Moreover, if we
identify  the $X_j\in V$  also with left-invariant vector fields on
$\HH_V,$ then \eqref{Xjmu} shows that
$$X_j^\mu=\pa_{x_j}+i\pi|\mu|\si(\cdot,X_j)$$ agrees with the
corresponding $|\mu|$-twisted differential operators $\tilde
X_j^{|\mu|}$ defined in \cite{dissip}.

Accordingly, our $\mu$-twisted Laplacian $L^\mu$ will agree with the $|\mu|$-twisted Laplacian
$$\tilde L_S^{|\mu|}=\tilde{\cal L}_{-I}^\mu=\sum_{j=1}^{d_1} (\tilde X_j^{|\mu|})^2
$$
associated to the symmetric matrix $A:=-I$ in \cite{dissip}. Here,
$$S=-A\frac 1{|\mu|}J_\mu=\frac 1{|\mu|}J_\mu.
$$
Consequently,
$$ e^{itL^\mu}=e^{it\tilde L_S^{|\mu|}}.
$$
From Theorem 5.5 in \cite{dissip} we therefore obtain that
for $f\in L^2(V)$
$$
\exp (\frac {it}{|\mu|} \tilde L_S^{|\mu|})f=f \times_{|\mu|}\Gamma_{t,iS}^{|\mu|}, \quad
t\ge0,
$$
where, $\Gamma_{t,iS}^{|\mu|}$ is a tempered distribution  whose Fourier
transform is given by
$$
\widehat{\Gamma_{t,iS}^{|\mu|}}(\xi)=\frac 1{\sqrt{\det(\cos 2\pi itS)}}
e^{-\frac{2\pi}{|\mu|}\sigma(\xi,\tan(2\pi itS)\xi)}
$$
whenever $\det(\cos (2\pi it S)\ne 0$. Since $S^2=-I$ because of \eqref{Jmusquared}, one sees that
$$\sin(2\pi itS)=i\sin(2\pi t) S,\ \cos(2\pi itS)=\cos(2\pi t)I.
$$
Note also that $\si(\xi,\eta)=\lan S\xi,\eta\ran.$ We thus see that \eqref{twistedschr} and  \eqref{gammatmuF} hold true,
 and  the formula \eqref{gammatmu} follows by Fourier inversion (\cf. Lemma 1.1 in \cite{MR1}).

\section{An approximate subordination formula}\label{subordination}

We shall use Proposition \ref{eit} and the following subordination formula to
obtain manageable expressions for the wave operators.

\begin{prop} \label{subordop}Let $\chi_1\in C^\infty$ so that $\chi_1(s)=1$ for $s\in [1/4,4]$. Let  $g$ be a $C^\infty$ function  supported in $(1/2,2)$.
Then there are $C^\infty$ functions  $a_\la$ and $\rho_\la$,
depending linearly on $g$,  with $a_\la$
  supported in $[1/16,4]$,  and $\rho_\la$ be supported in $[1/4,4]$,
so that for all $K=2,3,\dots$, $N_1, N_2\ge 0$,  and all $\la\ge 1$
\Be\label{alaest}
\sup_s\big|\partial_s^{N_1} \partial_\la^{N_2}  a_\la(s)\big|\le c(K)
\la^{-N_2} \sum_{\nu=0}^K\|g^{(\nu)}\|_\infty,
\quad \,\,N_1+N_2< \frac{K-1}2,
\Ee
\Be\label{rholaest}
\sup_s\big|\partial_s^{N_1} \partial_\la^{N_2}  \rho_\la(s)\big|\le c(K,N_2)
\la^{N_1+1-K}\sum_{\nu=0}^K\|g^{(\nu)}\|_\infty,
 \quad N_1\le K-2\,.
\Ee
 and the formula
\Be
g(\la^{-1}\sqrt L)e^{i\sqrt L}=\chi_1(\la^{-2} )L)
\sqrt{\la }\int e^{i\frac{\la }{4s}}  a_\la(s) \,
e^{i s L/\la}\,
ds\,+\, \rho_\la(\la^{-2}L)
\label{localmultop}
\Ee
holds.
For any $N\in \bbN$, the functions $\la^{N} \rho_\la$ are uniformly bounded in the topology of the Schwartz-space, and the operators  $\rho_\la(\la^{-2}L)$
are  bounded on $L^p (G)$, $1\le p\le \infty$,
 with operator norm $O(\la^{-N})$.
\end{prop}

For explicit formulas of $a_\la$ and $\rho_\la$ see
Lemma \ref{subordlemma} below.
The proposition is essentially an application of the method  of stationary phase
where we keep track on how  $a_\la$, $\rho_\la$ depend on $g$.
We shall need an auxiliary lemma.

\begin{lemma} \label{stationaryphase}
Let $K\in \Bbb N$ and  $g\in C^K(\Bbb R)$. Let $\zeta_1\in
C^\infty(\bbR)$ be
supported in $(1/2,2)\cup(-2,-1/2)$ and $\Lambda\ge 1$ and $\ell\ge 0$.
Then, for all nonnegative integers $M$,
\begin{multline}
\label{fixedellest}
\Big| \int y^{2M} g(y) \zeta_1(\La^{1/2}2^{-\ell} y) e^{i\La y^2}
dy\Big|\ \\
\le C_{M,K} 2^{-2\ell K} \big( 2^{\ell} \La^{-1/2}\big)^{1+2M} \sum_{j=0}^K
 (2^\ell \La^{-1/2})^j \|g^{(j)}\|_\infty\,.
\end{multline}
Moreover, for  $0\le m<\frac{K-1}2$,
\Be\label{sumell}
\Big|\Big(\frac{d}{d\Lambda}\Big)^m \int  g(y) e^{i\La y^2}
dy\Big|
\le C_{K} \La^{-m-\frac 12}  \sum_{j=0}^K
 \La^{-j/2} \|g^{(j)}\|_\infty.
\Ee
\end{lemma}
\begin{proof}
 By
induction on $K$ we  prove the following assertion labeled

\medskip

\noindent $(\cA_K)$: \
If $g\in C^K$ then
\begin{multline} \label{indclaim}
\int y^{2M} g(y) \zeta_1(\La^{1/2}2^{-\ell} y) e^{i\La y^2}
dy \\=
\La^{-K}\sum_{j=0}^K \int g^{(j)}(y) \zeta_{j,K,M,\La}(y)
e^{i\La y^2} dy
\end{multline}
where $\zeta_{j,K,M,\La}$
is supported  on $\{y: |y|\in
[2^{\ell-1}\La^{-1/2}, 2^{\ell+1}\La^{-1/2}]\}$ and,
for $0\le j\le K$,  satisfies the
differential inequalities
\Be\label{diffineq}
\big|\zeta_{j,K,M,\La}^{(n)}(y)\big|
\le C(j,K,M,n) (2^{-\ell} \La^{1/2})^{n-2M} 2^{-\ell(2K-j)} \La^{K-j/2}\,.
\Ee
Clearly this assertion implies
\eqref{fixedellest}.

We set
$ \zeta_{0,0,M,\La}(y) =y^{2M}
\zeta_1(\La^{1/2}2^{-\ell} y)$ and the claim
$(\cA_K)$ is immediate  for $K=0$.
It remains to show that  the implication
$(\cA_K)\implies (\cA_{K+1})$, holds for all  $K\ge 0$.

Assume 
$(\cA_K)$ for some $K\ge 0$  and let  $g\in C^{K+1}$.
We let $0\le j\le K$ and examine the $j$th term in the sum in
\eqref{indclaim}.
Integration by parts yields
\begin{multline*}
\int g^{(j)}(y) \zeta_{j,K,M,\La}(y)
e^{i\La y^2} dy\\
=
i \int \Big[ \frac{  g^{(j+1)}(y)}{2y\La} \zeta_{j,K,M,\La}(y)
+  g^{(j)}(y) \frac{d}{dy}\Big(\frac{\zeta_{j,K,M,\La}(y) }{2y\La}\Big)
\,\Big] e^{i\La y^2} dy\,.
\end{multline*}
The sum $\La^{-K}\sum_{j=0}^K \int g^{(j)}(y) \zeta_{j,K,M,\La}(y)
e^{i\La y^2} dy$ can now be rewritten as
\[
\La^{-K-1}\sum_{\nu=0}^{K+1} \int g^{(\nu)}(y) \zeta_{\nu,K+1,M,\La}(y)
e^{i\La y^2} dy
\]
where
\begin{align*}
\zeta_{0,K+1,M,\La}(y) &= i
\frac{d}{dy}\Big(\frac{\zeta_{0,K,M,\La}(y) }{2y}\Big)\,,
\\
\zeta_{\nu,K+1,M,\La}(y) &= i
\frac{d}{dy}\Big(\frac{\zeta_{\nu,K,M,\La}(y) }{2y}\Big)
+i \frac {\zeta_{\nu-1,K,M,\La}(y) }{2y}, \quad 1\le \nu \le K,
\\
\zeta_{K+1,K+1,M,\La}(y) &=i
\frac {\zeta_{K,K,M,\La}(y) }{2y}\,.
\end{align*}
On  the support of the cutoff functions we have $|y|\ge 2^{\ell-1}\La^{-1/2}$ and the asserted  differential inequalities for
the
functions $\zeta_{\nu,K+1,M,\La}$ can be  verified using the Leibniz rule.
This finishes the proof of
the implication  $(\cA_{K})\implies
(\cA_{K+1})$ and thus the proof of
\eqref{fixedellest}.


We now prove \eqref{sumell}. Let $\zeta_0$ be an even $C^\infty$ function
supported in $(-1,1)$  and assume that $\zeta_0(s)=1$ for $|s|\le
1/2$.
Let $\zeta_1(s)=\zeta_0(s/2)-\zeta_0(s)$ so that $\zeta_1$ is supported
in $[-2,-1/2]\cup[1/2,2]$, as in the   statement of \eqref{fixedellest}.
We split the left hand side of
\eqref{sumell} as $\sum_{\ell=0}^\infty I_{\ell,m}$ where
$$
I_{\ell,m}=
 \int (iy^{2})^m  g(y) \zeta_1(\La^{1/2}2^{-\ell} y) e^{i\La y^2}
dy, \qquad \text{for } \ell>0\,
$$
and $I_{0,m}$ is defined similarly with
$\zeta_0(\La^{1/2} y)$ in place of $\zeta_1(\La^{1/2}2^{-\ell} y)$.
Clearly $|I_{0,m}|\lc \La^{-m-1/2}\|g\|_\infty$ and by
\eqref{fixedellest}
$$I_{\ell,m}\lc_{m,K} \sum_{j=0}^K
 2^{-\ell (2K-2m-j-1)}  \La^{-\frac{1+2m+j}2} \|g^{(j)}\|_\infty.$$
Since $j\le K$ we can sum in $\ell$ if $m<\frac{K-1}2$ and the
assertion \eqref{sumell} follows.
\end{proof}

\begin{lemma}\label{subordlemma}
Let $K\in \Bbb N$ and let $g\in C^K(\Bbb R)$ be supported in $(1/2,2)$, and let $\chi_1\in C^\infty_c(\bbR)$ so that $\chi_1(x)=1$ on $(1/4,4)$.
Also let $\varsigma$ be a $C^\infty_0(\bbR)$ function
  supported in $[1/9, 3]$ with the property that $\varsigma(s)=1$ on
  $[1/8,2]$.
Then
\Be
g(\sqrt x)e^{i\la\sqrt x}= \chi_1(x)
\Big[
\sqrt{\la }\int e^{i\frac{\la }{4s}}  a_\la(s) \,
e^{i\la s x}\,
ds\,+\,\widetilde  \rho_\la(x)\Big]
\label{localmult}
\Ee
where $a_\la$ is  supported in $[\frac 1{16},4]$,
and
\Be \label{alambda}
a_\lambda(s) =\pi^{-1}\sqrt{\lambda} \varsigma(s)\int (y+\tfrac 1{2s}) g(y+\tfrac 1{2s})
e^{-i\la sy^2} \,dy
\Ee
and
\Be \label{rholambda} \cF[\widetilde{\rho}_\la](\xi )=
(1- \varsigma(\tfrac{2\pi\xi}{\la})) \cF[g(\sqrt{\cdot}) e^{i\la \sqrt{\cdot}}] (\xi)\,.
\Ee
Let $\rho_\la=\chi_1\widetilde \rho_\la$.
Then the   estimates \eqref{alaest} and \eqref{rholaest}
hold for all  $\la\ge 1$.
\end{lemma}

\begin{proof}
Let $\Psi_\la$ be the Fourier transform of  $x\mapsto g(\sqrt x)
e^{i\la \sqrt{x}} $, i.e.
\Be \Psi_\la(\xi)=\int g(\sqrt x)e^{i\la  \sqrt{x}}
e^{-2\pi i\xi x}\, dx\,
= \int 2sg(s) e^{i(\la  s-2\pi \xi s^2)}\, ds
\label{Psifirst}
\Ee

Observe that $g(\sqrt x)
=0$ for $x\notin (1/4,4) $, thus  $g(\sqrt x)=\chi_1(x)g(\sqrt x)$.
By the Fourier inversion formula we have
$$g(\sqrt x)e^{i\la   \sqrt{x}} =\chi_1(x)\big(\ups_\la(x)+\rho_\la(x)\big)$$ where
\Be\label{vlarhola}
\begin{aligned}
\ups_\la(x)
&=\int \varsigma(\tfrac{2\pi \xi}{\la})\Psi_\la(\xi)e^{2\pi ix\xi} d\xi
\\
\widetilde \rho_\la(x)
&=\int \big(1-\varsigma(\tfrac{2\pi\xi}{\la})\big)
\Psi_\la(\xi)e^{2\pi ix\xi} d\xi
\end{aligned}
\Ee
so that $\widetilde \rho_\la$ is as in \eqref{rholambda}.

We first consider $\widetilde \rho_\la$ and verify that
the inequalities \eqref{rholaest} hold.
On the support of $1-\varsigma(2\pi\xi/\la)$ we have either $|2\pi\xi|\le
\la/8$
 or $|2\pi\xi|\ge 2\la$.
Clearly, on the support of $g$ we have
$|\partial_s (\la  s-2\pi \xi s^2)|\ge \la/2$ if $|2\pi\xi|\le
\la/8$ and
$|\partial_s (\la  s-2\pi\xi s^2)|\ge |2\pi\xi|/2$ if  $|2\pi\xi|\ge 2\la$.
Integration by  parts in \eqref{Psifirst}  yields
$$\big|\partial^{M_1}_\xi \partial^{M_2}_\la
\big[(1-\varsigma(2\pi \xi/\la))\Psi_\la(\xi)\big]\big| \le C_{M_1,M_2,K}
\|g\|_{C_K}(1+|\xi|+|\la|)^{-K} .
$$ Thus, if $N_1\le K-2$,
\begin{align*}
\Big|\Big(\frac{d}{dx}\Big)^{N_1}&\widetilde \rho_\la(x)\Big|= \Big|\int (2\pi\xi)^{N_1}
(1-\varsigma(2\pi \xi/\la))\Psi_\la(\xi) e^{2\pi ix\xi}d\xi\Big|
\\&\le C_{N_1,K} \|g\|_{C^K}\,\int\frac{(1+|\xi|)^{N_1}}{(1+|\xi|+|\la|)^{K}}
d\xi
\le C_{N_1,K}' \|g\|_{C^K} \la^{-K+N_1+1}.
\end{align*}
This yields \eqref{rholaest} for $N_2=0$, and the same argument
applies to the $\la$-derivatives.

It remains to represent the function $\la^{-1/2} \ups_\la$ as the
integral in \eqref{localmult}.
Let
\Be\label{wtg}\widetilde g(s) = 2sg(s)\,.\Ee
By a change of variable we may
write
\Be\label{Psisecond}
\Psi_\la(\xi)= e^{\frac{i\la^2}{8\pi\xi}}
 \int \widetilde g
  (y+\tfrac{\la}{4\pi\xi}) e^{-2\pi i\xi y^2} dy.
\Ee
We compute
from \eqref{vlarhola}, \eqref{Psisecond},
$$
\ups_\la(x)
= \la\int\varsigma(s) e^{i\frac{\la}{4s}+i \la sx} \la^{-1/2}a_\la(s) ds
$$
where
$$ a_\la(s) =(2\pi)^{-1}\sqrt \la \, \varsigma(s)\int \widetilde g
  (y+\tfrac{1}{2s}) e^{-i\la s y^2} dy\,,
$$
i.e. $a_\la$ is as in \eqref{alambda}.
In order to show the estimate  \eqref{alaest} observe
$$2\pi\partial_\la^{N_2} (\la^{-1/2} a_\la(s))=
\, \varsigma(s)\int \widetilde g
  (y+\tfrac{1}{2s})
(-isy^2)^{N_2} e^{-i\la s y^2} dy
$$ and then by the Leibniz rule
$\partial_s^{N_1} \partial_\la^{N_2}  [\la^{-1/2}a_\la(s)]
$ is a linear combination of  terms of
the form
\Be\label{typicalterm}
\Big(\frac{d}{ds}\Big)^{N_3}\big[\varsigma(s) s^{N_2}]
 \int y^{2N_2} (\la y^2)^{N_5}
\Big(\frac{d}{ds}\Big)^{N_4}
\big[ \widetilde g
  (y+\tfrac{1}{2s}) \big]
e^{i\la sy^2}
\, dy
\Ee
where and $N_3+N_4+N_5=N_1$.
By estimate \eqref{sumell} in Lemma \ref{stationaryphase}
we see that the  term
\eqref{typicalterm} is bounded (uniformly in
$s\in [1/9,3]$)  by a constant times
$$\la^{-N_2-\frac 12}  \big\|(\tfrac{d}{ds})^{N_4}[ \widetilde g
  (\cdot+\tfrac{1}{2s}) \big] \big\|_{C^{K-N_4}}
$$
provided that $N_2+N_5 < (K-N_4-1)/2$.
This  condition is satisfied
if $N_1+N_2<(K-1)/2$ and under this condition we get
$$
\sup_s|\partial_s^{N_1} \partial_\la^{N_2}  [\la^{-1/2}a_\la(s)]|
\lc \la^{-N_2-\frac 12} \|g\|_{C^K}.
$$
Now \eqref{alaest} is a straightforward consequence.
\end{proof}

\begin{proof}[Proof of Proposition \ref{subordop}]
The identity
\eqref{localmultop} is an immediate consequence of the
  spectral resolution $L=\int_{\bbR^+} xdE_x,$
Lemma \ref{subordlemma} (applied with $x/\la$ in place of $x$) and Fubini's theorem. Note that in
  view of the  symbol estimates \eqref{rholaest} any Schwartz norm of
  $\rho_\la(\la^{-2}\,\cdot)$ is $O(\la^{-N})$ for every  $N\in\NN.$
The statement
  on the operator norms of $\rho_\la(\la^{-2} L)$  follows then from  the known multiplier theorems
(such as the  original one by Hulanicki and Stein, see \cite{hulanicki}, \cite{FollSt}).
\end{proof}

Thus in order to get manageable formulas for our wave operators it
will be important to get
explicit formulas for the
Schr\"odinger group
$e^{isL}$, $ s\in \bbR$.

\section{Basic decompositions of the wave operator and  statements of refined results}
\label{dyadicdecsect}
We consider operators $a(\sqrt L)e^{i\sqrt L}$  where $a\in S^{(d-1)/2}$
(satisfying \eqref{symbols} with
$\gamma=\frac{d-1}{2}$).
We split off the part of the symbol supported near $0$.
Let $ \chi_0\in C^\infty_c(\bbR)$ be  supported in $[-1,1]$; then we observe
that the operator $\chi_0(\sqrt L) \exp (i\sqrt L)$
extends to a bounded operator on
 $L^p(G)$, for $1\le p\le \infty$.
To see this we decompose
$\chi_0(\sqrt{\tau})e^{i\sqrt\tau}
=\chi_0(\sqrt \tau)+\sum_{n=0}^\infty \alpha_n(\tau)$,
$\tau>0$,
where $$\alpha_n(\tau)=\chi_0(\sqrt \tau) (e^{i\sqrt \tau}-1)
( \zeta_0(2^{n-1}\tau)-\zeta_0(2^n\tau))$$
where $\zeta_0$ is as in \S\ref{cutoffsect}.
Clearly $\chi_0(\sqrt \cdot)\in C^\infty_0$.
Thus by
 Hulanicki's theorem \cite{hulanicki}
the convolution kernel of
$\chi_0(\sqrt L)$
is a Schwartz function and hence $\chi_0(\sqrt L)$ is bounded on $L^1(G)$.
Moreover the functions
$2^{n/2}\alpha_n(2^{-n}\cdot)$ belong to  a bounded set of Schwartz functions
supported in $[-2,2]$. By dilation invariance and  again
Hulanicki's theorem the convolution kernels of
$2^{n/2}\alpha_n(2^{-n}L)$  are Schwartz functions
and form a bounded subset of the Schwartz space $\cS(G)$.  Thus, by rescaling,
the operator $\alpha_n(L)$ is bounded on  $L^1(G)$ with operator norm
$O(2^{-n/2})$. We may  sum in $n$ and obtain the desired bounds for
$\chi_0(\sqrt{\tau})e^{i\sqrt\tau}$.

The above also implies that for any $\la$  the operator
$\chi(\la^{-1}\sqrt L) \exp (i\sqrt L)$  is bounded on $L^1$ (with a
polynomial and nonoptimal growth in  $\la$).
Thus, in what follows it suffices to consider  symbols $a\in S^{-(d-1)/2}$
with the property that $a(s)=0$ in a neighborhood of $0$. Then
\Be\label{dyadicdec}a(\sqrt L)e^{i\sqrt L}\,=\, \sum_{j>C} 2^{-j\frac{d-1}{2}}
g_j (\sqrt{2^{-2j} L}) e^{i\sqrt L},
\Ee
where the $g_j$ form a family of smooth functions supported in $(1/2,2)$ and bounded in the $C^\infty_0$ topology.  In many   calculations below when $j$
is fixed we shall also  use the parameter $\la$ for $2^j$.

Let $\chi_1$ be a smooth function such that
\begin{subequations}\label{chi1}
\begin{align} \label{chi1support}
&\supp(\chi_1) \subset (2^{-10}, 2^{10})\,,
\\
\label{chi1=1}
&\chi_1(s)=1 \text{ for $s\in (2^{-9}, 2^9)$.}
\end{align}
\end{subequations}
By Proposition \ref{subordop} and
Lemma \ref{subordlemma} we may thus write
\Be\label{adydecomposition}a(\sqrt L)e^{i\sqrt L}\,=\,
m_{\text{negl}}(L)+\sum_{j>100} 2^{-j\frac{d-1}{2}}
\chi_1(2^{-2j}L) m_{2^j}(L),
\Ee
where the ``negligible'' operator $m_{\text{negl}}(L)$ is a convolution with a Schwartz kernel,
\Be \label{mj} m_{\la}(\rho)= \sqrt{\la} \int e^{i\la/(4\tau )}
a_\la(\tau)
e^{i\tau \rho/\la}
d\tau, \quad \text{ with } \la =2^j,
\Ee
and the $a_\la$
form a family of smooth functions supported in $(1/16,4)$,  bounded in the $C^\infty_0$ topology.

We shall use  the formulas \eqref{gammatmu}, which give explicit expressions for the  partial Fourier transform in the central variables  of  the Schwartz kernel of
$e^{it L}.$  In undoing this partial Fourier transform, it
 will  be useful to
recall from \S3 that if  $\rho_1$ denotes  the spectral parameter for $L$
then the joint spectrum of the operators $L$ and $|U|$ is  contained
in the closure of \Be\label{jtspec}\{(\rho_1,\rho_2):\rho_2\ge 0, \, \rho_1
=(\tfrac{d_1}{2}+2q)\rho_2 \text{ for some nonnegative
integer $q$}\}\,.\Ee


As the phase in \eqref{gammatmu} exhibits periodic singularities
it  natural to introduce an equally spaced decomposition in the central
Fourier
variable (i.e., in the spectrum of the operator $|U|$).
Let $\eta_0$ be a $C^\infty$ function such that
\begin{subequations}
\begin{align}
&\supp (\eta_0)
\subset (-\tfrac {5\pi}8 ,
\tfrac {5\pi}8 )\,,
\label{suppeta0}
\\
\label{eta0equalto1}
&\eta_0(s)=1 \,\text{ for } s\in
(-\tfrac {3\pi}8, \tfrac{3\pi}{ 8})\,,
\\
\label{etapartitionof1}
&\sum_{k\in \Bbb Z} \eta_0(t-k\pi)=1 \text{  for  $t\in \bbR$.}
\end{align}
\end{subequations}
We decompose
\Be\label{kdecomposition}
\chi_1(\la^{-2} L)m_\la(L)= \sum_{k=0}^{\infty} \chi_1(\la^{-2} L)
T^k_\la,
\Ee
where
\Be\label {Tkla}
T^k_\la=
\la^{1/2}\int e^{i\la/(4\tau)}
 a_\la(\tau) \eta_0 (\tfrac {\tau}{ \la}|U|- k\pi)
 e^{i\tau L/\la}
d\tau\,.
\Ee
The  description \eqref{jtspec} of the  joint spectrum of $L$ and $|U|$
 gives a restriction on the
summation in $k$.
Namely the operator $\eta_0 (\tfrac {\tau}{ \la}|U|- k\pi)
\chi_1(\la^{-2} L)$  is identically zero unless there exist positive
$\rho_1$ and $\rho_2$  with $\rho_1\ge \rho_2 d_1/2$ such that
$\frac{\la^2} 5<\rho_1<5\la^2$ and
$(k\pi- \tfrac{5\pi}{8})\frac{\la}{\tau}
<\rho_2 <
(k\pi+ \frac{5\pi}{8})\frac{\la}{\tau}$  for some
$\tau \in (\tfrac 1{16},4)$.
A necessary condition for these two conditions to hold simultaneously is
of course  $\frac{d_1}{2} (k\pi-\frac 58 \pi)\frac{\la}{ 4} \le 5\la^2$ and since $d_1\ge 2$ and $\la\ge 1$ we see that the sum in \eqref{kdecomposition} extends only over $k$ with
\Be\label{kleeightla}0\le k<8\la.\Ee

We now
derive formulas for the
convolution kernels of  $T^{k}_\la$, which we denote
by
 $K^{k}_\la$.
The identity \eqref{gammatmu} first gives formulas for the partial Fourier transforms
$\cF_{\bbR^{d_2}}K^{k}_\la$. Applying the Fourier
inversion formula we get
\begin{multline}\label{Kklaexpression}
K^{k}_\la(x,u) \,=\, \la^{1/2} \int_{\bbR^{d_2}}
\int_\bbR e^{i\frac{\la}{4\tau}}
 a_\la(\tau) \eta_0 (2\pi|\mu|\tfrac {\tau}{ \la}- k\pi)\,\times
\\
\Big(\frac{|\mu|}{2\sin (2\pi|\mu|\tau/\la)}\Big)^{d_1/2}
e^{-i|x|^2\frac {\pi}{ 2} |\mu|\cot (2\pi  |\mu|\tau/\la)} d\tau
\, e^{2\pi i \inn{u}{ \mu}} d\mu\,.
\end{multline}

We note that the term
 $|\mu|\cot(2\pi t|\mu|)$ in \eqref{Kklaexpression}
is singular for $2t|\mu|\in \bbZ\setminus \{0\}$  and therefore we shall   treat the operator $T_\la^0$ separately from $T^k_\la$ for $k>0$.
We shall see that $T^0_\la$, and the operators $\sum_j \chi(2^{-2j}L) T^0_{2^j}$  can be handled using known results about Fourier integral operator, while the operators $T^k_{2^j}$ need a more careful treatment due to the singularities of the phase function.
We shall see that the decomposition into the operators $T^k_{2^j}$ encodes useful information on the singularities of the wave kernels.

In \S\ref{Tlaksection}, \S\ref{L1section}
we shall prove the following  $L^1$ estimates

\begin{theorem} \label{TjL1thm}
(i) For $\la\ge 2^{10}$
\Be\label{Tj0L1} \|T^0_\la\|_{L^1\to L^1} \lc \la^{(d-1)/2}\,.\Ee
(ii) For $\la \ge 2^{10}$, $k=1,2,\dots$,
\Be \label{TlakL1}
 \|T^k_\la\|_{L^1\to L^1} \lc k^{-\frac{d_1+1}{2}}
\la^{(d-1)/2}\,.
\Ee
\end{theorem}

Note that $d_1\ge 2$ and thus
the estimates \eqref{TlakL1} can be summed in $k$. Hence
Theorem \ref{main-theoremL1}
is an immediate consequence of Theorem \ref{TjL1thm}.

\subsection*{\it Dyadic decompositions}
For the Hardy space bounds we shall need to combine the dyadic pieces in $j$ and also refine the dyadic decomposition in \eqref{adydecomposition}.

Define
\begin{align}
\label{Vjdefinition}
V_j &= 2^{-j(d-1)/2} \chi_1(2^{-2j}L) T^0_{2^j}
\\
\label{Wjdefinition}
W_j&= 2^{-j(d-1)/2}\chi_1(2^{-2j} L) (m_{2^j}(L)- T^0_{2^j})
\end{align}

In section \S\ref{FIO} we shall use standard estimates on Fourier integral operators to prove
\begin{theorem}\label{Vtheorem}
The operator $\cV=\sum_{j>100}V_j$ extends to a bounded operator from
$h^1_\iso$ to $L^1$.
\end{theorem}

We further decompose the pieces $W_j$ in \eqref{Wjdefinition} and let
\Be\label{Wjndefinition}
\begin{aligned}
W_{j,0} &= \zeta_0(2^{-j}|U|)\,
W_j
\\
W_{j,n} &=  \zeta_1(2^{-j-n}|U|) W_j\,;
\end{aligned}
\Ee
here again  $\zeta_0$, $\zeta_1$ as in  \S\ref{cutoffsect}, i.e.
$\zeta_0$ supported
in $(-1,1)$, $\zeta_1$ supported in $\pm(1/2,2)$ so that
$\zeta_0+\sum_j \zeta_1(2^{1-j}\cdot)\equiv 1.$

By the description \eqref{jtspec} of the joint spectrum of $L$ and $|U|$
and the support property \eqref{chi1support}
we also have
$$\chi_1(2^{-2j} L) \zeta_1(2^{-j-n}|U|)= 0  \text{ when } 2^{2j+10} \le
2^{j+n-1}\,,
$$
i.e when $j\le n-11$ and thus
\Be \label{vanishingterms}  W_{j,n}=0 \text{ when $n\ge j+11$ }.\Ee

Observe from \eqref{jtspec}, as in the discussion following \eqref{Tkla} that,
for $k=1,2,\dots$,
$$\zeta_0(2^{-j}\rho_2) \eta_0(\frac{\tau}{2^j}\rho_2-k\pi)= 0
\text {
for  $\tau \in (\frac 1{16},4)$, $\rho_2\ge 0$, } \text{ if } 2^j\le (k-\tfrac 58)\pi 2^{j}/4\,,
$$
and \begin{multline*}
\zeta_1(2^{j-n}\rho_2)
 \eta_0(\frac{\tau}{2^j}\rho_2-k\pi)= 0 \text {
for  $\tau \in (\frac 1{16},4)$, $\rho_2\ge 0$, }
\\ \text{ if }  2^{j+n+1} \le 2^j(k-\tfrac 58)\pi/4 \text { or } 16\cdot 2^j (k+\tfrac 58)\pi \le 2^{j+n-1}\,.
\end{multline*}
Thus  we have
for $k=1,2,\dots$,
\begin{align*}
&\zeta_0(2^{-j} |U|) T^k_{2^j} = 0 \text{ when } k\ge 2\,,
\\
&\zeta_1(2^{-j-n} |U|) T^k_{2^j} = 0 \text{ when } k\notin
[2^{n-8}, 2^{n+2}]\,.
\end{align*}
Let
\Be\label{jndef}
\cJ_n=\begin{cases} \{1\}, &n=0\,,\\
\{k:  2^{n-8}\le k\le 2^{n+2}\}, &n\ge 1\,.
\end{cases}
\Ee
Then by \eqref{kdecomposition} we have
$m_{2^j}(L)- T^0_{2^j}=\sum_{k=1}^\infty T^k_{2^j}$
and therefore we get
\begin{subequations}
\begin{align} \label{Wj0k}
W_{j,0} &=
2^{-j(d-1)/2} \chi_1(2^{-2j}L) \zeta_0(2^{-j}|U|)\sum_{k\in \cJ_0}
T^k_{2^j}\,,
\\
\label{Wjnk}
W_{j,n} &= 2^{-j(d-1)/2} \chi_1(2^{-2j}L) \zeta_1(2^{-j-n}|U|)
\sum_{k\in \cJ_n}T^k_{2^j}\,.
\end{align}
\end{subequations}

Observe that Theorem \ref{TjL1thm}
implies
\begin{equation}\label{Wjnest}
\|W_{j,n}\|_{L^1\to L^1} \lc 2^{-n(d_1-1)/2}
\end{equation}
uniformly in $j$.

Define
for $n=0,1,2,\dots$
\Be \label{Wnopdefinition}
\cW_n =\sum_{j>100} W_{j,n}
\Ee
Theorem \ref{h1thm} will then be a consequence of Theorem \ref{Vtheorem}
and
\begin{theorem} \label{refinedh1thm}
The operators $\cV$ and $\cW_n$ are bounded from $h^1_\iso$ to $L^1$; moreover
\Be\label{Wnhardybd}
\|\cW_n\|_{h^1_\iso\to L^1}\lc (1+n)2^{-n(d_1-1)/2}
\Ee
\end{theorem}
The proofs will be given in \S\ref{FIO} and  \S\ref{hardyspaceestimates}.

\section{Fourier integral estimates}
\label{FIO}
In this section we shall reduce the proof of the estimates for $T^0_\la$ and $\cV$ in Theorems \ref{TjL1thm} and \ref{refinedh1thm} to standard bounds for Fourier integral operators in \cite{SSS} or \cite{beals}.

We will prove  a preliminary lemma that allows us to add or suppress
$\chi_1(\la^{-2} L)$ from the definition  of $T_\la^0$.

\begin{lemma} \label{Tla0error}
For $\la>2^{10} $ we have
$$\|T^0_\la -\chi_1(\la^{-2}L) T^0_\la\|_{L^1\to L^1}\lc C_N \la^{-N}
$$ for any $N$.
\end{lemma}.

\begin{proof}
The operator $T^0_\la -\chi_1(\la^{-2}L) T^0_\la$ can be written as
$b_\la(|L|, |U|)$ where
$$b_\la(\rho_1,\rho_2)=
\la^{1/2} (1-\chi_1(\la^{-2}\rho_1))
\la^{1/2} \int a_\la(\tau) e^{i\varphi(\tau,\rho_1,\la)}
\eta_0(\tau\rho_2/\la) d\tau
$$
with $$\varphi(\tau,\rho_1,\la) = \frac{\la}{4\tau}+\frac{\tau\rho_1}{\la}\,.$$
Only the values of $\rho_1\le  \la^{2} 2^{-9}$ and $\rho_1\ge 2^{9}\la^2$ are relevant.
Now $$\frac{\partial\varphi}{\partial \tau}= -\frac{\la}{4\tau^2}+\frac{ \rho_1}\la$$ and $(\partial/\partial\tau)^n \varphi= c_n \la\tau^{-n-1}$ for $n\ge 2$.
Note that for $\rho_1\ge 2^9 \la^2$ we have
 $|\varphi'_\tau| \ge \rho_1/\la - (16^2/4)\la \ge
 \rho_1\la^{-1} (1- 2^{-9}2^6)\ge \rho_1/(2\la)$.
 Similarly for
 $\rho_1\le 2^{-9} \la^2$ we have
 $|\varphi'_\tau| \ge \la/16- 16 \cdot 2^{-9}\la \ge 2^{-5}\la$.
 Use integrations by parts to conclude
 that
 $$
\Big| \frac{\partial^{n_1+n_2}[b_\la(\la^2 \cdot,\la \cdot )]
}{(\partial\rho_1)^{n_1}(\partial\rho_2)^{n_2}}
(\rho_1,\rho_2)
\Big|  \le C_{n_1,n_2,N}\la^{-N}
 $$
 and in view of the compact support of
$ b_\la(\la^2 \rho_1,\la \rho_2)$ the assertion can be deduced from
 a result in  \cite{MRS2}
 (or alternatively from Hulanicki's result \cite{hulanicki} and a Fourier expansion in $\rho_2$).
\end{proof}

\subsection*{\it The convolution kernel for $T_\la^{0}$} It  is given by
\begin{multline*}
K^{0}_\la(x,u) \,=\, \la^{1/2} \int_{\bbR^{d_2}}
\int_\bbR e^{i\frac{\la}{4s }}
 a_\la(s ) \eta_{0} (2\pi|\mu|\tfrac {s }{ \la})\,\times
\\
\Big(\frac{|\mu|}{2\sin (2\pi|\mu|s /\la)}\Big)^{d_1/2}
e^{-i|x|^2\frac {\pi}{ 2} |\mu|\cot (2\pi  |\mu|s /\la)} ds
\, e^{2\pi i \inn{u}{ \mu}} d\mu\,.
\end{multline*}

We introduce frequency variables   $\theta= (\om, \sigma)$ on the cone
\Be\label{Gammacone}\Gamma_\delta= \{\theta=(\om, \sigma)\in \bbR^{d_2}\times \bbR: \, |\om|\le ( \pi -\delta)\sigma, \,\,\sigma>0\},\Ee
Set
$$\om=\frac{\pi \mu}2, \quad \sigma= \frac{\la}{4s }.$$
Note that $\sigma \approx \la$ for $s   \in \supp(a_\la)$.
We note that we will consider the case $\delta=\pi/4$ in view of the support of $\eta_0$ but any choice of $\delta\in (0,\pi/4)$ is permissible with some constants below depending on $\delta$.

If we set \Be\label{gdef}
g(\tau):= \tau \cot\tau,
\Ee
the above integral becomes
\Be\label{oscillintK00}
K^{0}_\la(x,u) \,=\, \iint e^{i\Psi(x,u,\om,\sigma)} \beta_\la(\om,\sigma) d\om\,d\sigma
\Ee
with
$$\Psi(x,u,\om,\sigma)=\sigma \big(1-|x|^2 g(|\om|/\sigma )\big) +\inn {4u}\om$$
and
$$
\beta_\la(\om,\sigma)=
4^{-1}\Big(\frac 2\pi\Big)^{\frac{d_1}{2}+d_2}
\la^{3/2}
\sigma^{\frac {d_1}{2}-2}
a_\la (\tfrac{\la}{4\sigma}) \eta_{0} (\tfrac{|\om|}{|\sigma})\Big(\frac{\tfrac{|\om|}{\sigma|}}
{2\sin(\tfrac{|\om|}{\sigma})}\Big)^{d_1/2} \,.
$$
The $\beta_{\la}$ are  symbols of order $\frac{d_1-1}{2}$
uniformly in $\la$, and  supported in $\Gamma$. The same applies to
 $\sum_{k>10} \beta_{2^k}$.

We will need formulas for the derivatives of $\Psi$ with respect to the frequency variables
$\theta=(\om,\sigma)$:
\Be
\label{Psinuder}
\begin{aligned}
&\frac{\partial \Psi}{\partial \om_i} = 4u_i -|x|^2  \frac{\om_i}{\sigma} \frac{g'(\tfrac{|\om|}{\sigma}}{\tfrac{|\om|}{\sigma}}
\\
&\frac{\partial \Psi}{\partial \sigma} =1-|x|^2\big( g(\tfrac{|\om|}{\sigma}) -  \tfrac{|\om|}{\sigma} g'(\tfrac{|\om|}{\sigma})\big)
\end{aligned}
\Ee
Now $g$ is analytic for $|\tau|<2\pi$ and we have
\begin{subequations}
\begin{align}
\label{gfirstderiv}
g'(\tau)&= \frac{\sin(2\tau)-2\tau}{2\sin^2\tau}
\\
\label{gfsecderiv}
g''(\tau)&=\frac{2(\tau \cos \tau - \sin \tau)}{\sin^3 \tau}
\end{align}
\end{subequations}
Observe that $$g'(\tau)<0 \text { and } g''(\tau)<0 \text{  for $0<\tau<\pi$.}$$
Moreover as $\tau \to 0$,
$$g(\tau)=1-\tau^2/3+O(\tau^4)$$
and hence $g'(0)=0$ and $g''(0)=-2/3$.
The even expression $$g(\tau)-\tau g'(\tau)= 1+ \int_0^\tau (-sg''(s)) ds $$ will frequently occur; from the above we get
\Be\label{g-tg'}\begin{aligned}
g(\tau)-\tau g'(\tau)\ge  1,  \text{ for } 0\le |\tau|<\pi\,,
\\
|g(\tau)-\tau g'(\tau)|\le  10,  \text{ for } 0\le |\tau|<3\pi/4\,.
\end{aligned}
\Ee

\begin{lemma} \label{decaylemmalargexu} We have
\Be\label{largexu}|K_{\la}^{0}(x,u)| \lc
\la^{\frac{d_1+2d_2+1}{2}-N}
(|x|^2+|u|)^{-N},
\text{ $|x|^2+4|u|>2$.}
\Ee
and
\Be\label{xsmall}|K_{\la}^{0}(x,u)| \lc \la^{\frac{d_1+2d_2+1}{2}-N}
(1+|u|)^{-N} , \text{ $|x|^2\le 1/20.$}\Ee
\end{lemma}

\begin{proof} If $|x|\ge \sqrt 2$ we may integrate by parts with respect to $\sigma$ (using \eqref{g-tg'}), and obtain
$$|K_{\la}^0(x,u)|\lc_N \la^{\frac{d_1+2d_2+1}{2}-N}  |x|^{-N}, \quad |x|\ge \sqrt 2\,.
$$
If $|u|\le 10|x|^2$ this also yields \eqref{largexu}.
Since $\max_{|\tau|\le 3\pi/4}|g'(\tau)|\le 3\pi/2$ we have
$|\nabla_\om \Psi|\ge 4|u|- (3\pi/2)|x|^2$ and hence $|\nabla_\om \Psi|\ge |u|$ when
 $|u|\ge 10|x|^2$.
 Thus integration by parts in $\omega$ yields
$$|K_{\la}^0(x,u)|\lc_N \la^{\frac{d_1+2d_2+1}{2}-N}  |u|^{-N},
\quad \text{$|u|\ge 10|x|^2$.}
$$
This proves
\eqref{largexu}.

Since $|g'(\tau)|\le 3\pi$ for $|\tau|\le 3\pi/2$
we have $|\nabla_\om \Psi|\ge 2|u|$  if $|x|^2 \le 2|u|/3\pi$ and
$|\Psi_\sigma |\ge 1/2$ if $|x|^2\le 1/20$.
Integrations by parts imply \eqref{xsmall}.
\end{proof}

\subsection*{\it Fourier integral operators}

Let $\rho\ll 10^{-2}$. Let $\chi\in C^\infty_c(\bbR^d\times \bbR^d)$ so that
$$\chi(x,u,y,v)=0 \text{  for } \begin{cases}
|y|+|v|\ge \rho,\\ |x-y|<1/20, \\ |x-y|^2+ |u-v| \ge 4 .\end{cases}
$$
Let
$$b_\la(x,y,u,v,\om,\sigma)= \chi(x,u,y,v) \beta_\la(\om,\sigma),$$
 let as before $g(\tau)= \tau\cot\tau$, and let
\Be\label{Phidef}\begin{aligned}
\Phi(x,u,&y,v,\om,\sigma)\,=\, \Psi(x-y,u-v+
\tfrac 12 \inn {\vec Jx}{y},
 \om,\sigma)\\
&=\sigma \big(1-|x-y|^2 g(|\om|/\sigma )\big)
+\sum_{i=1}^{d_2}
(4u_i-4v_i-2 x^\intercal J_i y)\om_i\,.
\end{aligned}
\Ee
Let $\fF_\la$ be the Fourier integral operator with Schwartz kernel
\Be \label{Fourierint}\cK_\la(x,u,y,v)=
 \iint e^{i\Phi(x,u,y, v, \om,\sigma)}  b_\la(\om,\sigma) d\om\,d\sigma.
\Ee

Given Lemma \ref{decaylemmalargexu} it suffices to prove the inequalities
\Be\label{L1Fla}
\| \fF_\la\|_{L^1\to L^1} \le \la^{\frac{d-1}{2}}.
\Ee
and
\Be\label{h1F}
\Big\|\sum_{k>C} 2^{-k(d-1)/2} \fF_{2^k}  \Big\|_{h^1\to L^1} <\infty .
\Ee
To this end we apply results in \cite{SSS} on Fourier integral operators associated with canonical graphs and now check the required hypotheses.

\subsection*{\it Analysis of the phase function $\Phi$}
We compute the first derivatives:
\begin{align*}
\Phi_{x_j}&= -2\sigma (x_j-y_j) g(\phqu) -2 \sum_{i=1}^{d_2} \om_i e_j^\intercal J_i y
\\
\Phi_{u_i}&=4\om_i
\\
\Phi_{\om_i}&=-|x-y|^2 g'(\phqu) \tfrac{\om_i}{|\om|}
+4u_i-4v_i-2 x^\intercal J_i y
\\
\Phi_{\sigma}&=\big(1-|x-y|^2 g(\phqu )\big) + |x-y|^2 \phqu g'(\phqu)
\end{align*}
For the second derivatives we have, with $\delta_{jk}$ denoting the Kronecker delta and $J^\om=\sum_{i=1}^{d_2} \om_i J_i$
$$ \begin{aligned}
\Phi_{x_jy_k}&= 2\sigma g(\phqu)\delta_{jk} - 2  e_j^\intercal J^\om e_k\,,
\\ \Phi_{x_j v_l}&=0\,,
\\
\Phi_{x_j\om_l} &=-2(x_j-y_j)g'(\phqu) \tfrac{\om_l}{|\om|}  - 2 e_j^\intercal J_l y \,,
\\
\Phi_{x_j\sigma} &=2(x_j-y_j) \big(-  g(\phqu)+ \phqu g'(\phqu)\big)\,,
\end{aligned}
$$ and
$$\Phi_{u_iy_k}=0\,, \quad
\Phi_{u_i v_l}=0\,,\quad \Phi_{u_i \om_l}= 4\delta_{il}\,, \quad
\Phi_{u_i\sigma}=0\,.
$$
Moreover
$$
\begin{aligned}
&\Phi_{\om_i y_k}=2(x_k-y_k) g'(\phqu) \tfrac{\om_i}{|\om|}- 2 x^\intercal J_ie_k
\\
&\Phi_{\om_iv_l}=-4 \delta_{il}
\\
&\Phi_{\om_i\om_l}= -|x-y|^2 \big( g'(\phqu)
 \tfrac{\delta_{il}|\om|^2-\om_i\om_l}{|\om|^3} + g''(\phqu) \tfrac{\om_i\om_l}{\sigma|\om|^2} \big)
\\
&\Phi_{\om_i\sigma}= |x-y|^2 \tfrac{\om_i}{\sigma^2} g''(\phqu)
\end{aligned}$$ and
$$
\begin{aligned}
&\Phi_{\sigma y_k}=2(x_k-y_k) \big( g(\phqu)-\phqu g'(\phqu) \big)
\\
&\Phi_{\sigma v_l}=0
\\
&\Phi_{\sigma \om_l}=|x-y|^2 \tfrac{\om_l}{\sigma^2}g''(\phqu)
\\
&\Phi_{\sigma\sigma}=-|x-y|^2 \tfrac{|\om|^2}{\sigma^3} g''(\phqu)
\end{aligned}
$$
The required
$L^2$ boundedness properties follow if we can show that
associated canonical relation is locally the graph of a canonical transformation;
this follows from the invertibility of the matrix
\Be\label{PhimixedHess}
\begin{pmatrix}
\Phi_{xy}&\Phi_{xv}&\Phi_{x\om}&\Phi_{x\sigma}
\\
\Phi_{uy}&\Phi_{uv}&\Phi_{u\om}&\Phi_{u\sigma}
\\
\Phi_{\om y}&\Phi_{\om v}&\Phi_{\om\om}&\Phi_{\om\sigma}
\\
\Phi_{\sigma y}&\Phi_{\sigma v}&\Phi_{\sigma \om}&\Phi_{\sigma\sigma}
\end{pmatrix}\,,
\Ee
see \cite{hoermander-fio}.
This matrix is given by
$$
\begin{pmatrix}
2\sigma g I_{d_1}-2 J^\om &0&  (*)_{13} & 2(x-y) (\tau g'-g)
\\
0&0&4I_{d_2}&0
\\
(*)_{31} & -4I_{d_2} &(*)_{33} &(*)_{34}
\\
2(x-y)^\intercal (g-\tau g')&0 &(*)_{43} &-|x-y|^2\sigma^{-1} \tau^2g''
\end{pmatrix}\,,
$$
where
 $\tau=\phqu$,  $g, g', g''$ are evaluated at $\tau=\phqu$, and
$x-y$ is considered a $d_1\times 1$ matrix,
 $ (*)_{13}$ is a $d_1\times d_2$-matrix,
  $ (*)_{31}$ is a $d_2\times d_1$-matrix,
  $ (*)_{33}$ is a $d_2\times d_2$-matrix,
    $ (*)_{34}$ is a $d_2\times 1$-matrix,  and
    $ (*)_{43}=(*)_{34}^\intercal$.

The determinant $D$
  of the displayed matrix is equal to
\Be\label{detD}D=16^{d_2} \det
\begin{pmatrix}
2\sigma g I_{d_1}-2J^\om & 2(x-y) (\tau g'-g)
\\
2(x-y)^\intercal (g-\tau g')& -|x-y|^2 \sigma^{-1}\tau^2g''
\end{pmatrix}\,.
\Ee

To compute this we use the formula
$$
\begin{pmatrix}
I&0\\a^\intercal &1
\end{pmatrix}
\begin{pmatrix} A&-b\\b^\intercal &\gamma
\end{pmatrix}
\begin{pmatrix} I&-a\\0&1
\end{pmatrix}
\,=\,
\begin{pmatrix}A&
-Aa-b
\\a^\intercal A+b^\intercal & -a^\intercal A a- 2a^\intercal b+\gamma
\end{pmatrix}\,.
$$
If $A$ is invertible we can choose
 $a= - A^{-1} b$.
 Since $b^\intercal S b=0$ for the skew symmetric matrix $S= (A^{-1})^\intercal -A^{-1}$
 this  choice of $a$ yields the matrix
 $$\begin{pmatrix}
 A&0\\-b^\intercal (A^{-1})^\intercal A +b^\intercal &-b^\intercal (A^{-1})^\intercal b-2b^\intercal A^{-1} b +\gamma
 \end{pmatrix}=\begin{pmatrix} A&0\\  *  & \gamma +b^\intercal A^{-1} b\end{pmatrix}
 $$
 and
 hence
 \Be \label{determinantformula}
 \det \begin{pmatrix} A&-b\\b^\intercal &\gamma\end{pmatrix} = ( \gamma+ b^\intercal A^{-1} b) \det (A)\,.
 \Ee

 \begin{lemma}\label{skewsymnew} Let $c,\Lambda\in \bbR$, $c^2+\La^2\neq 0$. Let $S$ be a skew symmetric $d_1\times d_1$- matrix satisfying
 $S^2=-\La^2 I$. Then $cI+S$ is invertible with
 $$(cI+S)^{-1}= \frac{c}{c^2+\Lambda^2} I-\frac{1}{c^2+\Lambda^2} S,$$
 and $\det (cI+S)=(c^2+\Lambda^2)^{\frac {d_1}2}.$
 \end{lemma}
 \begin{proof}  $(cI+S)(cI+S)^*=(cI+S)(cI-S)=c^2 I-S^2= (c^2+\Lambda^2)I$.
 \end{proof}

 In our situation \eqref{detD}  we have $A=cI+S,$ with
$$\begin{aligned}
c&=2\sigma g(\phqu), \\ S&=-2J^\om\,,\end{aligned}$$
moreover,
$$\begin{aligned}
\Lambda&=2|\om|,\\
\gamma&=-|x-y|^2\sigma^{-1} (\tfrac{|\om|}{\sigma})^2 g''(\phqu),
 \\b&=2(x-y) (g(\phqu)-\phqu g'(\phqu))\,.\end{aligned}
$$
In particular, if we recall that $\tau=|\om|/\sigma,$ we see that
$$\det A=\big((2\si g(\tau))^2+(2|\om|)^2\big)^{\frac{d_1}2}=(2\sigma)^{d_1}\Big(\frac \tau {\sin \tau}\Big)^{d_1}.
$$
Moreover,
\begin{align*}&\gamma+ b^\intercal A^{-1} b\\
&=|x-y|^2\Big( -  \frac{|\om|^2}{\sigma^3} g''(\phqu)
+ 4\big(g(\phqu)-\phqu g'(\phqu)\big)^2 \frac{ 2\sigma g(\phqu)}
{4\sigma^2 g(\phqu)^2 + 4|\om|^2}\Big)\\
&=\frac{|x-y|^2}{\sigma} \Big(-\tau^2 g''(\tau) +2(g(\tau)-\tau g'(\tau))^2 \frac{g(\tau)}{g(\tau)^2+\tau^2} \Big).
\end{align*}
From  \eqref{gfirstderiv}, we get $$g(\tau)-\tau g'(\tau)=\Big(\frac
\tau{\sin\tau}\Big)^2,$$  and in combination with
 \eqref{gfsecderiv}  this implies after a calculation  that
 $$
 \gamma+ b^\intercal A^{-1} b=\frac{|x-y|^2}{\sigma}2\Big(\frac \tau {\sin \tau}\Big)^2.
 $$
 Thus we see from \eqref{determinantformula} that  the  determinant of the matrix \eqref{PhimixedHess} is  given by
 \begin{equation}\label{ detHessian}
D=2^{d_1+4d_2+1} \sigma^{d_1-1}\Big(\frac \phqu {\sin \phqu}\Big)^{d_1+2}.
\end{equation}
This shows that $D>0$  for $\phqu\in [0,\pi),$ and $D\sim \sigma^{d_1-1}$ for $\phqu\in [0,\pi-\de],$
for every  sufficiently small $\de>0.$ In particular, the matrix \eqref{PhimixedHess} is invertible for
$\phqu\in [0,\pi-\de].$

\bigskip

\medskip

We now write $$\fF_\la f(x)= \int K_\la(x,y) f(y) dy$$ where $K_\la$ is given by our oscillatory integral representation \eqref{Fourierint}. In that formula
 we have $d_2+1$  frequency variables $d_2+1$, and thus, given any $\alpha\in \bbR$  the operator
convolution with $\sum_{k>C} \fF_{2^k} 2^{-k\alpha}$  is   a Fourier integral operator of
order $$\frac{d_1-1}{2}-\alpha- \frac{d-(d_2+1)}{2}=-\alpha\,.$$

With these observations we can now apply the boundedness result of \cite{SSS}
and deduce that
 $$ \|\fF_\la f\|_1\lc \la^{\frac{d-1}{2} }\|f\|_1$$ and
 $$ \Big\|\sum_{k>C} 2^{-k(d-1)/2} \cF_{2^k}  f_\rho\Big \|_1 \lc 1$$
 for  standard $h_1$ atoms supported in $B_\rho$.
 But atoms associated to  balls centered at the  origin are also atoms in our Heisenberg Hardy space
$h^1_\iso$.  Thus if we also take into account Lemma
\ref{decaylemmalargexu} and use translation invariance under Heisenberg translations we get
$$\Big\|\sum_{k\ge 0} T^0_{2^k} f \Big \|_1\lc \|f\|_{h^1_\iso}.$$

{\it Remark.} We also have
\begin{multline*}
\begin{pmatrix}
\Phi_{\om\om}&\Phi_{\om\sigma}\\ \Phi_{\sigma\om}&\Phi _{\sigma\sigma}
\end{pmatrix}
\\=|x-y|^2
\begin{pmatrix}
-\big(g'(\phqu)\frac{I_{d_2}|\om|^2-\om\om^\intercal}{|\om|^3}
+g''(\phqu) \frac{\om\om^\intercal}{\sigma|\om|^{2}}\big)
&  \tfrac{\om}{\sigma^2} g''(\phqu)
\\
\tfrac{\om^\intercal}{\sigma^2}
g''(\phqu)& -
\tfrac{|\om|^2}{\sigma^3} g''(\phqu)
\end{pmatrix} \,
\end{multline*}
which has   maximal rank $d_2+1-1=d_2$.
Thus the above result can also be deduced from Beals \cite{beals},
via the equivalence of phase functions theorem.

\section{The operators $T_\la^k$}
\label{Tlaksection}
We now consider the operator  $T^\la_k$, for $k\ge 1$,  as defined in \eqref{Tkla}.
In view of the singularities of $\cot$
we need a further decomposition in terms of the
distance to the singularities.
For $l=1,2,\dots$ let
$\eta_l(s)=\eta_0(2^{l-1} s)-\eta_0(2^l s)$ so that $$\eta_0(s)=
\sum_{l=1}^\infty  \eta_l(s)\text{  for $s\neq 0$.}$$
Define
\Be
\label{Tlakdef}
T^{k,l}_\la =\la^{1/2}\int e^{i\frac{\la}{4\tau}}
 a_\la(\tau) \eta_l (\tfrac {\tau}{ \la}|U|- k\pi)
e^{i\tau L/\la}
d\tau;
\Ee
then
\Be\label{decompTlak}
T^k_\la =\sum_{l=1}^\infty
T^{k,l}_\la\,.
\Ee

 From the formula \eqref{Kklaexpression} for the kernels $K_\la^k$ we get
 a corresponding formula for the kernels $K^{k,l}_\la$, namely
 \begin{multline*}
K^{k,l}_\la(x,u) \,=\, \la^{1/2} \int_{\bbR^{d_2}}
\int_\bbR e^{i\frac{\la}{4\tau}}
 a_\la(\tau) \eta_l (2\pi|\mu|\tfrac {\tau}{ \la}- k\pi)\,\times
\\
\Big(\frac{|\mu|}{2\sin (2\pi|\mu|\tau/\la)}\Big)^{d_1/2}
e^{-i|x|^2\frac {\pi}{ 2} |\mu|\cot (2\pi  |\mu|\tau/\la)} d\tau
\, e^{2\pi i \inn{u}{ \mu}} d\mu\,.
\end{multline*}
Now we use polar coordinates in $\bbR^{d_2}$ and the fact that the
Fourier transform of the surface carried measure on the unit sphere in
$\bbR^{d_2}$ is given by  $$(2\pi)^{d_2/2}\bess(2\pi|u|), \text{ with }
 \bess(\sigma):=\sigma^{-\frac{d_2-2}2} {J}_{\frac{d_2-2}{2}} (\sigma)$$
(the standard Bessel function formula, \cf. \cite{stw}, p.154).
Thus
\begin{multline*}
K^{k,l}_\la(x,u) \,=\,
\la^{1/2} \int_{0} ^\infty
\int_\bbR e^{i\frac{\la}{4\tau}}
 a_\la(\tau) \eta_l (2\pi\tau \rho / \la- k\pi)\,\times
\\
\Big(\frac{\rho}{2\sin (2\pi\tau \rho/\la)}\Big)^{d_1/2}
e^{-i\frac {\pi}{ 2}|x|^2
 \rho\cot (2\pi  \rho\tau/\la)} \,d\tau\,(2\pi)^{d_2/2} \bess (2\pi
\rho |u|)\,\rho^{d_2-1} d\rho\,.
\end{multline*}
In this integral we introduce new  variables
\Be\label{stvariables}
(s,t)= \Big( \frac{1}{4\tau},  \frac{2\pi \tau \rho}{\la}\Big),
\Ee
so that $(\tau,\rho)=((4s)^{-1}, 2\la ts/\pi)$ with $d\tau d\rho= \la(2\pi s)^{-1} ds dt$. Then we obtain  for $k\ge 1$
\begin{multline} \label{Klaklrep}
K^{k,l}_\la(x,u) = \la^{d_2+\frac{d_1+1}2}
\quad \times \\\iint
 \beta_\la(s)
\eta_l(t-k\pi)
\Big(\frac{t}{\sin t}\Big)^{d_1/2} t^{d_2-1}
e^{i\la s \psi(t,|x|) } \bess (4s \la  t|u| )
\, ds \, dt
\end{multline}
 where
\begin{equation}\label{psiphase}
\psi(t,r)= 1- r^2 t\cot t
\end{equation}
and \Be\label{betala}\beta_\la(s)= 2^{\frac{3d_2}2-2}\pi^{-\frac{d_1+d_2}{2}}
a_\la(\tfrac 1{4s}) \,s^{\frac{d_1}{2}+d_2-2};\Ee
thus $\beta_\la$ is $C^\infty$ with bounds uniform in $\la$, and  $\beta_\la$ is  also supported in $[1/16,4]$.


In the next two sections  we shall prove
the
$L^1$ estimates
\Be\label{L1boundsKlakl}
\sum_{k<8\la}\sum_{l=0}^{\infty}\iint \la^{-\frac{d-1}{2}} |K^{k,l}_\la(x,u)| \, dx\, du = O(1)
\Ee
and
Theorem \ref{TjL1thm} and then also
Theorem \ref{main-theoremL1} will  follow by summing the pieces.
 Moreover we shall give some
refined estimates which will be used in the proof of Theorem
\ref{refinedh1thm}.

\subsection{\it An $L^\infty$ bound for the kernels}
The  expression
\Be
\label{Clakl}
\fC_{\la,k,l}
= \la^{1+\frac{d_2}{2}}  k^{d_2-1}(2^{l}k)^{\frac{d_1}{2}}
\Ee
will frequently appear in pointwise estimates,
namely as  upper bounds for the integrand in the integral defining
$\la^{-\frac{d-1}{2}}K^{k,l}_\la$.
%
Note that
\Be \label{Claklinfty}
\|\la^{-\frac{d-1}{2}}K^{k,l}_\la\|_\infty
\lc 2^{-l}\fC_{\la,k,l}\,;
\Ee
the additional factor of $2^{-l}$ occurs since the
  integration  in $t$ is  over the union of two  intervals of length $\approx 2^{-l}$.

\subsection{\it Formulas for the phase functions}\label{reductions}


For later reference we gather some formulas for the $t$-derivatives
of the phase $\psi(t,r)=1-r^2 t\cot t$:
\begin{subequations}
\begin{align}
\psi_\mut(\mut,r)&= r^2 \Big(\frac{\mut}{\sin^2\mut}-\cot\mut\Big)
\label{psit}
\\&= r^2\Big(\frac{2t-\sin(2t)}{2\sin^2 t}\Big);
\label{psitalt}
\end{align}
\end{subequations}
moreover
\begin{align}
\psi_{\mut\mut}(\mut,r)&= \frac{2r^2}{\sin^3\mut}
\big(\sin \mut-\mut \cos \mut\big)\,=\, \frac{2r^2}{\sin^3\mut}\int_0^t\tau\sin\tau\,d\tau.
\label{psitt}
\end{align}
Observe  that $\psi_{tt}=0$ when $\tan t=t$ and  $t\neq 0$ and thus
$\psi_{tt}(t,r)\approx r^2$ for $0\le t\le \tfrac {3\pi}{4}$,
namely, we use $\frac{2\sqrt 2}{3\pi}t\le\sin t\le t$ to get the crude estimate
\Be \label{psittequiv}
\pi^{-1} r^2< \psi_{tt}(t,r)< \pi^3 r^2, \quad 0< t\le \tfrac {3\pi}{4}\,.
\Ee
It is also straightforward to establish estimates for
the higher derivatives:
\begin{equation} \label{higherpsidersmallt}
|\partial_t^{n} \psi(t,r)|\lc r^2, \qquad |t|\le 3\pi/4
\end{equation}
and
\begin{equation} \label{higherpsider}
\partial_t^{n} \psi(t,r)
=O\Big(\frac{r^2 |t|}{|\sin t|^{n+1}}\Big), \quad
\end{equation}
for all $t$.

\medskip

\subsection{\it Asymptotics in the main case  $|u|\gg   (k\la)^{-1}$}
We shall see in the next section that there are straightforward  $L^1$ bounds
in the region where $|u|\lc (k+1)^{-1}\la^{-1}$. We therefore concentrate on
the region
$$\{(x,u):|u|\ge C(k+1)^{-1}\la^{-1}\}$$ where
we have to take into account the oscillation
of the terms
$\bess (4s \la  t|u| )$.
The standard asymptotics for Bessel functions imply that for
\Be \bess(\sigma) = e^{- i|\sigma|}\vpi_1(|\sigma|)+
e^{i|\sigma|}\vpi_2(|\sigma|),\qquad |\sigma|\ge 2,
\label{bessel}
\Ee
where $\vpi_1, \vpi_2\in S^{-(d_2-1)/2}$
are  supported  in $\Bbb R\setminus [-1,1]$.

Thus we may split, for $|u|\gg(k+1)^{-1}\la^{-1}$,
\Be
\label{ABsplitkl} \la^{-\frac{d-1}2} K^{k,l}_\la(x,u) = A_\la^{k,l}(x,u)+ B^{k,l}_\la
(x,u)\Ee
where, with $\fC_{\la,k,l}$ defined in \eqref{Clakl},
\Be\label{Adefin}
A_\la^{k,l} (x,u) = \fC_{\la,k,l} \iint \eta_{\la,k,l}(s,t)
e^{i\la s(\psi(t, |x|)- 4t|u|)}
\vpi_1(4\la st|u|)\,  dt \,ds\,,
\Ee
and
\Be\label{Bdefin}
B_\la^{k,l} (x,u) = \fC_{\la,k,l} \iint \eta_{\la,k,l}(s,t)
 e^{i\la s(\psi(t, |x|)+4t|u|)}
\vpi_2(4\la st|u|)\,  dt \,ds\,;
\Ee
here, as before
$\psi(t,r)=1-r^2 t\cot t$ and,
with $\beta_\la$ as in \eqref{betala},
\begin{subequations}
\label{etadefinitions}
\begin{align}
\label{etadefinitionzero}
\eta_{\la,0}(s,t)&=\beta_\la(s)\eta_0(t)
\big(\frac {t}{\sin t}\big)^{d_1/2} t^{d_2-1}\,,
\\
\label{etadefinitionkl}
\eta_{\la,k,l}(s,t)&= \beta_\la(s) \eta_l(t-k\pi)
\big(\frac {t/k}{2^l\sin t}\big)^{d_1/2}
(t/k)^{d_2-1}\,.
\end{align}
\end{subequations}

Note that $\|\partial_s^{N_1}\partial_t^{N_2} \eta_{\la,k,l}\|_\infty \le C_{N_1,N_2} 2^{lN_2}$. Moreover if
\Be \label{Jkl}
J_{k,l}:=
(k\pi- 2^{-l}\tfrac{5\pi}4, k\pi-2^{-l}\tfrac {3\pi}8]
\cup
[k\pi+ 2^{-l}\tfrac{3\pi}8, k\pi+2^{-l}\tfrac {5\pi}4)
\Ee
then
\Be\label{supportetakl}\eta_{\la,k,l}(s,t)\neq 0 \implies t\in J_{k,l}\,.
\Ee

The main contribution in our estimates comes from the kernels
 $A_{\la}^{k,l}$  while  the kernels
$B_{\la}^{k,l}$   are negligible  terms with rather small $L^1$ norm.
The latter will follow from the support properties of $\eta_{\la,k,l}$ and the observation
that $$\partial_t (\psi(t,|x|)+4t|u|)\neq 0, \quad (x,u)\neq (0,0);$$
\cf. \eqref{psitalt}.
As a consequence
only the kernels  $A_{\la}^{k,l}$  will exhibit the singularities of the kernel away from the origin.

\subsection{\it The phase functions and the  singular support}\label{phase-singsupp}
We introduce polar coordinates in $\bbR^{d_1}$ and $\bbR^{d_2}$
(scaled by a factor of $4$ in the latter)  and set
$$r=|x|\,,\qquad v= 4|u|\,.$$
We define  for {\it all}  $v\in \bbR$,
\Be\label{defphi}
\phi(t,r,v)\,:=\,\psi(t,r)-tv\,=\,1-r^2t\cot t -t v,\,.
\Ee

Then from \eqref{psitalt}  and \eqref{psit}
\Be \label{phimualt}
\begin{aligned}
\phi_t(t,r,v)
&= r^2\Big(\frac{2t-\sin(2t)}{2\sin^2 t}\Big)-v
\\&=  \frac{r^2\mut}{\sin^2\mut}-\frac{1}{\mut} +
\frac{\phi(\mut,r,v)}{t} \,.
\end{aligned}
\Ee
Moreover $\phi_{tt}=\psi_{tt}$, and we will use  the formulas
\eqref{psitt} and \eqref{higherpsider} for  the derivatives of $\phi_t$.

\medskip


 \medskip


If we  set
\Be\label{rv(t)}\begin{gathered}
r(t)= \Big| \frac{\sin t}{t}\Big|\,,\quad v(t)= \frac 1{\mut} - \frac{\sin(2\mut)}{2\mut^2}
\\
r(0)=1\,,\quad v(0)= 0
\end{gathered}
\Ee
then we have
\begin{subequations}\label{phaseversuscurve}
\begin{align}
 \label{phitversuscurve}
\phi_t(t,r,v)&= \frac{v(t)}{r^2(t)} r^2 -v\,=
- \Big(v-v(t) -
\, v(t)\frac{r^2-r(t)^2}{r(t)^2}\Big)\,,
\\
 \label{phiversuscurve}
\phi(t,r,v)&= \frac{r(t)^2-r^2}{r(t)^2} + t\phi_t(t,r,v)\,.
\end{align}
\end{subequations}
Thus

\Be\label{system}
\phi(\mut,r,v)=\phi_\mut(\mut,r,v)=0 \quad \iff \quad (r,v)=(r(t),v(t))\,.
\Ee
Only the points $(r,v)$ for which there exists a $t$ satisfying \eqref{system} may contribute to the singular support $\Gamma$  of $e^{i\sqrt L} \de_0.$
One recognizes the result
by Nachman \cite{nachman} who showed for the Heisenberg group
 that  the singular support of the
 convolution kernel of $e^{i\sqrt L} $   consists of those $(x,u)$
for which there is a $t>0$ with $(|x|,4|u|)=(r(t),v(t))$.


\begin{figure}
\bigskip
\vspace{0.2cm}
\includegraphics[width=.8\textwidth]{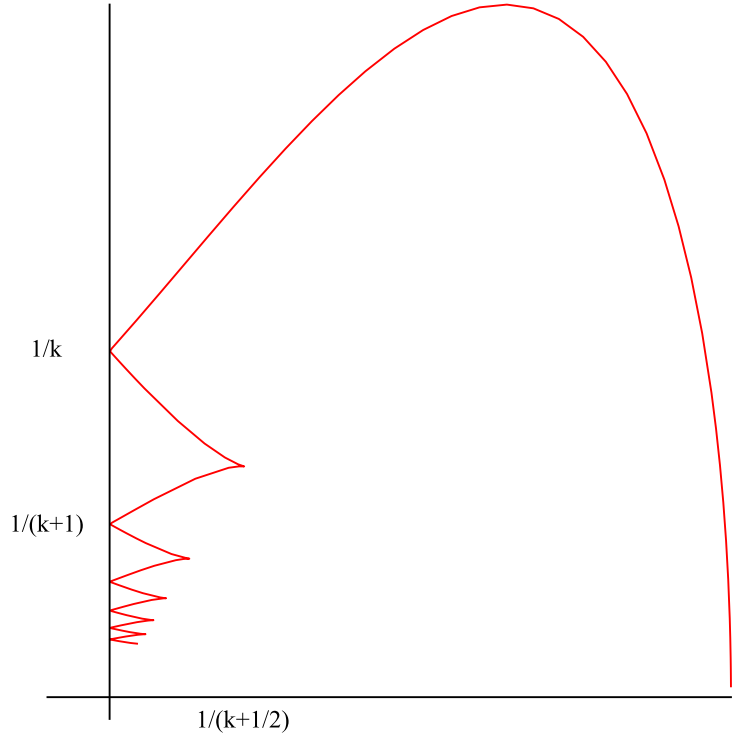}
\vspace{-0.2cm}
\medskip
\caption{$\{\pi(r(t),v(t)): t>0\}$}
\end{figure}

The figure pictures the singular support, including the contribution  near $|u|=0$ and $|x|$ near $1$. However we have taken care of the corresponding estimates in \S\ref{FIO}, and thus
we are only interested in the above formulas for $t>3\pi/8$.

For later reference we gather some  formulas and estimates for the derivatives
of $r(t)$ and $v(t)$.
For the vector of first derivatives we get, for $t\notin \pi\bbZ$,
\Be \label{rvprime}
\begin{pmatrix} r'(t)\\ v'(t)\end{pmatrix}=
\frac{\sin t-t\cos t}{t^2}\begin{pmatrix} -\sgn((\sin t)/t)\\ 2t^{-1}\cos t
\end{pmatrix}
\Ee
with $ r'(t)=O(t)$ and $v'(t)-\frac 23=O(t)$ as $t\to 0$.
Hence, for
$t\notin \pi\bbZ$,
\begin{equation} \label{slope}
\frac{v'(t)}{r'(t)}= -\sgn((\sin t)/t) \frac{2\cos t}t  = - 2r(t) \cot t\,.
\end{equation}
Clearly all derivatives of $t$ and $v$ extend to
functions continuous  at $t=0$.
Further computation yields   for positive  $t\notin \pi \ZZ$,
$\nu\ge1$,
\begin{subequations}\label{higherrvderform}
\begin{equation}
\sgn(\frac{\sin t}t)\,  r^{(\nu)}(t)=
 \sum_{n=1}^{\nu+1} a_{n,\nu} t^{-n} \sin t
+ \sum_{n=1}^\nu b_{n,\nu} t^{-n} \cos t
\end{equation} and
\begin{equation}
v^{(\nu)}(t)=\gamma_\nu t^{-\nu-1} +
\sum_{n=1}^{\nu+1} c_{n,\nu} t^{-n-1} \sin 2t
+ \sum_{n=1}^\nu d_{n,\nu} t^{-n-1} \cos 2t \,;
\end{equation}
\end{subequations}
here
$a_{n,\nu}=c_{n,\nu}=0$ if $n-\nu$ is even, and $b_{n,\nu}=d_{n,\nu}=0$
if $n-\nu$ is odd;
moreover $\gamma_\nu=(-1)^\nu(\nu-1)!$, and
  $a_{1,\nu}=(-1)^{\nu/2}$ for $\nu=2,4,\dots$.
For the coefficients in the  first derivatives formula   we get
$b_{1,1}=1$, $a_{2,1}=-1$, $d_{1,1}=-1$, and $c_{2,1}=1$.
For the second derivatives, we have the coefficients
$a_{1,2}=-1$,  $b_{2,2}=-2$, $a_{3,2}=2$,
$c_{1,2}=2$, $d_{2,2}=4$, $c_{3,2}=-3$.
Consequently, for the second derivatives we get the estimates
\Be\label{secondrvderbound}
|r''(t)|\lc t^{-1}|\sin t| + (1+t)^{-2}, \quad
|v''(t)|\lc t^{-2}|\sin 2t|+(1+t)^{-3}.
\Ee
Also, $|r^{(\nu)}(t)|\lc_\nu(1+t)^{-1}$, and $|v^{(\nu)}(t)|\lc_\nu (1+t)^{-2}$ for all $t>0$.


\section{$L^1$ estimates}\label{L1section}
In this section we prove the essential $L^1$ bounds needed for the
proof of Theorem \ref{main-theoremL1}. We may assume that $\la$ is large.

In what follows we frequently  need to  perform  repeated integrations by parts in the presence of oscillatory
terms with nonlinear phase functions and we start with a standard calculus lemma which will be used several times.

\subsection{\it Two preliminary lemmata}
Let $\eta\in C^\infty_0(\bbR^n)$ and let
$\Phi\in C^\infty$ so that $\nabla\Phi\neq 0$ in the support of $\eta$.
Then, after  repeated integration by parts,
\Be\label{repeatedinbyparts}
\int e^{i\la \Phi(y)} \eta(y) \,dy = (i/\la)^{N} \int e^{i\la
  \Phi(y)} \cL^N \!\eta(y) \, dy
\Ee where the operator $\cL$ is defined by
\Be \label{diffopdef}
\cL a
= \text{div} \big(\frac{a\nabla \Phi}{|\nabla\Phi|^2}\big).
\Ee
In order to analyze the behavior of $\cL^N$ we
shall need a lemma. We use multiindex notation, i.e. for $\beta=(\beta^1,\dots, \beta^n)\in (\bbN\cup\{0\})^n$
we write $\partial^\beta = \partial_{y_1}^{\beta^1}\cdots
\partial_{y_n}^{\beta^n}$ and let $|\beta|=\sum_{i=1}^n \beta^i$ be the order of the multiindex.

\begin{lemma} \label{iterateddiffop}
 Let $\cL$ be as in \eqref{diffopdef}. Then $\cL^N a$ is a linear
  combination of $C(N,n)$ terms of the form
$$\frac{\partial^\alpha a \prod_{\nu=1}^j \partial^{\beta_\nu} \Phi}
{|\nabla\Phi|^{4N}}$$ where $2N\le j\le 4N-1$ and
 $\alpha, \beta_1,\dots, \beta_{j}$ are multiindices  in $(\bbN\cup\{0\})^n$
with $1\le |\beta_\nu|\le |\beta_{\nu+1}|$,  satisfying
\begin{enumerate}
\item \ \ $0\le |\alpha|\le N$,
\item \ \ $|\beta_\nu|=1$ for $\nu=1,\dots, 2N$,
\item \ \ $|\alpha|+\sum_{\nu=1}^j|\beta_\nu|=4N$,
\item \ \ $\sum_{\nu=1}^j (|\beta_\nu|-1)=N-|\alpha|$.
\end{enumerate}
\end{lemma}
\begin{proof} Use induction on $N$. We omit the
straightforward details.\end{proof}

{\it Remark:}  {\it In dimension $n=1$  we see that $\cL^N a$ is a linear
  combination of $C(N,1)$ terms of the form
$$
\frac {a^{(\al)}}{(\Phi')^\al}\prod_{\beta\in \fI}\frac{\Phi^{(\beta)}}{(\Phi')^{\beta}},
$$
where $\fI$ is a set of integers $\beta\in\{2,\dots, N+1\}$ with the property that $\sum_{\beta\in \fI}(\beta-1)=N-\alpha$.
If $\fI$ is the empty set then we interpret the product as $1$. }

In what follows we shall often use the following
\begin{lemma}\label{integralobs}
Let $\Lambda>0$, $\rho>0$, $n\ge 1$ and $N>\frac{n+1}2$.
Then
\begin{equation*}\int_{-\infty}^\infty \frac{
(1+\La|v|)^{-\frac{n-1}{2}} |v|^{n-1}}
{(1+\La|\rho-v|)^{N}} dv
\lcs{n}
 \begin{cases} \La^{-\frac{n+1}{2}} \rho^{\frac{n-1}2}
&\text{ if } \La\rho\ge 1\,,
\\
\La^{-n} &\text{ if }
\La\rho\le 1\,.\qquad
\end{cases}
\end{equation*}
\end{lemma}
We omit the proof. Lemma \ref{integralobs}
 will  usually be
applied after using integration by parts with respect to the $s$-variable, with the
parameters $n=d_2$ and $\La=\la k$.

\subsection{\it Estimates for $|u|\lesssim (k+1)^{-1}\lambda^{-1}$}
We begin by proving an $L^1$ bound for   the part of the kernels
$K_{\la}^{k,l}$ for which the terms $\bess (4s \la  t|u| )$ have no
significant oscillation, i.e. for the region where $|u|\le C (\la
k)^{-1}$ (or $|u|\lc \la^{-1}$ if $k=0$).

\begin{lemma} \label{smallulemmakl}Let $\la\ge 1$,  $k\ge 1$, $l\ge 1$. Then
\Be \label{Kklsmallu}
\iint_{|u|\lc (\la k)^{-1}} |\la^{-\frac{d-1}{2}}K^{k,l}_\la(x,u)|
dx\, du
\lc
(2^l k)^{-1}\la^{1-\frac{d}{2}}.
\Ee
\end{lemma}

\begin{proof}
First we integrate the pointwise  bound
\eqref{Claklinfty}
over the region where
$|x|\le(\la  k 2^l)^{-1/2}$, $|u|\le (\la k)^{-1}$
and obtain
\begin{multline*}
\iint_{\substack{
|x|\le C (\la  k 2^l)^{-1/2}\\|u|\le C (\la k)^{-1}}}
|\la^{-\frac{d-1}{2}}K^{k,l}_\la(x,u)| dx\, du \\
\lc
2^{-l} \fC_{\la,k,l} (\la k 2^l)^{-d_1 /2} (\la k)^{-d_2}
=(2^l k)^{-1}\la^{1-\frac{d_1+d_2}{2}}.
\end{multline*}
If  $|x|\ge C (\la k 2^{l})^{-1/2}$ then from \eqref{psitalt}, \eqref{higherpsider} we get that $|\psi_t(t,|x|)|
  \gc 2^{2l} k  |x|^2$ on the support of $\eta_l(t-k\pi)$, moreover
$(\partial/\partial t)^{(n)} \psi(t,|x|)
=O(|x|^2 k 2^{l(n+1)})$.
The $n$th $t$-derivative of $\eta_l(t-k\pi) \bess(4s\la t|u|)$ is
$O(2^{ln})$. Thus an integration by parts gives
$$\la^{-\frac{d-1}{2}} |K^{k,l}_\la (x,u)|\le C_N
2^{-l}\fC_{\la,k,l} (\la 2^l k |x|^2)^{-N}
$$ for
$|x|\ge (\la k 2^{l})^{-1/2}$  and
$|u|\le (\la k)^{-1}$.
The bound
$O((2^l k)^{-1}\la^{1-\frac{d}{2}})$ follows by integration by parts.
\end{proof}

\medskip

\subsection{\it Estimates for  $|u|\gg  (k+1)^{-1}\lambda^{-1}$}

We now proceed to give $L^1$ estimates for the kernels $A^{k,l}_\la$
and $B^{k,l}_\la$ for $k\ge 1$, in the region where $|u|\gg (k\la)^{-1}$.

\subsubsection{An estimate for small $x$}
As a first application we prove  $L^1$ estimates  for
$|x|\lc (2^l \la k)^{-1/2}$,  $k\ge 1$.

\begin{lemma}\label{smallx}
Let $C\ge 1$. Then
\Be\label{smallxest}
\iint\limits_{\substack {(x,u):\\|x|\le C
(2^l\la k)^{-1/2}}}
\big[|A_{\la}^{k,l}(x,u)| + |B_{\la}^{k,l}(x,u)|\big]\, dx\, du
\lc\cin C
(2^l k)^{-1} \la^{-\frac {d_1-1}2 }.
\Ee
\end{lemma}
\begin{proof}
Integration by parts with respect to $s$ yields
\begin{multline}\label{intbyparts}
|A_{\la}^{k,l}(x,u)|+|B_{\la}^{k,l}(x,u)|
\\ \lc_N
\sum_\pm
\frac{\fC_{\la,k,l}}{(1+\la k|u|)^{\frac{d_2-1}{2}}}
\int_{|t-k\pi|\lc 2^{-l}}
(1+\la k\big| \pm |4u|-|x|^2 \cot \mut +\mut^{-1}\big|)^{-N} d\mut.
\end{multline}
We first integrate  in $u$.
Notice that by
Lemma \ref{integralobs}
we have  for
 fixed $t$ and fixed $r\le
(2^l\la k)^{-1/2}$
$$\int_0^\infty \frac{(1+\la kv)^{-\frac{d_2-1}{2}}v^{d_2-1}}
{(1+\la k\big| \pm |v|-r^2 \cot \mut +\mut^{-1}\big|)^{N}}
dv \lc \la^{-\frac{d_2+1}2}k^{-d_2}.
$$
We integrate in $x$ over a set of measure
$\lc (2^lk\la)^{-d_1/2}$  and then in $t$ (over an interval of
length $\approx 2^{-l}$) and \eqref{smallxest}  follows.
\end{proof}


\subsubsection{$L^1$-bounds for $B_\lambda^{k,l}$}

\begin{lemma}\label{Bkllemma}
For  $\la\ge 1$, $0<k\le 8\la$ ,
\Be \label{Bklestimate}
\big\|B_{\la}^{k,l}\big\|_1 \lc
(2^l k)^{-1} \la^{-\frac {d_1-1}2 } .\Ee
 \end{lemma}

\begin{proof}
The bound for the region
with  $|x|\lc(2^l\la k)^{-1/2}$
(for which there is no significant oscillation in the $t$ integral) is
proved in Lemma \ref{smallx}.

Consider the region where $|x|\approx 2^m (2^l\la k)^{-1/2}.$
We perform $N_1$ integration by parts in $t$
followed by $N_2$ integrations by parts with respect to $s$.
Denote by $\cL_t$ the operator defined by
 $\cL_t g=\partial_t (\tfrac{g(t)}{\psi_t(t,|x|)+4|u|})$. Then
\begin{multline*}B^{k,l}_\la(x,u)=\fC_{\la,k,l}(i/\la)^{N_1} \times
\\
\iint
 e^{i\la s(\psi(t, |x|)+4t|u|)}  \frac{(I-\partial_s^2)^{N_2}\big[
s^{-N_1} \cL_t^{N_1}\{  \eta_{\la,k,l}(s,t)
\vpi_2(4\la st|u|)\,\}\big]}
{(1+ \la^2|\psi(t,|x|)+4t|u||^2)^{N_2}}
  dt \,ds
\end{multline*}

From \eqref{psitalt},
$$|\partial_t (\psi(t,|x|)+4t|u|)| \gc 2^{2l}k|x|^2+4|u| \gc 2^{2m+l}\la^{-1}.$$
Moreover, for $\nu\ge 2$,  $\partial_t^\nu \psi = O(2^{2m+l\nu}\la^{-1})$ and
$\nu$ differentiations of the amplitude
produce  factors of $2^{l\nu}$.
Thus
we obtain the bound
\begin{multline*}|B_{\la}^{k,l}(x,u)|\lc
\frac{\fC_{\la,l,k}}
{(1+4\la k|u|)^{\frac{d_2-1}{2}}}
2^{-2m N_1}\quad \times\\
\int_{|t-k\pi|\lc 2^{-l}}
(1+\la k\big||t^{-1}-|x|^2 \cot t +4|u||)^{-2N_2} \,dt.
\end{multline*}
From Lemma \ref{integralobs} (with $n=d_2$, $\La=\la k$, $\rho\lc k^{-1}
\max\{1, 2^{2m}\la^{-1}\}$)
\begin{multline}\label{vintegralfixedt}
\int_{v=0}^\infty\frac{(1+\la kv)^{-\frac{d_2-1}{2}}v^{d_2-1}}
{(1+\la k\big|v-|x|^2 \cot t +\mut^{-1}\big|)^{N}}
dv\\
\lc \la^{-\frac{d_2+1}{2}} k^{-d_2} \max\{ 1,
(2^{2m}\la^{-1})^{\frac{d_2-1}{2}}\}.
\end{multline}
We integrate in $t$ over an interval of length $O(2^{-l})$ and in $x$
over the annulus $\{x: |x|\approx 2^m (2^l\la k)^{-1/2}\}$.
This gives
\begin{align}\label{Bklinannulus}
&\iint_{\substack{(x,u):\\
|x|\approx 2^{m}(2^l\la k)^{-1/2} }}
|B_{\la}^{k,l}(x,u)| dx\, du\,
\\& \lc \, 2^{-2m N}  \,2^{-l} \Big(\frac{2^{m}}{\sqrt{ 2^l \la k}}\Big)^{d_1}
\fC_{\la,k,l}
\la^{-\frac{d_2+1}{2}} k^{-d_2} \max\{ 1,
(2^{2m}\la^{-1})^{\frac{d_2-1}{2}}\}
\notag
\\
&\lc
(2^{l}k)^{-1} \la^{-\frac {d_1-1}2 } 2^{-m(2N-d_1)}\max\{ 1,
(2^{2m}\la^{-1})^{\frac{d_2-1}{2}}\}
\notag
\end{align}
and choosing $N$ sufficiently large the lemma follows by summation in $m$.
\end{proof}

\subsubsection{$L^1$-bounds for
$A^{k,l}_\la$, $2^lk\ge10^5\la$}
\begin{lemma}\label{kgrlalemma}
For $k\le 8\la$, $2^l\ge 10^5\la/k$,
\Be \label{kgrla}
\big\|A_{\la}^{k,l}\big\|_1 \lc
(2^l k)^{-1} \la^{-\frac {d_1-1}2 } .\Ee
 \end{lemma}
\begin{proof}
We use  Lemma \ref{smallx} to obtain  the appropriate $L^1$ bound in the  region
$\{(x,u): \,|x|\le C_0
(2^l\la k)^{-1/2}\}$. Next,  consider the region where
\Be \label{gain2m}
 2^m (2^l\la k)^{-1/2}\le |x|\le  2^{m+1} (2^l\la k)^{-1/2}
\Ee for large $m$.
This region is then  split into two subregions,
one where $4|u|=v\le 10^{-2} 2^{2m+l}\la^{-1}$
and the complementary region.

For  the region with small $v$ we  proceed as in Lemma \ref{Bkllemma}.
From formula \eqref{psitalt} we have $|\psi_t|\ge kr^22^{2l}/20$
and hence $|\psi_t|\ge 2^{2m+l-5}\la^{-1}$.
Thus if   $v\le 10^{-2}2^{2m+l}\la^{-1}$
then
$|\phi_{\mut}|\approx k2^{2l} r^2  \approx 2^{2m+l}\la^{-1}$.
Moreover  $\partial_t^\nu \phi = O(
2^{2m+l\nu}\la^{-1})$ for $\nu\ge 2$.
Therefore,  if we perform integration  by parts in $\mut$
several times, followed by integrations by parts on $s$,
we obtain the bound
\begin{multline*}|A_{\la}^{k,l}(x,u)|\lc
\frac{\fC_{\la,l,k}}
{(1+\la k|u|)^{\frac{d_2-1}{2}}}
2^{-2m N}\quad \times\\
\int_{|t-k\pi|\lc 2^{-l}}
(1+\la k\big||x|^2 \cot \mut -\mut^{-1}-4|u|\big|)^{-N} d\mut.
\end{multline*}
In the present range $|x|^2|\cot t|\approx
2^{2m}(\la k)^{-1}$ and $t^{-1}\approx k^{-1}$ and thus
we see from Lemma \ref{integralobs} that inequality \eqref{vintegralfixedt}
in the proof of Lemma \ref{Bkllemma} holds. From this we proceed as in
\eqref{Bklinannulus} to bound
\begin{multline*}\iint_{\substack{
|x|\approx 2^{m}(2^l\la k)^{-1/2}
\\
4|u|\le 10^{-2} 2^{2m+l} \la^{-1} }}
|A_{\la}^{k,l}(x,u)| dx du\,\\
\lc
(2^{l}k)^{-1}\la^{-\frac {d_1-1}2 }  2^{-m(2N-d_1)}\max\{ 1,
(2^{2m}\la^{-1})^{\frac{d_2-1}{2}}\} \,.
\end{multline*}
For large  $N_1$  we can sum in $m$ and obtain the
bound
$C(2^{l}k)^{-1}\la^{-\frac {d_1-1}2 }$.


Next assume that
$v\ge 2^{2m+l}\la^{-1}/100$
(and still keep  \eqref{gain2m}).
Then
\Be \label{philowerbdlargev}
|tv+r^2 t\cot t-1|\ge k|v|
\text{ for $t\in \supp( \eta_{\la,k,l})$}\,.\Ee
Indeed, we have $tv\ge 2^{2m} 2^lk\la^{-1}/100 \ge 10^3$
and
$$r^2t|\cot t|\le 2^{2m+2}(2^l\la k)^{-1}
t[\sin(\tfrac{3\pi}{8}2^{-l})]^{-1}  \le
 2^{2m+6}\la ^{-1}\le \frac{2^{2m+l}}{100\la} 2^5 10^2 2^{-l}
$$
where we used \eqref{gain2m} and $\sin \alpha>2\alpha/\pi$ for $0\le \alpha\le \pi/2$.
By  our  assumptions  $2^l\ge 10^5 \la/k>10^4$ and thus the right hand side of the display is $\le v/10$.
Now
\eqref{philowerbdlargev}
is immediate by the triangle inequality.

We use \eqref{philowerbdlargev}  to get from an  $N_1$-fold integration by
parts in $s$
$$|A^{k,l}_\la(x,u)|
\lc 2^{-l}\fC_{\la,l,k}
(\la kv)^{-N_1-\frac{d_2-1}{2}}.
$$
Then
\begin{align*}
&\iint_{\substack{|x|\approx 2^m (2^l \la k)^{-1/2}\\
4|u|\ge 10^{-2} 2^{2m+l}\la^{-1} }}|A^{k,l}_\la(x,u)| \,du \,dx
\\
&\quad \lc 2^{-l} \fC_{\la,l,k} \Big(\frac{2^m}{\sqrt{\la 2^l k}}\Big)^{d_1}
(\la k)^{-N_1-\frac{d_2-1}{2}} \Big(\frac{2^{2m+l}}{\la}\Big)^{-N_1+ \frac{d_2+1}{2}}
\\
&\quad \lc
\la^{1-\frac{d_1}2-\frac{d_2}{2}}2^{-l(N_1-\frac{d_2-1}{2})}k^{\frac{d_2-1}{2}-N_1}
2^{m(d_1+d_2+1-2N_1)}\,.
\end{align*}
For $N_1$ large we may sum in $m$ to finish the proof.
\end{proof}

\subsubsection{Estimates for  $A_\la^{k,l}$, $2^l\lc\la/k$.}
In the early   approaches to prove $L^p$ boundedness for Fourier integral operators the oscillatory integral were analyzed  using the   method of stationary
phase (\cite{peral}, \cite{miyachi}, \cite{beals}).
This creates some difficulties in our case at points where $\phi$, $\phi_t$ and  $\phi_{tt}$ vanish simultaneously, namely at positive $t$ satisfying
 $\tan t=t$.
 To avoid this difficulty we use a decomposition in the spirit of \cite{SSS}.

In what follows we assume
$k\le 8\la$ and $2^l\le C_0\la/k$ for large $C_0$ chosen independently of
$\la, k,l$. The choice  $C_0=10^{10}$ is suitable. We decompose the interval
$J_{k,l}$ into smaller subintervals of length $\eps \sqrt{\frac {k}{2^l\la}}$ (which is $\lc 2^{-l}$ in the range under consideration), here  $\eps \ll 10 ^{-100}$ (to be chosen
sufficiently small but  independent of $\la, k,l$).

This decomposition is motivated by the following considerations: according to \eqref{splitting},
$\la \phi(t,r,v)$ contains the term  $-\la(r-r(t))^2 t\cot t$ depending entirely on $r$ and $t$. For $t\in J_{k,l},$ this is of size $\la k 2^l |r-r(t)|^2,$ hence  of order $O(1)$ if $|r-r(t)|\lesssim (\la k 2^l)^{-1/2}.$
Moreover, on a subinterval $I$ of $J_{k,l}$ on which $r(t)$ varies  by at most  a small fraction of the  same size, the term
 $-\la(r-r(t))^2 t\cot t$ is still $O(1)$ and contributes to no oscillation in the integration with respect to $s.$ Since $|r'(t)|\sim 1/k$ by  $\eqref{rvprime}, $  this suggests to choose intervals $I$ of length $\ll k(\la k 2^l)^{-1/2}=\sqrt{k2^{-l}\la^{-1}}.$ Similarly, the first term  of $\la \phi(t,r,v)$ in \eqref{splitting} is of size $\la k|w(t,r,v)|$ and does not contribute to any oscillation in the integration with respect to $s$  if $|w(t,r,v)|\lesssim (\la k)^{-1}.$  These considerations also motivate our later definitions of the  set $\cP_0$ and the sets $\cP_m, m\ge 1$, \cf. \eqref{cPm}.

\medskip

As before we denote by $\eta_0$ a $C^\infty_0(\bbR)$ function  so that $\sum_{n\in\bbZ} \eta_0(t-\pi n)=1$ and $\supp(\eta_0)\subset (-\pi,\pi)$.
 Define, for
$b\in \pi\eps\sqrt{k2^{-l}\la^{-1}}\,\bbZ$,
\Be\label{defetaklb} \eta_{\la,k,l,b}(s,t)=  \eta_{\la, k,l}(s,t)
\eta_0 \big(\eps^{-1}\sqrt{\tfrac{\la 2^l}{k}}(t-b)\big).
\Ee
Then we may split
\Be\label{bdecomp}A_{\la}^{k,l}=\sum_{b\in \cT_{\la,k,l}} A_{\la,  b}^{k,l}
\Ee
where
 $\cT_{\la,k,l}\subset
\pi\eps\sqrt{ k2^{-l}\la^{-1}}\,
 \bbZ \cap J_{k,l}$ (\cf.\eqref{Jkl}),
 $\#\cT_{\la,k,l}=O(\eps^{-1}\sqrt{\la 2^{-l} k^{-1}}  )$, and
\begin{multline}\label{Blakb}
A_{\la,b}^{k,l}(x,u)=
\\
\fC_{\la,l,k}
\iint \chi(s)
\eta_{\la,k,l,b}(t)
e^{i\la s(1-|x|^2\mut\cot \mut -\mut|4u|)} \vpi_1(\la s\mut |4u|) d\mut ds\,.
\end{multline}

\medskip

We now give some formulas relating the phase
$\phi(t,r,v)=1-r^2\mut\cot \mut -\mut v$
 to the geometry of the curve $(r(t), v(t))$ (\cf.\eqref{rv(t)}).
By \eqref{system} and \eqref{slope},
\begin{align}
&\frac{\phi(t,r,v)}{t}=
\frac{\phi(t,r,v)-\phi(t, r(t),v(t))}{t}
\notag
\\
&=(r(t)^2-r^2) \cot t +v(t)-v
\notag
\\
&= v(t)-v  -\big(r-r(t)\big) 2 r(t) \cot t
-(r-r(t))^2 \cot t
\notag
\end{align}
and, setting
\Be\label{wdefin}w(t,r,v)= v-v(t)- \frac{v'(t)}{r'(t)}(r-r(t))\,,\Ee
we get
\Be
\label{splitting}
\frac{\phi(t,r,v)}{t}= - w(t,r,v)-(r-r(t))^2 \cot t\,.
\Ee

Moreover,
\begin{align} \notag
\phi_t(t,r,v)&=
\frac{\phi(t,r,v)}{t}+
\frac{r^2 t}{\sin^2 t}- \frac 1t
\\&=\frac{\phi(t,r,v)}{t}+\frac{t}{\sin^2 t} (r+r(t))(r-r(t))
\label{phit-splitting}
\end{align}

We shall need  estimates describing  how $w(t,r,v)$ changes in $t$.
Use
\eqref{splitting} and  the expansion
\begin{align*}
&w(t,r,v)- w(b,r,v)\,=\,
- \big[ v(t)-v(b)- \frac{v'(b)}{r'(b)} (r(t)-r(b))\big]
\\& \qquad - \Big[ \frac{v'(t)}{r'(t)}
- \frac{v'(b)}{r'(b)}\Big](r-r(b))
+ \Big[ \frac{v'(t)}{r'(t)} - \frac{v'(b)}{r'(b)}\Big](r(t)-r(b)).
\notag
\end{align*}
From
\eqref{secondrvderbound}
 we get $|r''|+k|v''|\lc 2^{-l}k^{-1}+k^{-2}$ on $J_{k,l}$, thus the first term
in the displayed formula is  $\lc (2^{-l}k^{-2}+k^{-3} )|t-b|^2$.
Differentiating in \eqref{slope}  we also get
$(v'/r')'=
  O(2^{-l}k+ k^{-2})$ on $J_{k,l}$,
and  see that the second term in the display is
$\lc (2^{-l}k^{-1}+k^{-2})
|t-b| |r-r(b)|$ and the third is
$\lc ( 2^{-l}+k^{-1})k^{-2}(t-b)^2$.
Hence
\Be\label{wdiff2}
|w(t,r,v)- w(b,r,v)|\lc
(2^{-l}+k^{-1})|t-b|\Big(\frac{|t-b|}{k^{2}}+ \frac{|r-r(b)|}{k}\Big)\,.
\Ee

We now turn to the estimation of $A^{k,l}_{\la,b}$ with $k\ge 1$ and
$b\in \cT_{\la,k,l}$.
Let, for $b>1/2$, $l=1,2,\dots$, and  $m=0,1,2,\dots$,
\begin{multline}\label{cPm}
\cP_m\equiv\cP_m(\la,l,k;b):=\big\{(r,v)\in (0,\infty)\times (0,\infty):\\
\quad
v\ge (\la k)^{-1},\,\,
 |r-r(b)|\le 2^m (\la k 2^l)^{-1/2}, \,\, |w(b,r,v)|\le 2^{2m}
(\la k)^{-1}  \big\}
\end{multline}
and let
\Be \label{Omm}
\Om_m\equiv\Om_m(\la,l,k;b):=
\begin{cases}
&\{(x,u):\, (|x|, 4|u|)\in \cP_0\}\,,\text{ if } m=0,
\\
&\{(x,u):\, (|x|, 4|u|)\in \cP_m\setminus \cP_{m-1}\} \text{ if } m>0\,.
\end{cases}
\Ee

For later reference we note that in view
$2^l \le \la/k$, $|t-b|\le \eps \sqrt{ \frac{k}{\la 2^l}}$
and the upper bound $|r'(t)|\le 2t^{-1}$ we have
$r(t)-r(b)= O\big(\frac{\eps}{\sqrt{k\la 2^l}}\big),$
and, by \eqref{wdiff2},
\Be\label{wdiffestimate}
|w(t,r,v)- w(b,r,v)|\lc \eps 2^m (\la k)^{-1},\quad (r,v)\in \cP_m.
\Ee
Moreover it is easy to check  that, still for
$|t-b|\le \eps \sqrt{ \frac{k}{\la 2^l}}$,
\Be\label{sectermdiff}
|(r-r(t))^2\cot t- (r-r(b))^2\cot b|
\lc \eps 2^{2m} (\la k)^{-1}.\Ee

\medskip

\begin{prop} \label{vert-hor}
Assume that  $1\le k\le8\la$, $l=1,2,\dots$, and  $2^l\le C_0 \la/k$
(and let $\eps$ in the definition \eqref{defetaklb}
be $\le C_0^{-1} 10^{-100}$).
Let $b\ge 1$ and $b\in \cT_{\la,k,l}$.
Then
\begin{align} \label{corridorzero}
&\iint_{\Om_0(\la, l,k;b)} |A^{k,l}_{\la,b}(x,u)|
\, dx du \, \lc\, (2^lk)^{-\frac{d_1+ 1}2}
\sqrt{\tfrac{2^lk}{\la}}
\\ \label{corridorm}
&\iint_{\Om_{m}(\la, l,k;b)} |A^{k,l}_{\la,b}(x,u)|  \, dx du \,
\lcs{N}\, 2^{-m N} (2^lk)^{-\frac{d_1+ 1}2}
 \sqrt{\tfrac{2^lk}{\la}}\,.
\end{align}
\end{prop}

\begin{proof}
Note that, for fixed $k\ge 1$, $l\ge 1$, $b\in \cT_{\la,k,l}$,
\Be \label{rvinPm}
(r,v)\in \cP_m \,\implies r\lc 2^m (2^l k)^{-1} \text {and } v\lc
 2^{2m}k^{-1}\,.
\Ee
This is immediate in view of $2^lk\lc\la$, $r(b)\approx (2^{l}k)^{-1}$,
 $v(b)\approx k^{-1}$
and thus
\Be\label{rvsizebounds}
\begin{aligned}&r\lc (2^l k)^{-1}(1+ 2^m \sqrt{\tfrac{k2^l}{\la}})\,
\lc 2^m (2^l k)^{-1}\,,
\\&v\lc k^{-1}(1+2^{2m}\la^{-1})\,\lc 2^{2m}k^{-1}\,.
\end{aligned}
\Ee
Also recall  that  $v=4|u|\ge (\la k )^{-1}$ for $(x,u)\in \Omega_m(\la,l,k;b)$.

A  crude size estimate yields
\Be\label{basicsizeest}
\iint_{(|x|,4|u|)\in \cP_m}  |A^{k,l}_{\la,b}(x,u)|
\, dx \,du \, \lc\,
2^{m(d_1+d_2+1)} (2^{l}k )^{-(d_1+1)/2} \sqrt {\tfrac{2^l k}{\la}}.
\Ee
Indeed, the left hand side is $\lc \eps \sqrt{\tfrac{k}{2^l\la}}\,\fC_{\la,k,l}  \,\cI$ where
$$\cI:=  \iint_{\substack{|r-r(b)|\lc 2^m(2^l \la k)^{-1/2}\\|w(b,r,v)|\lc 2^{2m}(\la k)^{-1}}}
(\la k v)^{-\frac{d_2-1}{2}}
v^{d_2-1} r^{d_1-1} \,dvdr
$$ is
$\lc
 \frac{2^m}{\sqrt{\la 2^l k}}
\big(\frac{2^m}{2^l k}\big)^{d_1-1}
\frac{2^{2m}}{\la k} \big(\frac{2^{2m} k^{-1}}{\la k}\big)^{\frac{d_2-1}{2}},$ in view of \eqref{wdefin} and \eqref{rvsizebounds}.
This yields \eqref{basicsizeest}.
In regard to its dependence on $m$ this bound is nonoptimal and
will be used for  $2^m\le C(\eps)$.


\medskip

We now derive an  improved $L^1$ bound for the region $\Om_m$ when $m$ is large.
For  $(r,v)\in \cP_m\setminus \cP_{m-1}$ we
distinguish two cases $I$, $II$ depending on the size of
$|\phi(b,r,v)|$ and define for $m>0$, and fixed $k,l,b$,
$$\begin{aligned}
\cR_m^I&= \{(r,v)\in \cP_m\setminus \cP_{m-1}:
|\phi(b,r,v)| >  2^{l-100}(r-r(b))^2\,\}\,,
\\
\cR_m^{II}&= \{(r,v)\in \cP_m\setminus \cP_{m-1}:
|\phi(b,r,v)| \le  2^{l-100}(r-r(b))^2\,\}\,.
\end{aligned}
$$
We also have the corresponding decomposition  $\Om_m=\Om_m^I+\Om_m^{II}$ where $\Omega_m^I$
and $\Omega_m^{II}$
consist of those $(x,u)$ with $(|x|, 4|u|)\in \cR_m^I$
and  $(|x|, 4|u|)\in \cR_m^{II}$,
respectively.

\medskip

\noindent{\it Case I:  $|\phi(b,r,v)| \ge 2^{l-100}k(r-r(b))^2$.}
We shall show that
\Be\label{philowbd}
|\phi(t,r,v)| \gc c 2^{2m}\la^{-1},\quad \text{ for }(r,v)\in
\cR_m^{I},
\quad|t-b|\le \eps
\sqrt{\tfrac{k}{2^l\la}}\,.
\Ee
with $c>0$ if $0<\eps\ll 10^{-100}$ is chosen sufficiently small. Given
\eqref{philowbd} we can use  an $N_2$-fold integration by parts in $s$ to  obtain a gain of
$2^{-2m N_2}$ over the above straightforward size estimate \eqref{basicsizeest}, which
leads to
\Be\label{OmIest}
\iint_{\Om_m^{I}}  |A^{k,l}_{\la,b}(x,u)|
\, dx \,du \, \lc_{\eps,N_2}
2^{m(d_1+d_2+1-2N_2)} (2^lk)^{-\frac{d_1+1}2}
\sqrt{\tfrac{2^lk}{\la}}\,.
\Ee

It remains to show \eqref{philowbd}.
We distinguish between two subcases. First
 if  $|r-r(b)|\ge 2^{m-5}(\la k2^l)^{-1/2}$ then by the {\it Case I}
 assumption
we have
$|\phi(b,r,v)|\ge 2^{l-100}k  2^{2m-10} (\la k2^l)^{-1} = 2^{2m-110}\la^{-1},$
 and by  \eqref{splitting},
\eqref{wdiffestimate}  and \eqref{sectermdiff} we also get
\eqref{philowbd} provided that $\eps\ll 2^{-200}$.

For the second subcase we have $|r-r(b)|\le 2^{m-5}(\la k2^l)^{-1/2}$.
Since $(r,v)\notin \cP_{m-1}$ this implies that
$|w(b,r,v)|\ge 2^{2m-2}(\la k)^{-1}$, and since the quantity
$b(r-r(b))^2 |\cot b|$ is bounded by $2^{l+4}b(r-r(b))^2 \le 2^{2m-6} (b/k)\la^{-1}$ we also get
$|\phi(b,r,v)|\ge 2^{2m-3}\la^{-1}$, by \eqref{splitting}.
Now by \eqref{splitting}, \eqref{wdiffestimate} and \eqref{sectermdiff} we also get
$|\phi(t,r,v)|\ge 2^{2m-4}\la^{-1}$, if $\eps$ is sufficiently small.
Thus
\eqref{philowbd} is verified and
\eqref{OmIest} is proved.

\medskip

\noindent{\it Case II:  $|\phi(b,r,v)| \le 2^{l-100}k(r-r(b))^2
$.}
We show
\begin{multline} \label{phiderivlowerbound}
|\phi_t(t,r,v)|\ge  2^{m-20} 2^{3l/2} k^{1/2} (r+r(b))\la^{-1/2}\\
\text{ if } (r,v)\in \cR^{II}_m\,,\quad
|t-b|\le \eps\sqrt{\tfrac{\la 2^l}{k}}\,.
\end{multline}
and this will enable us to get a gain when integrating by parts in $t$.
To prove \eqref{phiderivlowerbound} we first establish
\Be\label{rminusrblwbd}
|r-r(b)|\ge 2^{m-10} (\la k2^l)^{-1/2}\,\text{ for }
(r,v)\in \cR^{II}_m\,.
\Ee
Note that if $|w(b,r,v)|\le 2^{2m-3}(\la k)^{-1}$ then
$|r-r(b)|\ge 2^{m-1}
(\la k2^l)^{-1/2}$ since $\cR_m^{II}\subset \cP_{m-1}^\complement$. Thus to verify
 \eqref{rminusrblwbd} we may assume
$|w(b,r,v)|\ge 2^{2m-3}(\la k)^{-1}$. In this case we get from
  \eqref{splitting}, $(r,v)\in\cP_m$ and the {\it Case II} assumption
\begin{align*}&(r-r(b))^2 |\cot b|\,\ge\,
|w(b,r,v)|- b^{-1}|\phi(b,r,v)|\\
&\,\ge\, 2^{2m-3}(\la k)^{-1}- b^{-1} k2^{l-100}2^{2m}(\la k2^{l})^{-1}\,\ge\, 2^{2m-4}(\la k)^{-1}
\end{align*}
and hence $(r-r(b))^2 2^{l+4} \ge 2^{2m-4}(\la k)^{-1}$ which implies
\eqref{rminusrblwbd}.
In order to prove
\eqref{phiderivlowerbound} we use \eqref{phit-splitting} and
\eqref{rminusrblwbd} to estimate
\begin{align*}
|\phi_t(b,r,v)|&\ge \frac{b}{\sin^2 b} (r+r(b))|r-r(b)|-2^{l-100} \frac kb
(r-r(b)^2
\\&\ge \frac{|r-r(b)|}{b} \Big(\frac{r+r(b)}{r(b)^2}- \frac{2^l k}{2^{100}}|r-r(b)|\Big)\ge \frac{(r+r(b))|r-r(b)|}{2b r(b)^2}
\\&\ge2^{2l-4} k(r+r(b)) \frac{2^{m-10}}{\sqrt{\la k2^l}}
\ge 2^{m-15} k^{1/2} 2^{3l/2} (r+r(b)) \la^{-1/2}
\end{align*}
which yields \eqref{phiderivlowerbound} for $t=b$.
We need  to show the lower bound for
 $|t-b| \le \eps \sqrt{k/(2^l\la)}$.
By \eqref{psitt} we have $|\phi_{tt}(t',r,v)|\le r^2 b 2^{3l+4}$ for
$|t'-b| \le \eps \sqrt{\frac{b}{2^l\la}}$ and thus
$$|\phi_t(t,r,v)-\phi_t(b,r,v)|\le 2^6  r^2 2^{3l} k\eps
\sqrt{\tfrac{k}{2^l\la}} \le 2^{m-30} 2^{3l/2} k^{1/2}\la^{-1/2}(r+r(b))
$$
if $\eps \ll 2^{-100}$.
The second inequality in the last display is easy to check. If $r\le 2r(b)$
then  use $r\lc (2^l k)^{-1}\approx r+r(b)$ and if $r>2r(b)$ then use
$r-r(b)\approx r+r(b)\approx r$. In both cases  the asserted inequality holds for small $\eps$ and thus \eqref{phiderivlowerbound} holds for
$|t-b| \le \eps \sqrt{k/(2^l\la)}$.
We note that under the condition \eqref{rminusrblwbd}
the range $r\le 2 r(b)$ corresponds to $2^m \lc \sqrt{\la(2^l k)^{-1}}$
and the range $r\ge 2 r(b)$ corresponds to $2^m \gc \sqrt{\la(2^l k)^{-1}}$.

We now estimate the $L^1$ norm over the region
where $(r,v)\in \cR_m^{II}$.
Let $\cL_t$ be the differential operator defined by $\cL_t g= \frac{\partial}{\partial t}( \frac{g}{\phi_t})$.
By $N_1$  integration by parts in $t$ we get
(with $|x|=r$, $4|u|=v$)
\begin{multline*}
A_{\la,b}^{k,l} (x,u) \,= \,
i^{N_1}\la^{-N_1}\fC_{\la,k,l} \,\times \\
\iint
e^{i\la s\phi(t,|x|,4|u|)}
s^{-N_1}\cL_t^{N_1} [\eta_{\la,k,l,b}(s,t) \vpi_1(\la stv)]
\,  dt \,ds\,.
\end{multline*}
To estimate the integrand use the lower bound on $|\phi_t|$,
\eqref{phiderivlowerbound}.
Moreover  we have the upper bounds
\eqref{higherpsider}  for the higher derivatives of $\psi$  (and then $\phi$)
which give $\partial_t^n\phi=O(2^{l(n+1)}br^2)$ for $n\ge 2$.
Each differentiation of the cutoff function produces a factor of
$(\la 2^l k^{-1})^{1/2}$.
By the one-dimensional version of Lemma \ref{iterateddiffop}  described in the subsequent remark the expression
$\la^{-N_1}(\la bv)^{(d_2-1)/2} |\cL_t^{N_1}
[\eta_{\la,k,l,b}(s,t) \vpi_1(\la st v)]|$  can be estimated by a sum of $C(N_1)$ terms of the form
\Be \label{factorsintbyparts}
\la^{-N_1}\frac{ (\la 2^l/k)^{\alpha/2}}
{(2^m2^{3l/2}k^{1/2}(r+r(b))\la^{-1/2})^\alpha}
\prod_{\beta \in \fI} \frac{2^{l(\beta+1)}kr^2}{(2^m 2^{3l/2}k^{1/2} (r+r(b))\la^{-1/2})^\beta}
\Ee
where $\alpha\in \{0,\dots, N_1\}$, $\fI$ is a set of integers $\beta\in \{2,\dots, N_1+1\}$
 with the property that $\sum_{\beta\in \fI}(\beta-1)=N_1-\alpha$.
If $\fI$ is the empty set then we interpret the product as $1$.
We observe that for $(r,v)\in \cR_m^{II}$ we have $|r-r(b)|\approx 2^m(\la k 2^l)^{-1/2}$.  Thus if $2^m \le \sqrt{\la(2^l k)^{-1}}$
we have
 $r\lc (2^lk)^{-1}$ and   $r+r(b)\approx (2^lk)^{-1}$  while for
$2^m > \sqrt{\la(2^l k)^{-1}}$ we have
 $r\approx r-r(b) \approx r+r(b)\approx 2^m (\la k2^l)^{-1/2}$.

A short computation  which uses these observations shows that in the case
 $2^m \le \sqrt{\la(2^l k)^{-1}}$
the terms
\eqref{factorsintbyparts}  are $\lc
2^{-m\alpha} \prod_{\beta\in \fI} \big[ 2^{-m\beta} (2^lk/\la)^{\beta/2-1}\big]$.
In the case $2^m > \sqrt{\la(2^l k)^{-1}}$
the terms
\eqref{factorsintbyparts}  are dominated by a constant times
$ (\la2^{-l} k^{-1})^{\alpha/2}2^{-2m\alpha} \prod_{\beta\in \fI}2^{-m(\beta-1)}$.
In either case the terms \eqref{factorsintbyparts}  are $\lc 2^{-mN_1}$
(since $\alpha+\sum_{\beta\in\fI}\beta\ge N_1$).
This means that  we gain a factor of $2^{-mN_1}$ over the size estimate
\eqref{basicsizeest}.
Consequently,
\Be\label{OmIIest}
\iint_{\Omega^{II}_m}
|A^{k,l}_{\la,b}(x,u)| \, dx \,du \, \lc 2^{m(d_1+d_2+1-N_1)} (2^lk)^{-\frac{d_1+1}2}
\sqrt{\tfrac{2^lk}{\la}}\,.
\Ee
The  assertion of the proposition follows then from \eqref{OmIest} and
\eqref{OmIIest}.
\end{proof}

\subsection{\it $L^1$ estimates for $T^k_\la$ and $W_{j,n}$}

\begin{proof} [Proof of  \eqref{TlakL1}]
   Let us recall that $k\le 8\la.$
If we sum the bounds in Proposition
\ref{vert-hor}  in $b\in \cT_{2^j,k,l}$ we get
$$\|A^{k,l}_{2^j}\|_{L^1} \lc (2^l k)^{-\frac{d_1+1}{2}},
\quad 2^l\lc \frac{2^{j}}k\,.$$

We also have
\Be
\label{KklminusAkl}
\|2^{-j\frac{d-1}{2}}K^{k,l}_{2^j}-A^{k,l}_{2^j}\|_{1}
\lc (2^l k)^{-1} 2^{-j\frac{d_1-1}{2}};
\Ee
for the part of $ K^{k,l}_{2^j}$ where $|u|\lesssim 1/{k\la}$ this follows from Lemma \ref{smallulemmakl}, and for the remaining part from Lemma \ref{Bkllemma}.
Combining these two estimates, we find that
\Be
\label{Kkllsmall}
\|2^{-j\frac{d-1}{2}}K^{k,l}_{2^j}\|_{1}
\lc (2^l k)^{-\frac{d_1+1}{2}}, \quad 2^l\lesssim  \frac{2^{j}}k.
\Ee
Moreover, by
Lemma \ref{Bkllemma} and Lemma \ref{kgrlalemma}, we have
\Be
\label{Kkllarge}
\|2^{-j\frac{d-1}{2}}K^{k,l}_{2^j}\|_{1}
\lc (2^l k)^{-1} 2^{-j\frac{d_1-1}{2}}, \quad 2^l\ge 10^6  \frac{2^{j}}k.
\Ee
Altogether this leads to
\Be
\label{Tlaklbounds}
2^{-j(d-1)/2}\|T_{2^j}^{k,l}\|_{L^1\to L^1} \lc (2^l k)^{-\frac{d_1+1}{2}}\,.
\Ee
and \eqref{TlakL1} follows if we sum in $l$.
\end{proof}

\subsection{\it An estimate away from the singular support} \label{awayestim}
For later use in the proof of Theorem \ref{multipliers} we need the following
observation.

\begin{prop}\label{junkaway}  Let $\la \ge 1$, $K_\la$ be the convolution kernel for the
operator $\chi(\la^{-1}\sqrt L) e^{i\sqrt L},$  where $\chi\in\cS(\RR),$ and let $R\ge 10$.
Then, for every $N\in\NN,$
$$\int_{\max\{|x|,|u|\} \ge R} |K_\la(x,u)| \, dx\, du \le C_N (\la R)^{-N}\,.
$$
Moreover, the constants $C_N$ depend only on $N$ and a suitable Schwartz norm of $\chi.$
\end{prop}
\begin{proof}  This estimate is implicit in our arguments above, but it is easier to establish it as a direct  consequence of the finite propagation speed of solutions to the wave equation \cite{melrose}. Indeed, write
$$
\chi(\la^{-1}\sqrt L) e^{i\sqrt L}=\chi(\la^{-1}\sqrt L) \cos{\sqrt L}
+i\la\tilde\chi(\la^{-1}\sqrt L)\frac {\sin{\sqrt L}}{\sqrt L},
$$
with $\tilde\chi(s)=s\chi(s),$ and denote by $\var_\la$ and $\cP$ the convolution kernels for the operators $\chi(\la^{-1}\sqrt L)$ and $\cos{\sqrt L},$  respectively. Then $\cP$ is a compactly supported distribution (of finite order). Indeed, $\cP$ is supported in the unit ball with respect to the optimal control distance  associated to  the H\"ormander system of vector fields $X_1,\dots,X_{d_1},$ which is contained in the Euclidean ball of radius $10.$  Moreover,  by homogeneity, $\var_\la(x,u)=\la^{d_1+2d_2}\var(\la x,\la^2 u),$ with a fixed Schwartz function $\var.$  Note also that by  Hulanicki's theorem \cite{hulanicki}, the mapping taking $\chi$ to $\var$ is continuos in the Schwartz topologies.  Since the convolution kernel $K_\la^c$ for the operator $\chi(\la^{-1}\sqrt L) \cos{\sqrt L}$ is given by $\var_\la *\cP,$  it  is then easily seen $K_\la^c(x,u)$ can be estimated by $C_N\la^M(\la|x|+\la^2|u|)^{-N}$ for every $N\in\NN,$ with a fixed constant $M.$ A very similar argument applies to $\tilde\chi(\la^{-1}\sqrt L)\frac {\sin{\sqrt L}}{\sqrt L},$ and thus we obtain the above integral estimate for $K_\la.$
\end{proof}

\section{$h_\iso^1\to L^1$ estimates for the operators $\cW_n$}\label{hardyspaceestimates}
In this section we consider the operatots
$\cW_n=\sum_j W_{j,n}$ and prove the relevant estimate 
in Theorem \ref{refinedh1thm}.
In the proof we shall use a simple $L^2$ bound which follows from the spectral theorem,
namely for $j_0>0$
\Be \label{WjnL2lemma}
\Big\| \sum_{j\ge j_0} W_{j,n}\Big\|_{L^2\to L^2} \lc
2^{-j_0(d-1)/2}\,.
\Ee

\subsection*{\it Preliminary considerations}
Let $\rho \le  1$ and let $f_\rho$ be an $L^2$-function satisfying
\Be \label{atomprop}
\begin{gathered}\|f_\rho\|_2 \le\rho^{-d/2}, \\
\supp(f_\rho)\subset Q_{\rho,E}:=\{(x,u): \max \{|x|,|u|\} \le \rho\}\,,
\end{gathered}
\Ee
and we also assume that
\Be \label{cancel} \iint f_\rho(x,u) dx\,du=0\,, \text{ if } \rho\le 1/2.
\Ee
In what follows we also need notation for the expanded
Euclidean "ball"
\Be\label{expandedball}
Q_{\rho,E, *}=\{(x,u): \max \{|x|,|u|\} \le C_*\rho\}\,,
\Ee where   $C_*= 10(1+d_2\max_i\|J_i\|)$.

We begin with the situation given by  \eqref{cancel}. By translation invariance and the definition of $h^1_\iso$ it will
suffice to
check   that
\Be\label{atomest}
\|\cW_n f_\rho\|_{L^1} \lc (1+n)2^{-n(d_1-1)/2}\,.
\Ee


We work with dyadic spectral decompositions for the operators $|U|$ and $\sqrt{L}$ and need to discuss how they act on the atom $f_\rho$.

For $j> 0$, $n\ge 0$, let $H_{j,n}$ be the convolution kernel
defined by
$$\chi_1(2^{-2j}L)\zeta_1(2^{-j-n}|U|)f= f*H_{j,n}.$$
From \eqref{jtspec} we see that
$$H_{j,n}=0 \text{ when } n> j+11\,.$$

\begin{lemma} \label{Hjnatomlem}
 Let $\rho\le 1$, and  $f_\rho$ be as in \eqref{atomprop}.
 Then

(i)  $\|f_\rho* H_{j,n}\|_1\lc 1$ and
\Be \label{Hjnatomaway}
\|f_\rho* H_{j,n}\|_{L^1(Q_{\rho,E,*}^{\complement})} \, \lc_N(2^{j}\rho)^{-N}.
\Ee

(ii) If $f_\rho$ satisfies  \eqref{cancel} then
\Be \label{Hjnatom}
\|f_\rho* H_{j,n}\|_1 \lc \min \{ 1,\, 2^{j+n}\rho\}\,.
\Ee
\end{lemma}

\begin{proof}
By Hulanicki's theorem \cite{hulanicki} the convolution kernel of $\chi_1(L)$ is a Schwartz function $g_1$ on
$\bbR^{d_1+d_2}$. The  convolution kernel of $\zeta_1(|U|)$ is $\delta\otimes g_2$
where $\delta$ is the Dirac measure in $\bbR^{d_1}$ and $g_2$
is   a Schwartz function on $\bbR^{d_2}$. Then
\Be\label{hjn} H_{j,n}(x,u) =
\int 2^{j(d_1+2d_2)}
g_1 (2^j x, 2^{2j}w) 2^{(j+n)d_2}g_2(2^{j+n}(u-w)) \, dw\Ee
Clearly $\|H_{j,n}\|_1 =O(1)$ uniformly in $j,n$ and since
$\|f_\rho\|_1\lc 1$ we get from Minkowski's inequality
$\|f_\rho*H_{j,n}\|_1\lc 1$.

For the proof of
\eqref{Hjnatomaway} we may thus assume
 $2^j\ge 1/\rho$ and it  suffices to verify that for every
 $(y,v)\in Q_{\rho,E}$ the $L^1(Q_{A\rho,E}^\complement)$ norm of $(x,u)\mapsto$
$$
\int \frac{2^{j(d_1+2d_2)}}{(1+2^j|x-y|+2^{2j}|w|)^{N_1}}
\frac{2^{(j+n) d_2}} {(1+2^{j+n}|u-v-w+\frac 12\inn{\vec Jx}{y}|)^{N_1}}
dw
$$
is bounded by $C(2^j\rho)^{-N}$ if $N_1\gg N+ d_1+2d_2$. This is  straightforward.
For the proof of \eqref{Hjnatom} we observe that \eqref{hjn} implies
$$
2^{-j}\| \nabla_x  H_{j,n}\|_1 +2^{-j-n}\|\nabla_u H_{j,n}\|_1 = O(1)\,.
$$
Moreover $2^{-n}\||x|\nabla_u H_{j,n}\|_1 = O(1)$.
By  the cancellation condition \eqref{cancel}
\begin{align*}
&f* H_{j,n}(x,u) \\
&\quad= \int f_\rho(y,v)\big[ H_{j,n}(x-y, u-v+ \tfrac 12 \inn{\vec Jx}{y})
-H_{j,n}(x, u)\big] \, dy\, dv
\\
&\quad=
-\int f_\rho (y, v) \Big(\int_0^1
\biginn{y}{\nabla_x H_{j,n}(x-sy, u-sv+\tfrac s2 \inn{\vec Jx}{y})}
\\
&\qquad\qquad+ \biginn {v+\tfrac 12 \inn{\vec Jx}{y}}{\nabla_u H(x-sy, u-sv+\tfrac s2 \inn{\vec Jx}{y})}\,ds\Big) dy\, dv\,.
\end{align*}
We also use $\inn{\vec Jx}{y})=\inn{\vec J(x-sy)}{y})$ and a change of
variable to estimate
$$\|f_\rho*H_{j,n}\|\lc  \|f_\rho\|_1\,\rho\,
\big[ \|\nabla_x H_{j,n}\|_1
+ \|\nabla_u H_{j,n}\|_1+ \||x|\nabla_u H_{j,n}\|_1\big]\,,
$$
and \eqref{Hjnatom} follows.
\end{proof}

\begin{proof}[Proof of \eqref{atomest}] For $n>0$ split
$$
\begin{aligned}
\cW_n f_\rho&=
\sum_{\substack {j\ge n-11\\2^j \rho< 2^{-10  n}}} W_{j,n}f_\rho
+
\sum_{\substack {j\ge n-11\\2^{-10 n}\le    2^j \rho\le 2^{10 n}}}
W_{j,n}f_\rho
+
\sum_{2^{10 n}<    2^j \rho}
W_{j,n}f_\rho
\\&:=
I_{n,\rho}+II_{n,\rho}+III_{n,\rho}\,.
\end{aligned}
$$
The main contribution comes from the middle term and by \eqref{Wjnest} and the estimate $\|f_\rho\|_1\lc 1$ we immediately get
\Be\label{twonrho}\|II_{n,\rho}\|_1 \lc (1+n)2^{-n(d_1-1)/2} \,.\Ee
Let $\cJ_n$ be as in \eqref{jndef}, so that $\#(\cJ_n)=O(2^n)$.
We use  the estimate
\eqref{Tlaklbounds}
in conjunction with
\eqref{Hjnatom} and estimate
\begin{align*}
\|I_{n,\rho}\|_1&\le \sum_{2^j \rho< 2^{-10 n}} \sum_{k\in\cJ_n}
\sum_{l=1}^\infty
\|2^{-j(d-1)/2}T_{2^j}^{k,l} (f_\rho*H_{j,n})
 \|_1
\\&\lc\sum_{2^j \rho< 2^{-10  n}} \sum_{k\in\cJ_n}\sum_{l=1}^\infty
 (2^l k)^{-\frac{d_1+1}2} 2^{n+j}\rho
\lc 2^{-n(9+ \frac{d_1-1}{2})}.
\end{align*}

We turn to the estimation of the term $III_{n,\rho}$.
Let $\fT_{\rho,n}$ be a maximal $\sqrt{\eps  \rho}$ separated set of
$[2^{n-6}, 2^{n+6}]$.
For each $\beta \in \fT_{\rho,n}$ let, for large $C_1\gg 1$,
\Be\label{nghood}
\cN_{n,\rho}(\beta)=
\{(x,u): \big||x|-r(\beta)\big|\le \sqrt{C_1\rho},\quad|w(\beta, x, 4|u|)|\le
C_1\rho\}
\Ee
and
$$\cN_{n,\rho}=\bigcup_{\beta\in \fT_{\rho,n}} \cN_{n,\rho}(\beta)\,.$$
Observe that  $\meas(\cN_{n,\rho}(\beta))\lc_{C_1}
2^{-n(d_1+d_2-2)}\rho^{3/2}$  (by \eqref{rv(t)} and \eqref{slope})  and thus
$\meas(\cN_{n,\rho})\lc_{C_1}\rho$.
We separately estimate the quantity $III_{n,\rho}$ on $\cN_{n,\rho}$ and
its complement.
First, by the Cauchy-Schwarz inequality and  \eqref{WjnL2lemma}
(with $2^{j_0}\approx 2^{10 n}\rho^{-1}$)
$$\| III_{n,\rho}\|_{L^1(\cN_{n,\rho})} \lc \rho^{1/2}
  \|III_{n,\rho}\|_2
\lc (2^{-10 n}\rho)^{\frac{d-1}{2}}\rho^{1/2} \|f_\rho\|_2
$$
and, since $\rho^{d/2}\|f_\rho\|_2\lc 1$,
\Be\label{estonexcset}
\| III_{n,\rho}\|_{L^1(\cN_{n,\rho})} \lc 2^{-5(d-1)n}\,.
\Ee

In the complement of the exceptional set $\cN_{n,\rho}$ we split the
term
 $III_{n,\rho}$ as
$$III_{n,\rho}= \sum_{2^j \rho> 2^{10  n}}\sum_{k\in\cJ_n}\sum_{l=1}^\infty
(III_{n,\rho,j}^{k,l}+IV_{n,\rho,j}^{k,l})
$$
where
$$\begin{aligned}
& III_{n,\rho,j}^{k,l}=
2^{-j\frac{d-1}{2}}  T^{k,l}_{2^j} [
  (f_{\rho}*H_{j,n})\chi_{Q_{\rho,E,*}}]
\\
& IV_{n,\rho,j}^{k,l}=
2^{-j\frac{d-1}{2}}  T^{k,l}_{2^j} [
  (f_{\rho}*H_{j,n})\chi_{Q_{\rho,E,*}^\complement}]
\end{aligned}
$$ and $Q_{\rho,E,*} $ is as in \eqref{expandedball}.
From   \eqref{Hjnatomaway}  and \eqref{Tlaklbounds}
we immediately get
$\| IV_{n,\rho,j}^{k,l}\|_1\lc_N (2^lk)^{-(d_1+1)/2}
(2^j \rho)^{-N} $ and thus
$$ \sum_{2^j \rho> 2^{10  n}}\sum_{l=1}^\infty
\sum_{k\in\cJ_n} \| IV_{n,\rho,j}^{k,l}\|_1
\lc 2^{-10 nN}\,.
$$
It remains to  show that
\Be\label{offexcset}
\sum_{l=1}^\infty\sum_{k\in\cJ_n}  \sum_{2^j\rho>2^{10 n}}
\| III_{n,\rho,j}^{k,l}\|_{L^1(\cN_{n,\rho}^\complement)} \lc 2^{-n(d_1-1)/2}\,.
\Ee
Let $F_{j,n,\rho}=(f_\rho*H_{j,n})\chi_{Q_{\rho,E,*}}$, so that
$\|F_{j,n,\rho}\|_1\,\lc 1$.
We shall show that for $k\approx 2^n$
\Be\label{offexcsetmain}
\|F_{j,n,\rho}*A^{k,l}_{2^j}\|_{L^1(\cN_{n,\rho}^\complement)}
\lc_N (2^{j-n}\rho)^{-N} 2^{-(l+n)
\frac{d_1}{2}},
\quad 2^l\le 10^{8} 2^{j-n}\,,
\Ee
and \eqref{offexcset}  follows by combining
\eqref{offexcsetmain} with the estimates \eqref{KklminusAkl} and
\eqref{Kkllarge}.

\medskip

{\it Proof of   \eqref{offexcsetmain}.}
We split $A_{2^j}^{k,l}=\sum_{b\in \cT_{2^j,k,l}}
A_{2^j,b}^{k,l}$
as in \eqref{bdecomp}.
For each $b\in  \cT_{2^j,k,l} $ we may assign  a $\beta(b)\in
\fT_{\rho,n}$
so that $|\beta(b)-b|\le \sqrt{\eps\rho}$.
Let
$ \cT_{2^j,k,l}^\beta$
be the set of $b\in\cT_{2^j,k,l}$ with
$\beta(b)=\beta$.
Then $\# \cT_{2^j,k,l}^\beta\lc 2^{-n/2}\sqrt{ 2^{l+j}\rho}$.
In order to see \eqref{offexcsetmain} it thus suffices to show that
for  $2^l\le 10^8 2^{j-n} $, $|\beta-b|\le\rho$,
\Be \label{offexcsetbeta}
\|F_{j,n,\rho}*A^{k,l}_{2^j,b}\|_{L^1((\cN_{n,\rho}(\beta))^\complement)}
\lcs{N_1} (2^{j-n}\rho)^{-N_1} 2^{-(l+n)\frac{d_1+1}{2}} 2^{(n+l-j)/2}.
\Ee
To prove this we verify the following
\Be\label{geomclaim}
\text{{\it Claim: }}
\begin{aligned}
&\text{If } (\widetilde x,\widetilde u)\in
    Q_{\rho,E, *}, \,
(x,u)\in (\cN_{n,\rho}(\beta))^\complement
\text{ and }  2^{2m-j+n}\le \rho
\\& \text{then }
\big(|x-\widetilde x|, 4|u-\widetilde u+\tfrac12 \inn{\vec Jx}{\widetilde
  x}|\big)\notin
\cP_m(2^j, l,k; b) \,;
\end{aligned}
\Ee
$\cP_m(2^j,l,k;b)$ was defined in \eqref{cPm}.
Indeed the claim implies
$$
\big\|F_{j,n,\rho}*A^{k,l}_{2^j,b}\big\|_{L^1((\cN_{n,\rho}(\beta))^\complement)}
\lc
\int_{\substack{(|x|, 4|u|)\notin \\\cP_m(2^j, l,k; b)}}
|A^{k,l}_{2^j,b}(x,u)|dx\,du
$$
since
$\|F_{j,n,\rho}\|_1  =O(1)$
and \eqref{offexcsetbeta} follows from Proposition
\ref{vert-hor}.

To  verify the claim \eqref{geomclaim} we pick $(x,u)
\notin \cN_{n,\rho}(\beta)$
and distinguish two cases:
\begin{align*}
&\text{\it Case 1: $||x|-r(\beta)|\ge \sqrt{C_1\rho}$.}
\\
&\text{\it Case 2: $|w(\beta, |x|, 4|u|)| \ge C_1\rho $ and
$||x|-r(\beta)|\le \sqrt{C_1\rho}$.}
\end{align*}
It is clear that the conclusion of the claim holds if we can  show that
under the assumption that  $C_1$ in the definition \eqref{nghood} is chosen sufficiently large
(depending only on $\vec J$ and  the dimension $d$)
we have for all $(\widetilde x,\widetilde u)\in Q_{\rho,E,*}$
\begin{align} \label{case1conclusion}
&\big||x-\widetilde x|-r(b)\big| \,\ge \, \frac{\sqrt  {C_1\rho}}{2}
  &\text{ in Case 1},
\\
\label{case2conclusion}
&\big|w(b, |x-\widetilde x|, 4|u-\widetilde u+ \tfrac 12
\inn{\vec   Jx}{\widetilde x}|)\big| \ge \frac{C_1\rho}{2}
&\text{in Case 2}.
\end{align}

The Case 1 assumption implies   for $(\widetilde x, \widetilde u)\in Q_{\rho,E,*}$ (and sufficiently large $C_1$)
\begin{align*}
&\big||x-\widetilde x|-r(b)\big| \,\ge \,
||x|-r(\beta)| -|\widetilde x|- |r(b)-r(\beta)|
\\
&\ge C_1 \rho^{1/2} - C_*\rho - C |b-\beta|2^{-n} \ge
\frac{\sqrt  {C_1\rho}}{2}
\end{align*}
which is \eqref{case1conclusion}.

Now assume that $(x,u)$ satisfies the Case 2 assumption.
We then  have for all
$(\widetilde x, \widetilde u)\in Q_{\rho,E,*}$
\begin{align*}
&\big|w(b, |x-\widetilde x|, 4|u-\widetilde u+ \tfrac 12
\inn{\vec   J x}{\widetilde x}|)
- w(\beta,|x|, 4|u|)
\big|
\\
\,&\le\,\big|w(b, |x|, 4|u|)-w(\beta,|x|, 4|u|)\big|
\\&\quad+ \,\big |w(b, |x-\widetilde x|, 4|u-\widetilde u+\tfrac 12
\inn{\vec J  x}{\widetilde x}) - w(b, |x|, 4|u|)\big|
\end{align*}
The first term on the right hand side can be estimated using
\eqref{wdiff2}
(with $(t,b)$ replaced by $(b,\beta)$), and we see that it is
$\le (C+\sqrt C_1)\rho$ under the present Case 2 assumption. The
second term on the right hand side is equal to
$$
\Big|4|u|-4|u-\widetilde u+\tfrac 12 \inn{\vec Jx}{ \widetilde x}|
-\frac{v'(b)}{r'(b)} (|x|-|x-\widetilde x|)\Big|
$$
and since the Case 2 assumption implies $|x|=O(1)$ we see that the
displayed expression is $O(\rho)$. Thus, if $C_1$ in the definition
is sufficiently large
we obtain \eqref{case2conclusion}. This concludes the proof of the
claim
\eqref{geomclaim} and thus the estimate  \eqref{offexcsetmain}.

\medskip

We finally consider the case where $1/2<\rho\le 1,$ in which condition \eqref{cancel} is not required. This case can easily be handled by means of Proposition \ref{junkaway}. To this end, we decompose
$$
a(\sqrt{L})e^{i\sqrt{L}}=\sum_{j\ge 10}2^{-\frac{d-1} 2 j}g_j(2^{-j}\sqrt{L})e^{i\sqrt{L}},
$$
with $g_j(s)=2^{\frac{d-1} 2 j} a(2^{j}s)\chi_1(s).$ The family of functions $g_j$ is uniformly bounded in the Schwartz space. If $K_j$ denotes the convolution kernel for the operator  $g_j(2^{-j}\sqrt{L})e^{i\sqrt{L}},$  we thus obtain from Proposition \ref{junkaway} the uniform estimates
$$\int_{\max\{|x|,|u|\} \ge 100} |K_j(x,u)| \, dx\, du \le C_N  2^{-jN}\,.
$$
This implies that $$\int_{\max\{|x|,|u|\} \ge 200}|(a(\sqrt{L})e^{i\sqrt{L}}f_\rho)(x)|dxdu\lesssim \|f\|_1\lesssim 1.$$
And, by H\"older's inequality,  $$\int_{\max\{|x|,|u|\} \le 200}|(a(\sqrt{L})e^{i\sqrt{L}}f_\rho)(x)|dx\lesssim \|(a(\sqrt{L})e^{i\sqrt{L}}f_\rho)\|_2\lesssim\|f_\rho\|_2\lesssim 1.$$
This concludes the proof of Theorem \ref{refinedh1thm}.
\end{proof}

\section{Interpolation of Hardy spaces}\label{interpolation}

By interpolation for analytic families
Theorem \ref{main-theorem} can be deduced from the Hardy space estimate
if we show that  $L^p(G)$ is an interpolation space  for the couple
$(h^1_\iso(G), L^2(G))$,  with respect to
Calder\'on's
complex $[\cdot,\cdot]_\vth$ method.

\begin{theorem} \label{interpolthm} For $1<p\le 2$,
\Be\label{interpol} [h^1_\iso(G), L^2(G)]_\vth = L^p(G), \quad
\vth =2-2/p,
\Ee with equivalence of norms.
\end{theorem}


\begin{proof}
We deduce \eqref{interpol} from an analogous formula for the Euclidean
local Hardy spaces $h^1_E,$  more precisely, the vector-valued extension
\Be\label{interpolvect}
 [\ell^1(h^1_E), \ell^2(L^2)]_\vth = \ell^p(L^p), \quad
\vth =2-2/p.\Ee
Here $\ell^p\equiv \ell^p(\bbZ^{d_1+d_2})$.
To do this one uses the method of retractions and coretractions
(\cf. \cite{triebelinterpol}); \eqref{interpolvect} follows from the
definition of the complex interpolation method if operators
\begin{align*}
R: \ & h^1_\iso+L^2
\to  \ell^1(h^1_E) +\ell^2(L^2)
\\
S: \ &\ell^1(h^1_E) +\ell^2(L^2)  \to h^1_\iso+L^2
\end{align*}
can be constructed such that
\[R: \begin{cases}   h^1_\iso
\to  \ell^1(h^1_E)
\\ L^2 \to \ell^2(L^2)
\end{cases}
\qquad S:\begin{cases} \ \ell^1(h^1_E)
\to h^1_\iso
\\
 \ell^2(L^2) \to L^2  \end{cases}
\]
are bounded
and
$$SR= I, $$
the identity operator on $L^p$ or $h^1_\iso$.

To define $R$ and $S$  let $\vphi_1\in C^\infty_0(\bbR^{d_1})$,
$\vphi_2\in C^\infty_0(\bbR^{d_2})$  supported in
in $(-1,1)^{d_1}$  and $(-1,1)^{d_2}$  respectively
and such
that
for all  $ x\in \bbR^{d_1}$, $u\in \bbR^{d_2}$
\Be \label{unity}\sum_{X\in \bbZ^{d_1}}
\vphi^2_1(x+X) =1, \quad
\sum_{U\in \bbZ^{d_2}}
\vphi^2_2(u+U) =1.
\Ee
We define $\varphi(x,u)=\varphi_1(x)\varphi_2(u)$ and set
\begin{multline*} Rf= \{R_{X,U} f\}_{(X,U)\in \bbZ^{d_1}\times\bbZ^{d_2}} \\
  \text{ where  }
R_{X,U} f(x,u)= \vphi(-x,-u) f(x+X, u+U+ \tfrac 12\inn{\vec J X}{x});
\end{multline*}
moreover for $H=\{H_{X,U}\}_{(X,U)\in \bbZ^{d_1}\times\bbZ^{d_2}}\in \ell^1(h^1_E) $
we set
\begin{multline*} SH(x,u) =\\
\sum_{(X,U)\in
    \bbZ^{d_1}\times\bbZ^{d_2}} \vphi(X-x, U-u-\tfrac 12\inn{\vec J X}{x})
H_{X,U} (x-X, u-U- \tfrac 12\inn{\vec J X}{x})
\end{multline*}
One  verifies quickly from \eqref{unity} that $SR$ is the identity.

We now   examine the boundedness properties of $R$ and $S$.
For the $  h^1_\iso
\to  \ell^1(h^1_E) $
of $R$ we consider a (Heisenberg-)$(P,\rho)$ atom $a$ with $P=(x_P,u_P)$ and
$\rho\le 1$.
Note that
$\vphi(-x,-u) a(x+X, u+U+ \inn{\vec J X}{x})$ is then supported on the
set of $(x,u)\in (-1,1)^{d_1+d_2}$ such that
$$|x_P-X-x|^2 +
|u_P-U-u -\inn{\vec J(X-x_P)}{x}|^2\le \rho^2.$$
Thus $R_{X,U} f$ is not identically zero only when $|X-x_P|+|U-u_P|\le
C_d$ some absolute constant $C_d$. And, since
$\inn{\vec J(X-x_P)}{X-x_P}=0$ we also see that in this case
the function \Be\label{atomchanged}(x,u)\mapsto a(x+X, u+U+ \tfrac 12\inn{\vec
  J X}{x})
\Ee is supported in a
Euclidean ball of radius $C \rho$ with center $(x_P-X, u_P-U)$. Since the
cancellation
property (if $\rho\le 1/2$) is not affected by the change of variable we see
that the function \eqref{atomchanged}  is equal to $c_b b$ where
$b$ is a Euclidean atom
and $|c_b|\lc 1$. Thus this function is in
$h^1_E$ with norm $\lc 1$. We also use that multiplication with
$\vphi(-x,-u)$ defines an operator which is bounded on the local
Hardy-space $h^1_E$.
Now it follows quickly that $R$ is
bounded as an operator from
$  h^1_\iso$ to $  \ell^1(h^1_E) $.
Indeed if $f=\sum_{c_\nu} a_{\nu}$ where $a_{\nu} $ are
$(P_\nu,r_\nu)$ atoms for suitable $r_\nu\le 1$ and $P_\nu$ then
\begin{align*}
&\|R f\|_{\ell^1(h^1_E)}
= \sum_{X,U}\big\|R_{X,U}\sum_\nu c_\nu a_{P,\nu}\big\|_{h^1_E}
\\ &\le C
\sum_{X,U} \sum_{\substack{\nu: |x_{P_\nu} -X|\le C_d\\ |u_{P_\nu}-U|\le C_d}}
|c_\nu| \le C' \sum_\nu |c_\nu|.
\end{align*}
This completes the proof of the
$h^1_\iso
\to  \ell^1(h^1_E)$ boundedness of $R$.

We now show that $S$ maps
$\ell^1(h^1_E)$ boundedly to $h^1_\iso$.
We first recall that the operation of multiplication with a smooth bump function maps $h^1_E$ to itself (\cf. \cite{goldberg}), thus $$\|\vphi(-\cdot) G_{X,U} \|_{h^1_E}\le C
\|G_{X,U} \|_{h^1_E}.$$
Using the atomic decomposition of $h^1_E$ functions we can
decompose
$$\vphi(-\cdot) G_{X,U} = \sum_{\nu} c_{X,U,\nu} a_{X,U,\nu}$$
where $\sum_{X,U,\nu} |c_{X,U,\nu}| \lc \|G\|_{\ell^1(h^1_E)}$ and the
$a_{X,U,\nu}$ are Euclidean atoms
supported in a ball
$$\{(x,u): |x-x_P|^2+ |u-u_P|^2 \le r^2\}\subset[-3,3]^{d_1+d_2};$$
with $P=P(X,U,\nu)$ and $r=r(X,U,\nu).$
Fix such an atom $a=a_{X,U,\nu}$. The function
\Be \label{movedatom}
\widetilde a_{X,U,\nu}: (x,u)\mapsto a_{X,U,\nu}(x-X, u-U-\tfrac 12\inn{\vec JX}{x})
\Ee
is supported in
\[\{(x,u): (|x-X-x_P|^2+ |u-U -\tfrac 12\inn{\vec JX}{x}
-u_P|^2)^{1/2} \le r\}\]
which is contained in the set of $(x,u)$ such that
$$\big(|x-(X+x_P)|^2+ |u-U - \tfrac 12\inn{\vec J(X+x_P)}{x} -u_P+
\tfrac 12 \inn{\vec Jx_P}{X+x_P}
|^2\big)^{\frac 12} $$
is $\le  (1+\tfrac 32\sqrt{ d_1}) r$.
Here we have used  that $|\inn{\vec J x_P}{ x-(X+x_P)}| \le |x_P| r $ and
$|x_P|\le 3\sqrt{d_1}$. The inclusion shows
that there is an  constant independent of $X,U,\nu$ so that
 function $\widetilde a_{X,U}/C$ is a Heisenberg atom associated  with
 a cube centered at $(X+x_P, U+u_P+\tfrac 12 \inn{\vec Jx_P}{X+x_P}$.
This statement holds at least if $r\le 1/(4d_1)$. If $r$ is close to
one then we can express $\widetilde a_{X,U}$ as a finite sum of
$6^d$ atoms supported in cubes of sidelength $1$.
 Thus we see that the function in \eqref{movedatom} has $h^1_{\iso}$
 norm $\lc 1$. This implies the
$ \ell^1(h^1_E)\to h^1_\iso$  boundedness of $S,$ since
it follows that
$$\|SG\|_{h^1_{\iso}}\lc \sum_{(X,U)}
\sum_{\nu} |c_{X,U,\nu}|
\|\widetilde a_{X,U,\nu}\|_{h^1_{\iso}}\lc \sum_{X,U,\nu}
|c_{X,U,\nu}|
\lc \|G\|_{\ell^1(h^1_E)}.
$$
Finally the
$ L^2
\to  \ell^2(L^2)$ boundedness of $R$ and the
 $ \ell^2(L^2)
\to L^2$ boundedness of $S$ are even more straightforward and
 follow by modifications of the arguments.
\end{proof}

\begin{proof}[Proof of Theorem \ref{main-theorem}]
By duality we may assume $1<p<2$. By scaling and symmetry
we may assume $\tau=1$. Let $a\in S^{-(d-1)(1/p-1/2)}$.
Consider the analytic family of operators
$$\cA_z= e^{z^2}\sum_{j=0}^\infty 2^{-jz\frac{d-1}{2}}
 2^{j(d-1)(\frac 1p-\frac 12)}
\zeta_j(\sqrt L) a(\sqrt L)e^{i\sqrt{L}}.$$
We need to check that $\cA_z$ is bounded on $L^p$ for $z=(2/p-1)$.
But  for $\Re(z)=0$ the operators $\cA_z$ are  bounded on $L^2$;
and
for $\Re(z)=1$ we have shown that $\cA_z$ maps   $h^1_\iso$ boundedly to $L^1,$ by
Theorem
\ref{h1thm}.
We apply the abstract version of the
interpolation theorem for
analytic families  in conjunction with
Theorem \ref{interpolthm} and the corresponding standard version
interpolation result for  $L^p$ spaces; the result is that
$\cA_\theta$ is bounded on $L^p$ for $\theta= 2/p-1$. This proves Theorem
\ref{main-theorem}.
\end{proof}

\section{Proof of Theorem \ref{multipliers}}
We decompose $m=\sum_{k\in \bbZ} m_k$ where $m_k$ is supported in
$(2^{k-1}, 2^{k+1})$ and where $h_k= m_k(2^k\cdot)$ satisfies
$$\sum_{\ell>1} \sup_k\int_{2^\ell}^{\infty}|\widehat h_k(\tau)|
\tau^{\frac{d-1}2}\, d\tau \,\le\, A \,.
$$
By the translation invariance and the  usual
Calder\'on-Zygmund arguments  (see, e.g., \cite{steinharmonic})
it
suffices to  prove that for all
$\rho>0$ and for all $L^1$ functions $f_\rho$ supported in the
Koranyi-ball $Q_\rho:=Q_\rho(0,0)$ and satisfying
$\int f_\rho \,dx=0$ we have
\Be\label{CZL1}
\sum_k \iint_{Q_{10\rho}^\complement}
|m_k(\sqrt{L})f_\rho| \,dx \lc A+\|m\|_\infty
\Ee
Let $\chi_1\in C^\infty_0$ be  supported in $(1/5,5)$ so that $\chi_1(s)=1$ for $s\in [1/4,4]$. Then for each $k$ write
$$m_k(\sqrt L)= h_k(2^{-k}\sqrt L)\chi_1(2^{-k}\sqrt L)
\,=\, \int \widehat {h_{k}}(\tau)
\chi_1(2^{-k}\sqrt L) e^{i 2^{-k} \tau\sqrt{L}}
d\tau
\,.
$$
By scale invariance and Theorem \ref{main-theoremL1},
the $L^1$ operator norm of the operator
$\chi_1(2^{-k}\sqrt L) e^{i 2^{-k} \tau\sqrt{L}}$ is
$O(1 +|\tau|)^{(d-1)/2}$ and thus
\Be\notag
\|m_k(\sqrt L)\|_{L^1\to L^1}
\lc
\int_{-\infty}^\infty
 |\widehat {h_{k}}(\tau)|  (1+|\tau|)^{\frac{d-1}{2}}
d\tau \,.
\Ee
Also observe that since the convolution kernel of $\chi_1(\sqrt L) $
is a Schwartz kernel we can use the cancellation and
support properties of
$f_\rho$ to get, with some $\eps>0$,
$$\|\chi_1(2^{-k}\sqrt L) f_\rho\|_1 \lc \min\{1, (2^k\rho)^{\eps}\}
 \|f_\rho\|_1\,.
$$
Thus the two preceding displayed inequalities yield
\begin{align}\notag
\sum_{k: 2^k\rho  \le M}
\|m_k(\sqrt L)f_\rho\|_{1}
&\le C_M
\sup_k \int_{-\infty}^\infty
 |\widehat {h_{k}}(\tau)|  (1+|\tau|)^{\frac{d-1}{2}}
d\tau  \, \|f_\rho\|_1
\\&\lcs{M} (\|m\|_\infty + \fA_2)   \, \|f_\rho\|_1\,
\label{firstgeomsum}
\end{align}
where for the last estimate we use
 $|\widehat {h_{k}}(\tau)|\le \|h_k\|_\infty \lc\|m\|_\infty$ when $|\tau|\le 2$.

We now consider the terms for $2^k\rho\ge M$ and $M$ large,
in the complement of the expanded Koranyi-ball $Q_{\rho,*}= Q_{C\rho}$
(for suitable large $C\gg 2$).
By a change of variable and an application of
 Proposition \ref{junkaway},
\begin{align*}
&\big\|e^{i2^{-k}\tau \sqrt L} \chi_1(2^{-k}\sqrt L)
 f_\rho \big\|_{L^1(Q_{\rho,*}^\complement)}=
\big\|e^{i\sqrt L} \chi_1(\tau^{-1}\sqrt L)
 f_\rho^{2^k/\tau} \big\|_{L^1(Q_{C_*\tau^{-1}2^k \rho}^\complement)}
\\ &\lc (2^k\rho \tau^{-1})^{-N}\quad  \text{  if $2^k\rho \gg \tau$}\,,
\end{align*}
where $f_\rho^{2^k/\tau}$ is a re-scaling of $f_\rho$ such that $\|f_\rho^{2^k/\tau}\|_1=\|f_\rho\|_1\lesssim 1.$

Hence if $M$ is sufficiently large  then for $2^k\rho>M$
\begin{multline*}
\|m_k(\sqrt L)f_\rho\|_{L^1(Q_{\rho,*}^\complement)}\,\lc_N\,\|f_\rho\|_1
\Big[\int_{|\tau|>2^k\rho}
|\widehat {h_k}(\tau)|(1+|\tau|)^{\frac{d-1}{2}} d\tau
\\+(2^k\rho)^{-N}
\int_{|\tau|\le 2^k\rho}|\widehat {h_k}(\tau)|(1+|\tau|)^{-N} d\tau\Big]
\end{multline*} and thus
\Be\label{largekcontribution}
\sum_{2^k\rho>M} \|m_k(\sqrt L)f_\rho\|_{L^1(Q_{\rho,*}^\complement)}
\lc \|m\|_\infty + \sum_{k: 2^k\rho>M}  \fA_{2^k\rho}\,.
\Ee
The theorem follows from  \eqref{firstgeomsum} and
\eqref{largekcontribution}.
\qed


\begin{thebibliography}{10}


\bibitem{beals}
R.M. Beals, \emph{$L\sp{p}$ boundedness of Fourier integral operators.}
  Mem. Amer. Math. Soc.  38  (1982), no. 264.

\bibitem{christ} M. Christ, \emph{$L\sp p$ bounds for spectral multipliers on nilpotent groups.}  Trans. Amer. Math. Soc.  328  (1991), no. 1, 73--81.



\bibitem{damek-ricci} E. Damek, F. Ricci, \emph{Harmonic analysis on solvable extensions of H-type groups.}   J. Geom. Anal.  2  (1992), 213--248.



\bibitem{folland} G.B.~Folland, Harmonic Analysis in Phase Space,
Annals of Math. Study 122, Princeton Univ. Press,  1989.

\bibitem{FollSt} G. Folland, E. M. Stein,
 Hardy spaces on homogeneous groups, Princeton Univ. Press,
Princeton University, 1982.

\bibitem{gaveau} B. Gaveau, {\it  Principe de moindre action,
propagation de la chaleur et estim\'es sous elliptiques sur certains
groupes nilpotents.}  Acta Math.  139  (1977), no. 1-2, 95--153.

\bibitem{ghk} P. Greiner, D.  Holcman, Y. Kannai,  \emph{Wave kernels related to second-order operators.} Duke Math. J. 114 (2002), no. 2, 329--386.


\bibitem{goldberg} D.  Goldberg, \emph{
A local version of real Hardy spaces.}
Duke Math. J. 46 (1979), no. 1, 27--42.



\bibitem{hebisch} W. Hebisch, \emph{
Multiplier theorem on generalized Heisenberg groups.}
Colloq. Math. 65 (1993), no. 2, 231--239.

\bibitem{hoermander-hypo}
L. H\"ormander,
\emph{Hypoelliptic second order differential equations.}
Acta Math. 119 (1967), 147--171.


\bibitem{hoermander-fio} \bysame,
\emph{Fourier integral operators, I.} Acta Math. 127 (1971), no. 1-2, 79--183.







\bibitem{hulanicki} A.  Hulanicki, \emph{A functional calculus for Rockland operators on nilpotent Lie groups.}
Studia Math. 78 (1984), no. 3, 253--266.


\bibitem{kaplan}
A.  Kaplan, \emph{Fundamental solutions for a class of hypoelliptic PDE generated by composition
of quadratic forms.}  Trans. Amer. Math. Soc. {258} (1980),
147--153.

\bibitem{kaplan-ricci}
A.  Kaplan,  F. Ricci, \emph{Harmonic analysis on groups of Heisenberg type.}  Lecture Notes in Math. {992}, Springer, Berlin  1983,
416--435.


\bibitem{martini} A. Martini, \emph{Analysis of joint spectral multipliers on Lie groups of polynomial
   growth.}  Ann. Inst. Fourier (Grenoble) {62}  (2012), no.4 , 1215--1263.













\bibitem{mauceri-meda} G.  Mauceri, S.  Meda,
\emph{Vector-valued multipliers on stratified groups.}
Rev. Mat. Iberoamericana  6  (1990),  no. 3-4, 141--154.

\bibitem{melrose}
R.  Melrose, \emph{Propagation for the wave group of a positive subelliptic second order differential operator.}   Hyperbolic equations and related topics, Academic Press, Boston, Mass.,  1986.



\bibitem{miyachi} A.  Miyachi,
\emph{On some estimates for the wave equation
in $L^p$ and $H^p$.}  Journ. Fac. Sci. Tokyo, Sci.IA,
{27} (1980), 331--354.



\bibitem{dissip} D.  M\"uller, \emph{Local solvability of second order
differential operators with double characteristics.
 II: Sufficient
conditions for left-invariant differential operators on the Heisenberg
group.} J. Reine Angew. Math. 607 (2007), 1--46.

\bibitem{edinburgh} \bysame, \emph{Analysis of invariant PDO's on the Heisenberg group,}  ICMS-Instructional Conference lecture, Edinburgh 1999, 	arXiv:1408.2634 [math.AP]

%

\bibitem{MR1} D.  M\"uller, F.   Ricci,
 \emph{ Analysis of second order differential operators on Heisenberg groups,  I.} Invent. Math.  101  (1990),  no. 3, 545--582.



\bibitem{MR-solv} \bysame,
\emph{Solvability  for a class of doubly characteristic  differential operators on 2-step nilpotent groups.} Annals of  Math.  {143}  (1996), 1--49.


\bibitem{MRS2}
D. M\"uller, F. Ricci and E.M. Stein,
\emph{
Marcinkiewicz multipliers and multi-parameter structure on Heisenberg(-type) groups, II.}
Math. Z. 221 (1996), no. 2, 267--291.


\bibitem{MuSt-mult}
D. M\"uller, E.M.  Stein,
\emph{On spectral multipliers for Heisenberg and related groups.}  J. Math. Pures Appl. (9)  73  (1994),  no. 4, 413--440.

\bibitem{MuSt-wave}\bysame,
\emph{$L\sp p$-estimates for the wave equation on the Heisenberg group.}
  Rev. Mat. Iberoamericana  {15}  (1999),  no. 2, 297--334.


\bibitem{nachman} A.I. Nachman, \emph{The wave equation on the Heisenberg group.} Comm. in PDE, {7} (1982), 675--714.

\bibitem{ne-st} E. Nelson, W. F. Stinespring,
\emph{Representations of elliptic operators in an enveloping algebra.}
Amer. J. Math. 81 (1959), 547-560.

\bibitem{peral} J. Peral, \emph{$L^p$ estimates for the wave
equation.} J. Funct. Anal., {36} (1980), 114--145.

\bibitem{ricci} F.  Ricci,
\emph{Harmonic analysis on generalized Heisenberg groups.}
Unpublished preprint.



\bibitem{SSS} A. Seeger, C.D. Sogge, E.M. Stein, \emph{
Regularity properties of Fourier integral operators.}  Annals of Math.,
{103} (1991), 231--251.







\bibitem{steinharmonic}
 E.M. Stein, Harmonic analysis: real-variable methods, orthogonality, and oscillatory integrals. With the assistance of Timothy S. Murphy.
Princeton Mathematical Series, 43.  Princeton University Press,
Princeton, NJ, 1993.


\bibitem{stw} E.M. Stein, G. Weiss,  Introduction to Fourier analysis
in Euclidean spaces, Princeton Univ. Press, 1971.


\bibitem{str} R.S. Strichartz,  \emph{$L\sp p$ harmonic analysis and Radon transforms on the Heisenberg group.}  J. Funct. Anal.
96  (1991),  no. 2, 350--406.

\bibitem{triebelinterpol}  H. Triebel, Interpolation theory, function spaces, differential operators. 2nd edition. Johann Ambrosius Barth Verlag,
Heidelberg, Leipzig, 1995.


\end{thebibliography}
\end{document}